\documentclass[11pt, reqno]{amsart}
\usepackage{amssymb}
\usepackage{times}
\usepackage{amsmath,amsthm}
\usepackage{amsfonts}
\usepackage{leftidx}
\usepackage{etoolbox}
\usepackage{mathrsfs}
\usepackage{enumerate}
\usepackage{abstract}
\usepackage{stmaryrd}
\usepackage{ulem}
\usepackage{graphicx}	
\usepackage{float}
\usepackage{hyperref}

\def\R{\mathbb{R}}

\def\d{|\nabla|}
\def\n{\nabla}

\def\p{\partial}

\def\be{\begin{equation}}
\def\ee{\end{equation}}
\def\vo{\vspace{1\baselineskip}}

\def\h{\frac{1}{2}}

\newtheorem{theorem}{Theorem} 

\newtheorem{lemma}{Lemma}[section]
\newtheorem{proposition}{Proposition}[section]
\theoremstyle{definition}
\newtheorem{definition}{Definition}[section]

\theoremstyle{remark}
\newtheorem{remark}{Remark}[section]
 
\numberwithin{equation}{section}
\begin{document}
 \title[3D GWW Above a Flat Bottom]{Global solution for the 3D gravity water waves system above a flat bottom}
\author{Xuecheng Wang}
\address{Mathematics Department, Princeton University, Princeton, NJ, USA 08544}

\address{Yau Mathematical Sciences Center, Tsinghua University, Beijing, China 100084}

\email{xuecheng@tsinghua.edu.cn,\, xuecheng.wang.work@gmail.com}
\thanks{}
\maketitle
\begin{abstract}
Given any suitably small, localized, and smooth initial data, in this paper, we prove global regularity for the $3D$ finite depth gravity water wave system. As a byproduct, we rule out the  small, localized traveling waves in $3D$, which do exist for the same system in $2D$. 
\end{abstract}
\setcounter{tocdepth}{1}
\tableofcontents

\section{Introduction}\label{introduction}
In this paper, we study the long time behavior of the motion of an inviscid incompressible fluid, e.g., water, inside a time dependent region $\Omega(t)$. A fascinating feature of this problem, which  is also known as the free boundary problem, is that the boundary of ``$\Omega(t)$'' will affect the motion of the  fluid and will also be affected by the motion of  the fluid. In other words, to study the motion of the fluid, we need to study the motion of the fluid and the motion of the boundary at the same time.

To be more precise of the problem setting, we assume that  there is a vacuum above the water region $\Omega(t)$ and there is no vorticity inside $\Omega(t)$. Moreover, we consider the gravity effect and neglect the surface tension effect.  The system under consideration is also known as the gravity water waves system.

Despite  recent gratifying progress devoted to improving the understanding of the long time behavior of the water waves system, which will be discussed later,  there are still many open questions. One of them is how the fixed bottom of the water region ``$\Omega(t)$'' changes the behavior of the solution  in the long run. Although  we do have evidence  that shows that the structure of bottom indeed plays an important role in the long run, the mechanism is not mathematically clear  even in the small data  regime. Here comes   evidence. For the $2D$ gravity water system,  small traveling waves don't exist in the infinite depth setting (without a bottom) but do exist in the flat bottom setting, see \cite{deng}. Here comes an open question, does the presence of the flat bottom affect the stability of zero solution? Please note that the zero solution is indeed stable under small perturbation for the infinite depth setting. See the work of Germain-Masmoudi-Shatah \cite{germain2} and Wu \cite{wu2} in $3D$; see the work of Ionescu-Pusateri\cite{IP1}, Alazard-Delort \cite{alazard}, Ifrim-Tataru \cite{tataru3} in $2D$. 

In this paper, we will answer this question definitely in $3D$ for   small initial data.  In conclusion, the global stability of the zero solution also holds for the $3D$ gravity water  waves system in the flat bottom setting. For any suitably small  initial data, the solution globally exists and scatters to a linear solution. Moreover, the nonlinear solution decays sharply  over time in a weak $L^\infty$-type space. 

We arrive in this conclusion by carefully analyzing the low-frequency part of the nonlinear solution, which is  the main difference between the infinite depth setting and the flat bottom setting. We remark that the high-frequency part of the nonlinear solution in two settings  are essentially same, see \cite{alazard1,alazard2}.

\subsection{Gravity water waves system above a flat bottom}

In this subsection, we give a more precise mathematical description of the $3D$ gravity water waves system in the flat bottom setting. 

Assume that the water region $\Omega(t)$ has a free interface $\Gamma(t)$  and a fixed flat bottom $\Sigma$.     We normalize both  \textit{ the depth} and \textit{the gravity constant} ``$g$'' to be ``$1$''. As a result,  we can describe the domain, the interface and the bottom in the Eulerian coordinates as follows, 
\[
\Omega(t):= \{(x,y): x\in \R^2, -1\leq y\leq \,h(t,x)\},\]
\[ 
\Gamma(t):= \{(x,y): x\in \R^2, y= \,h(t,x)\}, \quad \Sigma :=\{(x,y): x\in \R^2, y=-1\},
\]
	where $h(t,x)$ denotes the height of the interface at point $x$ and at time $t$. Since we will be in the small data regime, readers can imagine that  $h(t,x)$ is a small perturbation of ``$0$''.

 The evolution of the fluid is described by the Euler equation with boundary conditions as follows, 
\begin{equation}
\left\{\begin{array}{ll}
\p_t u + u\cdot \nabla u =-\nabla p -g(0,0,1) & \\
\nabla\cdot u =0,\,\, \nabla \times u =0, \,\,u(0)=u_0&\\
u\cdot \vec{\mathbf{n}}=0 & \textup{on\,\,$\Sigma$} \\
p=0 & \textup{on\,\, $\Gamma(t)$}\\
\p_t + u\cdot\nabla \textup{tangents to $\cup_{t} \Gamma(t)$} & \textup{on\, $\Gamma(t)$,}\\
\end{array}\right.
\end{equation}
As the  velocity field is irrotational, we can represent it in terms of velocity potential $\phi$. Let $\psi$ be the restriction of velocity potential on the boundary $\Gamma(t)$, i.e., $\psi(t,x):=\phi(t,x,\,h(t,x))$. From the divergence free condition and the boundary conditions, we can derive the following harmonic equation with two boundary conditions: a Neumann type condition on the bottom and a Dirichlet type condition on the interface,
\begin{equation}\label{harmoniceqn}
(\Delta_x + \p_y^2)\phi=0, \quad \frac{\p \phi}{\p \vec{\mathbf{n}}}\big|_{\Sigma}=0, \quad \phi\big|_{\Gamma(t)} = \psi.
\end{equation}

Following the work of Zakharov \cite{zakharov},   we can reduce  the motion of fluid to the  evolution of the height ``$h$" and the  velocity potential on the interface ``$\psi$" as follows,
\begin{equation}\label{waterwave}
\left\{\begin{array}{l}
\p_t h= G(h)\psi,\\
\p_t \psi = -h - \frac{1}{2} |\nabla \psi|^2 + \displaystyle{\frac{(G(h)\psi + \n h\cdot\n \psi)^2}{2(1+ |\n h|^2)}},
\end{array}\right.
\end{equation}
where $G( h)\psi= \sqrt{1+|\n h|^2}\mathcal{N}(h)\psi$  and $\mathcal{N}(h)\psi$ is the Dirichlet-Neumann operator at the interface $\Gamma(t)$.  For the gravity water waves system (\ref{waterwave}), the following conservation law holds as long as the solution exists over time, 
\begin{equation}\label{conservation}
\mathcal{H}( h(t), \psi(t)) = \int_{\R^2} \h | h(t) |^2 + \h \psi G(h(t))\psi(t) d x= \mathcal{H}( h(0), \psi(0)).
\end{equation}
 
\subsection{Previous results}
There is   extensive literature on the study of the water waves system. 
 Without being exhaustive on the progress  made so far, we only mention several results on the initial value problem here. For the results on the blow-up behavior and the ``splash singularity"  of solutions, interested readers please refer to \cite{fefferman, fefferman2, coutand2} and references therein.

On the local theory side, Nalimov \cite{nalimov} and Yosihara \cite{yosihara} considered the small initial data case,  Wu \cite{wu1, wu2} considered general initial data in Sobolev spaces, see also the subsequent works by Christodoulou-Lindblad \cite{christodoulou}, Lannes\cite{lannes}, Lindblad \cite{lindblad}, Coutand-Shkoller \cite{coutand1}, Shatah-Zeng\cite{shatah1} and Alazard-Burq-Zuily \cite{alazard1,alazard2}. If the effect of surface tension is also considered,  local existence also holds, see Beyer-Gunther \cite{beyer}, Ambrose-Masmoudi \cite{ambrose}, Coutand-Shkoller \cite{coutand1}, Shatah-Zeng \cite{shatah1} and Alazard-Burq-Zuily \cite{alazard1,alazard2}.

On the long time behavior side, we   have several results. For the gravity water waves system in the infinite depth setting. In the $3D$ case, Wu \cite{wu1} and Germain-Masmoudi-Shatah \cite{germain2} proved  global existence for small initial data. In the $2D$ case,  see the work of Wu\cite{wu2} and the work of Hunter-Ifrim-Tataru \cite{tataru3} for the almost global existence, see the work of Ionescu-Pusateri\cite{IP1},  Alazard-Delort \cite{alazard}, Ifrim-Tataru\cite{tataru3}, Wang \cite{wang1} for the global existence results. For the capillary water waves system in the infinite depth setting.  See the work of Germain-Masmoudi-Shatah \cite{germain3} for the $3D$ case. See the work of Ionescu-Pusateri \cite{IP4}  and Ifrim-Tataru \cite{tataru4} for the $2D$ case.

For the water waves system in the flat bottom setting. What we know so far about the flat bottom case can be summarized as follows: (i) on the one hand, the local existence holds (with bottom not necessarily flat) by the work of Lannes \cite{lannes} and the works of Alazard-Burq-Zuily \cite{alazard1,alazard2} and the large time existence holds  by the work of Alvarez-Samaniego and Lannes in \cite{Alvarez}; (ii) on the other hand,  there exist   traveling waves, which are   arbitrary small in $L^2$. Note that the existence of traveling waves    depends on the dimension and the ratio of the surface tension coefficient and the gravity constant. 

The existence of traveling waves makes the  global regularity problem more delicate and more complicated. Traveling waves are more likely to exist in $2D$. More precisely,  
in the $2D$ case, the traveling waves exist as long as $g\neq 0$ regardless the presence of surface tension effect , see  \cite{deng} and references therein. In the $3D$ case, the existence of the traveling waves are only known in  the strong surface tension case so far, more precisely the case when $\sigma/g > 1/3$, see \cite{deng}.

\subsection{Main result}
Before stating our main theorem, we first define the main function spaces.  Define a $L^\infty$-type space as follows, 
\[
\| f\|_{W^{\gamma, b}}:=\sum_{k\in\mathbb{Z}} (2^{\gamma k } + 2^{b k}) \| P_{k } f\|_{L^\infty},\quad \| f\|_{W^{\gamma}}:= \| f\|_{W^{\gamma,0}}, \quad  0 \leq b\leq \gamma,
\]
where `` $P_{k}$'' denotes the standard Littlewood-Paley projection operator, which will be defined precisely in the subsection \ref{notation}. 

We define the $Z$-normed space and the auxiliary space ``$B_{k,j}$ '' as follows,
\begin{equation}\label{definitionofZnorm}
\| f\|_{Z}:=  \sup_{k\in \mathbb{Z}}\sum_{j\geq \max\{-k, 0\}} \|  f\|_{B_{k,j}}, \quad \| f \|_{B_{k,j}} := 2^{\alpha k} (1+ 2^{6 k}) 2^{ j}\| \varphi_j^k(x)\cdot P_k f \|_{L^2}, \quad \alpha=1/10,
\end{equation}
where the cutoff function $\varphi_{j}^{k}(x)$ localizes the physical position with a threshold determined  by the localized frequency. The detailed formula of $\varphi_j^k(x)$ is postponed to the subsection \ref{notation}.

 The  $Z$-normed space of this type was first  introduced  by Ionescu-Pausader  in \cite{benoit} for the   Euler-Poisson system. A basic idea of using this $Z$-normed space is that   not only this atomic space has the  localized $L^2$-type structure, which is very convenient,  but also it  is stronger than the corresponding $L^1$-type space. Note that  the  $L^1$-norm of  the profile of the nonlinear solution, which is the pullback of the nonlinear solution along the linear flow,  suggests the decay rate over time for the nonlinear solution. Hence, we will control the $Z$-norm of the profile instead of the $L^1$-norm.

 Our main result is stated as follows, 
\begin{theorem}\label{maintheorem}
Let $N_0=1000$ and $\delta \in(0,10^{-9}]$ be fixed and sufficiently small. If the initial data  $(h _0, \psi_0)$ satisfies the following estimate, 
\begin{equation}\label{initialcondition}
\|h_0\|_{H^{N_0+1/2}} + \|\Lambda\psi\|_{H^{N_0}} + \| (h_0, \Lambda \psi_0)\|_{Z  } + \|\mathcal{F}[(   {h_0},  {\Lambda \psi_0})](\xi)\|_{L^\infty_\xi} \leq \epsilon_0\leq \bar{\epsilon},\quad \Lambda:=\sqrt{\d\tanh\d},
\end{equation}
for some sufficiently small constant $\bar{\epsilon}$, then there is a unique global solution for the system \textup{(\ref{waterwave})} with initial data $(h_0, \psi_0)$. Moreover, the following estimate holds,
\begin{equation}\label{boundovertime}
\sup_{t\in[0,\infty)} (1+t)^{-\delta} \| (h, \Lambda\psi)(t)\|_{H^{N_0}} + (1+t)\|(h, \Lambda\psi)(t)\|_{W^{4,2\alpha}} + \| e^{it \Lambda}(h+i\Lambda\psi)(t)\|_{Z}\lesssim \epsilon_0.
\end{equation}

\end{theorem}

\begin{remark}
From (\ref{boundovertime}), we know that there is no traveling wave below a certain smallness level determined by`` $\bar{\epsilon}$'' in the above theorem.
\end{remark}

\begin{remark}
As a byproduct of deriving the improved $Z$-norm estimate for the profile of the nonlinear solution,  we   know that the solution is scattering to a linear solution in a lower regularity Sobolev space, e.g.,  $H^5(\R^2)$. 
\end{remark}

\subsection{Summary of the local results for the gravity waves system}\label{summary}

In this subsection, we will discuss the local behavior of  the gravity waves system (\ref{waterwave}) studied in  \cite{wang2}, which is the starting point of this paper. Note that the local existence of the system (\ref{waterwave}) is already known, e.g., see \cite{alazard2}. Our goal is to  extend the lifespan of the nonlinear solution.   Hence, it is very natural to use the bootstrap argument to iterate the local result. To close the bootstrap argument, it is very essential to  have  a good understanding of the dispersion of the nonlinear solution. 

Because the gravity waves system (\ref{waterwave}) is   quasilinear    and  moreover the system (\ref{waterwave}) behaves badly at the  low-frequency part in the flat bottom setting, it looks unlikely that the decay overtime rate of the nonlinear solution will be same as the decay rate of the corresponding linear solution. Note that even the decay rate of the linear solution, which is $1/(1+t)$, is  barely integrable to close the bootstrap argument.  As a result, a rough energy estimate is not sufficient to control the growth of energy in the long run. 

To get around this issue, we introduced a new energy estimate in \cite{wang2}, in which we paid special attention to the low-frequency part of the nonlinear solution. The reason why we did so is due to the  expectation  that derivatives can compensate for the  decay rate over time  for the nonlinear solution of (\ref{waterwave}). The intuition of having this expectation is simple. If the main issue lies in the low-frequency part, then the derivatives at the  low-frequency part, which are small, will provide extra smallness.

 We state the new energy estimate obtained in \cite{wang2} as follows, 
\begin{theorem}\label{newenergyestimatetheorem}
 If the initial data $(\,h_0, \Lambda \psi_0)\in H^{N_0+1/2}(\R^2)\times H^{N_0}(\R^2)$  satisfies the  smallness condition \textup{(\ref{initialcondition})},
then there exists some $T>0$ and a unique solution $(\,h, \Lambda\psi)\in C^{0}\big([0,T]; H^{N_0}(\R^2)\times H^{N_0}(\R^2)\big)$. Moreover, the following energy estimate holds for any  $t\in[0,T]$,
\begin{equation}\label{energyestimatetheorem}
 \|(\,h, \Lambda\psi)(t)\|_{H^{N_0}}^2\lesssim_{} \epsilon_0^2+ \int_0^t   \big[ \|(\,h, \Lambda\psi)(s)\|_{{W^{4,1}}} + \| (\,h, \Lambda\psi)(s)\|_{{W^4}}^2\big] \|(\,h, \Lambda\psi)(s)\|_{H^{N_0}}^2 ds.
\end{equation}
\end{theorem}
\begin{remark}
The smallness assumption assumed in \cite{wang2} is weaker than the smallness assumption (\ref{initialcondition}). Only $W^4$-norm of initial data is required to be small in \cite{wang2}.  Hence, we can use the new energy estimate derived  there directly in this paper.
\end{remark}

There are two main ingredients to derive the energy estimate (\ref{energyestimatetheorem}): (i) Thanks to the works of Alazard-Burq-Zuily \cite{alazard1, alazard2}, we can use their paralinerarization and symmetrization procedures to avoid losing derivatives at the  high-frequency part; (ii) The careful study of the Dirichlet--Neumann operator at the low-frequency part.

\subsection{Some properties of the Dirichlet--Neumann operator}
Note that the gravity waves system (\ref{waterwave}) is fully nonlinear, which is very inconvenient to analyze. Thanks to  a fixed point type structure lies in the Dirichlet--Neumann operator, due to the small date regime, it enables us  to control the $Z$-norm  of  the remainder terms  (the cubic and higher order terms). We discuss it with details in this subsection as follows.

To identify the fixed point type structure inside the Dirichlet--Neumann operator, we need to reformulate the velocity potential inside the water region $\Omega(t)$.  More precisely, for any fixed  time ``$t$'',   we map the water region $\Omega(t)$ to the strip $\mathcal{S}:= \mathbb{R}^2\times [-1,0]$ via change of coordinates  as follows,
\[
(x,y)\rightarrow (x,z), \quad z: = \displaystyle{\frac{y-h(t,x)}{h(t,x)+1}}, \quad z\in [-1,0].
 \]

We define the velocity potential in $(x,z)$-coordinate system as follows,
 \be\label{potentialinxzcoordinate}
 \varphi(x,z):= \phi(x, h(t,x)+ (h(t,x)+1)z).
 \ee
 From (\ref{harmoniceqn}),  the following identity holds, 
\begin{equation}\label{changeco1}
(\Delta_{x}+\p_y^2)\phi = 0 \Longrightarrow P\varphi:=[ \Delta_{x}+ \tilde{a}\p_z^2 +  \tilde{b}\cdot \nabla \p_z + \tilde{c}\p_z ] \varphi=0, \varphi\big|_{z=0}=\psi, \p_z\varphi\big|_{z=-1}=0,
\end{equation}
where
\begin{equation}\label{coeff}
\tilde{a}= \frac{1+(z+1)^2|\nabla h|^2}{(1+h)^2},\quad 
\tilde{b} = \frac{-2(z+1)\nabla h}{1+h},\quad \tilde{c}= \frac{-(z+1)\Delta_{x} h}{(1+h)} + 2\frac{(z+1)|\nabla h|^2}{(1+h)^2}. 
\end{equation}

As a result of direct computations,  we can formulate the Dirichlet-Neumann operator in terms of $\varphi$  as follows,
\begin{equation}\label{DN1}
G(h)\psi = [-\nabla h\cdot\nabla \phi + \p_y\phi]\big|_{y=h(t,x)} = \frac{1+|\nabla  h|^2}{1+h} \p_z \varphi \big|_{z=0} -\nabla\psi \cdot \nabla h. 
\end{equation}

From (\ref{DN1}), it is easy to see that the only nontrivial term inside the Dirichlet--Neumann operator is $\p_z\varphi\big|_{z=0}$. Therefore, to estimate $G(h)\psi$ in a normed space, e.g., a $X$-normed space, it is sufficient to estimate $\p_z \varphi$ in the $L^\infty_z X$-normed space.

Now, we will show that a fixed point type structure for ``$\nabla_{x,z}\varphi$'' is hidden inside the elliptic equation (\ref{changeco1}). To see so, we reformulate the equation (\ref{changeco1}) as follows, 
\be\label{noveqn20}
(\partial_z +|\nabla|) (\partial_z -|\nabla|)\varphi = (1-\tilde{a})\p_z^2 \varphi -\tilde{b}\cdot \nabla \p_z \varphi -\tilde{c}\p_z \varphi=g(z)=\p_z g_1(z) + g_2(z)+\nabla \cdot g_3(z),
\ee
where
\begin{equation}\label{eqn12}
g_1(z) =  \frac{2\,h+\,h^2 - (z+1)^2 |\nabla\,h|^2}{(1+\,h)^2} \p_z \varphi +\frac{(z+1)\nabla\,h\cdot \nabla\varphi}{1+\,h},\quad g_1(-1)=0,
\end{equation}
\begin{equation}\label{eqn14}
g_2(z) =\frac{(z+1)|\nabla \,h|^2 \p_z\varphi}{(1+\,h)^2}  - \frac{\nabla \,h \cdot \nabla\varphi}{1+\,h} ,\quad g_3(z)= \frac{(z+1)\nabla \,h \p_z\varphi}{1+\,h}, z\in [-1,0].
\end{equation}

By treating `` $g(z)$'' in (\ref{noveqn20}) as some given nonlinearity, we can solve $\varphi(x,z)$ explicitly from the equation (\ref{noveqn20}) and the  boundary conditions in (\ref{changeco1}). As a result,  we can solve $\nabla_{x,z}\varphi$ ``explicitly'' as follows,
\[\nabla_{x,z}\varphi = \Bigg[ \Big[ \frac{e^{-(z+1)\d}+ e^{(z+1)\d}}{e^{-\d} + e^{\d}}\Big]\nabla\psi ,  \frac{e^{(z+1)\d}- e^{-(z+1)\d}}{e^{-\d}+e^{\d}} \d \psi\Bigg] + [\mathbf{0}, g_1(z)]+ \]
\[
+\int_{-1}^{0} [K_1(z,s)-K_2(z,s)-K_3(z,s)](g_2(s)+\nabla \cdot g_3(s))  ds \]
\begin{equation}\label{fixedpoint}
+\int_{-1}^{0} K_3(z,s)\d\textup{sign($z-s$)}g_1(s)  -\d [K_1(z,s) +K_2(z,s)]g_1(s)\, d  s,
\end{equation}
where $K_i(z,s)$, $i\in\{1,2,3\}$, are some linear operators that only depend  on $z$ and $s$,  see \cite{wang2} for their detailed formulas.

From (\ref{eqn12}) and (\ref{eqn14}), it is easy to see that  $g_i(z)$, $i\in\{1,2,3\}$,  are all linearly depending  on $\nabla_{x,z}\varphi$. Moreover, $g_i(z)$, $i\in\{1,2,3\}$ are at the higher order than $\nabla_{x,z}\varphi$ because of the smallness of the height of interface ``$h(t,x)$''. Because of this observation, now it is easy to see that there exists a fixed point type structure inside (\ref{fixedpoint}).

From the fixed point type formulation (\ref{fixedpoint}), we can derive the Taylor expansion for the Dirichlet-Neumann operator, which is crucial to the study of the long time behavior of the water waves system. 

Because the decay rate of the nonlinear solution  is critical in the $3D$ case  ($2D$ interface), it is    crucial to know precisely what the linear term and the  quadratic terms of the Dirichlet-Neumann operator are.

 From (\ref{fixedpoint}) and (\ref{DN1}), it's easy to see that the linear terms of $\nabla_{x,z}\varphi$ and $G(h)\psi$ are given as follows, 
 \begin{equation}\label{eqn773}
 \Lambda_{1}[\nabla_{x,z}\varphi]= \Bigg[ \Big[ \frac{e^{-(z+1)\d}+ e^{(z+1)\d}}{e^{-\d} + e^{\d}}\Big]\nabla\psi ,  \frac{e^{(z+1)\d}- e^{-(z+1)\d}}{e^{-\d}+e^{\d}} \d \psi\Bigg],  \end{equation}
 \be\label{noveqn99}
 \Lambda_{1}[G(h)\psi]= \Lambda_{1}[\p_z\varphi\big|_{z=0}]= |\nabla|\tanh(|\nabla|)\psi.
\ee

Now, we  identify the quadratic terms of the Dirichlet-Neumann operator. Because of the hierarchy of the smallness,  we can  plug-in the linear terms of $
 \nabla_{x,z}\varphi$ in (\ref{eqn773}) to (\ref{fixedpoint}) to calculate explicitly the quadratic terms of $\p_z \varphi$,  which further give us the quadratic terms of $G(h)\psi$ from (\ref{DN1}). As a result (see \cite{wang2}[Lemma 3.4]), we have

\be\label{noveqn100}
\Lambda_{2}[G(h)\psi]= -\nabla_x\cdot(h\nabla_x\psi) - \d\tanh\d(h \d\tanh\d\psi).
\ee

For the cubic and higher order terms, although it is not necessary to figure out explicitly what they are, we still need to estimate them over time to show that they  do not have much accumulated effect  in the long run. From (\ref{fixedpoint}), we can  derive a   fixed point type formulation for  $\Lambda_{\geq 3}[\nabla_{x,z}\varphi]$ as follows,

\[
\Lambda_{\geq 3}[\nabla_{x,z}\varphi] =\sum_{i=1,2} C_z^i(h, \psi, \tilde{h}_i) + h\tilde{C}_z^i(h, \psi, \tilde{h}_i) + (1+h)^2\tilde{C}_z(\tilde{h}_2, \tilde{h}_2,\Lambda_{\leq 2}[\nabla_{x,z}\varphi])+ h^2 \hat{C}_z(h, \tilde{h}_2, \psi)
\]
\begin{equation}\label{fixed1}
 + T_{z}^i(\tilde{h}_i, \Lambda_{\geq 3}[\nabla_{x,z}\varphi]) + \widetilde{C}_z^1(h, \tilde{h}_2, \Lambda_{\geq 3}[\nabla_{x,z}\varphi]) + (1+h)^2 \widetilde{C}_z^2(\tilde{h}_2, \tilde{h}_2, \Lambda_{\geq 3}[\nabla_{x,z}\varphi]),
\end{equation}
where $ C_z^i$, $\tilde{C}_z^i$, $\tilde{C}_z$, $\hat{C}_z$, $\widetilde{C}_z^i$, $i\in\{1,2\}$ are some trilinear operators with symbols that satisfy the rough estimate (\ref{symbolestimatecubic}), $T_{z}^i $  are some bilinear operators with symbols that satisfy the rough estimate (\ref{symbolestimatequadratic}),  and  $\tilde{h}_1$ and $\tilde{h}_2$ are defined as follows, 
\[
\tilde{h}_1 = \frac{2h+h^2}{(1+h)^2}, \quad \tilde{h}_2 = \frac{h}{1+h},
\]
see \cite{wang2}[Lemma 3.7]. Due to the small data regime, it is easy to see that  the formulation (\ref{fixed1}) provides a mechanism to estimate $\Lambda_{\geq 3}[\nabla_{x,z}\varphi]$. 
 
\subsection{Main ideas of the proof of Theorem \ref{maintheorem}} 
 From  the new  energy estimate (\ref{energyestimatetheorem}) in Theorem \ref{newenergyestimatetheorem}, we know that it would be sufficient to close the argument if we can prove that the decay rate of the nonlinear solution in $W^{4, 2\alpha}$ space is sharp, which is $1/(1+t)$ over time.  From the linear decay estimates in Lemma \ref{lineardecay} and the fact that  the $Z$-norm constructed in (\ref{definitionofZnorm}) is stronger than the corresponding $L^1$ type norm, we can reduce our goal to    prove that the $Z$--norm of the profile doesn't grow over time.

Although the proof presented in this paper 	is very complicated at the technical level, we mention three key observations that make it possible to close the argument.

The first key observation is  that  we can decompose the phases of quadratic terms into two parts which have the same sign.  For example,   
\[
\underbrace{\Phi^{+,-}(\xi, \eta)}_{\textup{Phase}}:=\Lambda(|\xi|) - \Lambda(|\xi-\eta|)+\Lambda(|\eta|)= \underbrace{\Lambda(|\xi|) - \Lambda(|\xi|+|\eta|)  +\Lambda(|\eta|)}_{\text{Positive}}
\]
\be\label{noveqn31}
 + \underbrace{\Lambda(|\xi|+|\eta|)   - \Lambda(|\xi-\eta|)}_{\text{Positive}},\quad \Lambda(|\xi|):=\sqrt{|\xi|\tanh|\xi|},\quad  \xi, \eta\in \mathbb{R}^2.
\ee
Because of this observation, we know that the phases always have    a lower bound despite it is of cubic level smallness. See also the Lemma \ref{roughestimatephase2}.

The second key observation is that, we can gain one degree of the smallness of the output frequency in the $1\times 1\rightarrow 0$ type interaction, which means that the frequencies of two inputs are of size ``$1$'' and the frequency of the output is of size ``$0$'', see the estimate (\ref{sizeofsymboluniform}) in Lemma \ref{sizeofsymbol}. Although this smallness is not sufficiently strong  to control  completely  the accumulated $1\times 1\rightarrow 0$ type interaction effect over time, it makes the choice of small ``$\alpha$" in the definition of $Z$-norm in (\ref{definitionofZnorm}) possible. The availability of a small ``$\alpha$" is important, because of  the following two facts: (i) the gain from the choice of ``$\alpha$'' in the $1\times 1\rightarrow 0$ type interaction  becomes the corresponding loss   in the $0\times 1 \rightarrow 1$ type interaction, which means that the frequencies of the output and one input are of size ``$1$'' and the frequency of the other input is of size ``$0$''; (ii) The  null structure is not available for the gravity waves system (\ref{waterwave}) in the Low $\times$ High type interaction, because the size of the symbol is   ``$1$'' instead of ``$0$'' in the $0\times 1\rightarrow 1$ type interaction.

The third key observation is that the angle between the output frequency and the relatively smaller input frequency plays an important role when the phases associated with quadratic terms degenerate. To illustrate this observation,  we use the phase $\Phi^{+,-}(\xi, \eta)$ and the case   $|\eta|\ll |\xi|$ as an example. From (\ref{noveqn31}), it is easy to see that 
\[
\Phi^{+,-}(\xi, \eta)\geq \Lambda(|\xi|+|\eta|) - \Lambda(|\xi-\eta|) \approx \Lambda'(|\xi|)\big(|\xi|+|\eta|-|\xi-\eta| \big)= \Lambda'(|\xi|)\frac{2|\xi||\eta|(1+\cos(\angle(\xi, \eta)))}{|\xi|+|\eta|+|\xi-\eta|}.
\]

From the above approximation, we can see that  the size of phase is of linear level smallness, which is not so small, if the angle between $\xi$ and $-\eta$ is not small and	the size of phase is of cubic level  smallness, which is the worst scenario, if $\xi$ and $-\eta$ is almost in the same direction. That is to say,   the size of the $\Phi^{+,-}(\xi, \eta)$ highly depends on the angle between ``$\xi$" and ``$-\eta$".

 Because of this observation, we will first localize the angle between the output frequency and the relatively smaller input frequency  if the associated phase is highly degenerated and then carefully analyze the role of this angle in the $Z$-norm estimate.

\subsection{The outline of this paper}

This paper is organized as follows.
\begin{enumerate}

	\item[$\bullet$]   In section \ref{somelemmas}, we  introduce the notation used in this paper, reduce the system (\ref{waterwave}) into a quasilinear dispersive equation, and then   prove some  bilinear estimates with the angle localized, which are very important in the $Z$-norm estimate.   

	\item[$\bullet$] In section \ref{improvedZnorm}, based on the behavior of the associated phases,  we first introduce the set-up of the  $Z$-norm estimate for the profile and  then decompose the quadratic terms into good type terms and bad type terms.   
	\item[$\bullet$]  In section \ref{goodimprovedZnorm} and section \ref{badimprovedZnorm}, we  derive the improved $Z$-norm estimate for the good type  terms and the bad type terms respectively by assuming that the improved  $Z$-norm estimate for remainder terms (cubic and higher order terms) holds. 
	\item[$\bullet$]
In section \ref{remainderZnorm}, we derive  the improved $Z$-norm estimate for the remainder terms. Hence finishing the bootstrap argument. 

	\item[$\bullet$]   In the Appendix \ref{analysisphase}, we analyze  properties of phases associated with system (\ref{waterwave}). 
\end{enumerate}

\vo
\noindent \textbf{Acknowledgment\quad } The author thanks his Ph.D. advisor Alexandru Ionescu for many helpful discussions and suggestions. The first version of the manuscript was completed when the author  visited Fudan University and BICMR, Peking University. The author thanks their warm hospitalities during the visits.

\section{Notation and Some Lemmas}\label{somelemmas}
\subsection{Notations }\label{notation} 
 For any two numbers $A$ and $B$, we use  $A\lesssim B$ and $B\gtrsim A$ to denote  $A\leq C B$, where $C$ is an absolute constant. We use $A\sim B$ to denote the case  $A\lesssim B$ and $B\lesssim A$. We use $A\approx B$ to denote the case  $|A-B|\leq c|A|$, where $c$ is some  small absolute constant. For any two vectors $\xi, \eta\in \R^2$, we use $\angle(\xi, \eta)$ to denote  the angle between $\xi$ and $\eta$. Moreover, we use the convention that $\angle(\xi, \eta)\in [0, \pi]$. 
 
Throughout this paper, we will slightly abuse the notation of $``\Lambda"$.  When  there is no lower script under $\Lambda$, then `` $\Lambda$'' denotes ``$	\sqrt{\tanh(|\nabla|)|\nabla|}$'', which is the linear operator associated for the system  (\ref{waterwave}). When there is a lower script ``$p$'' under $\Lambda$ where $p\in \mathbb{N}_{+}$, then   $\Lambda_{p}(\mathcal{N})$ denotes the $p$-th order terms of the nonlinearity $\mathcal{N}$  if a Taylor expansion of $\mathcal{N}$ is available. Also, we use  $\Lambda_{\geq p}[\mathcal{N}]$ to denote the $p$-th  and higher orders terms. More precisely, $\Lambda_{\geq p}[\mathcal{N}]:=\sum_{q\geq p}\Lambda_{q}[\mathcal{N}]$.  

We provide an example here to better illustrate the notation. For example, 
 $\Lambda_{2}[\mathcal{N}]$ denotes the quadratic term of $\mathcal{N}$ and $\Lambda_{\geq 2}[\mathcal{N}]$  denotes the quadratic and higher order terms of $\mathcal{N}$.

For an integrable function $f(x)$, the Fourier transform of $f$ is defined as follows, 
\[
\mathcal{F}(f)(\xi)= \int e^{-ix \cdot \xi} f(x) d x.
\]

We  will also use   $\widehat{f}(\xi)$  to denote the Fourier transform of $f$. We use $\mathcal{F}^{-1}(g)$ to denote the inverse Fourier transform of $g(\xi)$.

We  fix an even smooth function $\tilde{\psi}:\R \rightarrow [0,1]$ supported in $[-3/2,3/2]$ and equals to $1$ in $[-5/4, 5/4]$. For any $k\in \mathbb{Z}$, we define
\[
\psi_{k}(x) := \tilde{\psi}(x/2^k) -\tilde{\psi}(x/2^{k-1}), \quad \psi_{\leq k}(x):= \tilde{\psi}(x/2^k)=\sum_{l\leq k}\psi_{l}(x), \quad \psi_{\geq k}(x):= 1-\psi_{\leq k-1}(x).
\]

We use  $P_{k}$, $P_{\leq k}$ and $P_{\geq k}$ to denote the  Fourier multiplier operators with symbols $\psi_{k}(\xi),$ $\psi_{\leq k}(\xi)$ and $\psi_{\geq k }(\xi)$ respectively. We   use  $f_{k}(x)$ to abbreviate $P_{k} f(x)$.

 For an integer $k\in\mathbb{Z}$, we use $k_{+}$ to denote $\max\{k,0\}$ and  use $k_{-}$ to denote $\min\{k,0\}$.  The cutoff function ``$\varphi_j^k(x)$'' used in (\ref{definitionofZnorm}) is defined as follows, 
 \begin{equation}
 \varphi_j^{k}(x):=\left\{\begin{array}{cc}
 {\psi}_{\leq -k_{-} }(x) & \textup{if}\,  j =-k_{-}\\
 	{\psi}_j(x) & \textup{if} j> k_{-}. \\
\end{array}\right.
 \end{equation}

We define a linear operator ``$Q_{k,j}$'' as follows, 
\begin{equation}\label{spatialfrequencylocal}
Q_{k, j} f := P_{[k-2,k+2]}[\varphi_j^k(x)\cdot P_k f].
\end{equation}
We use  $f_{k,j}(x)$ to abbreviate $Q_{k,j}f(x)$.
From the above definition, it's easy to see  that the following decomposition holds
\be\label{dyadicdecompositionfs}
P_k f = \sum_{j\geq -k_{-}} Q_{k,j} f, \quad f = \sum_{k\in \mathbb{Z}} P_k f =  \sum_{k\in \mathbb{Z}}\sum_{j\geq -k_{-}} Q_{k,j} f.
\ee

For any integrable function $f$, we define
\be\label{signnotation}
f^{+}:=f,\quad P_{+}[f]:=f , \quad f^{-}:= \bar{f}, \quad P_{-}[f]:=\bar{f}.
\ee

For two localized functions $f(x)$ and $g(x)$ and  a bilinear form  $Q(f,g)$, we  use the convention that the symbol $q(\cdot, \cdot)$ of $Q(\cdot, \cdot)$  is defined in the following sense throughout this paper,
\begin{equation}
\mathcal{F}[Q(f,g)](\xi)=\frac{1}{4\pi^2} \int_{\R^2} \widehat{f}(\xi-\eta)\widehat{g}(\eta)q(\xi-\eta, \eta) d \eta. 
\end{equation}
Very similarly, for a trilinear form $C(f, g, h)$, its symbol $c(\cdot, \cdot, \cdot)$ is defined in the following sense throughout this paper, 
\[
\mathcal{F}[C(f,g,h)](\xi) = \frac{1}{16\pi^4} \int_{\R^2}\int_{\R^2} \widehat{f}(\xi-\eta)\widehat{g}(\eta-\sigma) \widehat{h}(\sigma) c(\xi-\eta, \eta-\sigma, \sigma) d \eta d \sigma.
\]

We define a class of symbol and its associated norms as follows,
\[
\mathcal{S}^\infty:=\{ m: \mathbb{R}^4\,\textup{or}\, \mathbb{R}^6 \rightarrow \mathbb{C}, m\,\textup{is continuous and }  \quad \| \mathcal{F}^{-1}(m)\|_{L^1} < \infty\},
\]
\[
\| m\|_{\mathcal{S}^\infty}:=\|\mathcal{F}^{-1}(m)\|_{L^1}, \quad \|m(\xi,\eta)\|_{\mathcal{S}^\infty_{k,k_1,k_2}}:=\|m(\xi, \eta)\psi_k(\xi)\psi_{k_1}(\xi-\eta)\psi_{k_2}(\eta)\|_{\mathcal{S}^\infty},
\]
\[
 \|m(\xi,\eta,\sigma)\|_{\mathcal{S}^\infty_{k,k_1,k_2,k_3}}:=\|m(\xi, \eta,\sigma)\psi_k(\xi)\psi_{k_1}(\xi-\eta)\psi_{k_2}(\eta-\sigma)\psi_{k_3}(\sigma)\|_{\mathcal{S}^\infty}.
\]

We have the following lemma on the multilinear estimates,
\begin{lemma}\label{multilinearestimate}
Assume that $m$, $m'\in S^\infty$, $p, q, r, s \in[1, \infty]$ , then the following multilinear estimates hold,    
\begin{equation}\label{productofsymbol}
\| m\cdot m'\|_{S^\infty} \lesssim \| m \|_{S^\infty}\| m'\|_{S^\infty},
\end{equation}
\begin{equation}\label{bilinearesetimate}
\Big\| \mathcal{F}^{-1}\big[\int_{\R^2} m(\xi, \eta) \widehat{f}(\xi-\eta) \widehat{g}(\eta) d \eta\big]\Big\|_{L^p} \lesssim \| m\|_{\mathcal{S}^\infty}\| f \|_{L^q}\| g \|_{L^r}  \quad \textup{if}\,\,\, \frac{1}{p} = \frac{1}{q} + \frac{1}{r},
\end{equation}
\begin{equation}\label{trilinearesetimate}
\Big\| \mathcal{F}^{-1}\big[\int_{\R^2}\int_{\R^2} m'(\xi, \eta,\sigma) \widehat{f}(\xi-\eta) \widehat{h}(\sigma) \widehat{g}(\eta-\sigma)  d \eta d\sigma\big] \Big\|_{L^{p}} \lesssim \|m'\|_{\mathcal{S}^\infty} \| f \|_{L^q}\| g \|_{L^r} \| h\|_{L^s},\,\, \end{equation}
where $ { {1}/{p} =  {1}/{q}+  {1}/{r} +   {1}/{s}}.$

\end{lemma}
\begin{proof}
The proof is standard. See \cite{ IP1} for details.
\end{proof}

To estimate the $\mathcal{S}^{\infty}_{k,k_1,k_2}$  or the $\mathcal{S}^{\infty}_{k,k_1,k_2,k_3}$ norms of symbols, we   use the following Lemma. 
\begin{lemma}\label{Snorm}
If $f:\mathbb{R}^{2i}\rightarrow \mathbb{C}$ is a smooth function, $i\in\{2,3\}$, then the following estimate holds for any  $k_1,\cdots, k_i\in\mathbb{Z}$,
\begin{equation}\label{eqn61001}
\| \int_{\mathbb{R}^{2i}} f(\xi_1,\cdots, \xi_i) \prod_{j=1}^{i} e^{i x_j\cdot \xi_j} \psi_{k_j}(\xi_j) d \xi_1\cdots  d\xi_i \|_{L^1_{x_1, \cdots, x_i}} \lesssim \sum_{m=0}^{i+1}\sum_{j=1}^i 2^{m k_j}\|\p_{\xi_j}^m f\|_{L^\infty} .
 \end{equation}
\end{lemma}
\begin{proof}
Let's first consider the case when $i=2$. Through scaling, it is  sufficient to prove the above estimate for the case when $k_1=k_2=0$. From Plancherel theorem, we have the following two estimates, 
\[
 \| \int_{\mathbb{R}^{2i}} f(\xi_1,\xi_2) e^{i (x_1\cdot \xi_1+ x_2\cdot \xi_2)} \psi_{0}(\xi_1) \psi_{0}(\xi_2) d \xi_1 d\xi_2 \|_{L^2_{x_1, x_2}}\lesssim \| f(\xi_1, \xi_2)\|_{L^\infty_{\xi_1, \xi_2}},
\]
\[
 \| (|x_1|+|x_2|)^3 \int_{\mathbb{R}^{2i}} f(\xi_1,\xi_2) e^{i (x_1\cdot \xi_1+ x_2\cdot \xi_2)} \psi_{0}(\xi_1) \psi_{0}(\xi_2) d \xi_1 d\xi_2 \|_{L^2_{x_1, x_2}}\lesssim  \sum_{m=0}^3\big[\|\p_{\xi_1}^m f\|_{L^\infty} + \| \p_{\xi_2}^m f \|_{L^\infty}\big],
\]
which are sufficient to finish the proof of (\ref{eqn61001}). The proof of the case  $i=3$  is very similar. Hence we omit the details here.
\end{proof}
 
We will use  the following lemma to derive the  $L^\infty$-decay estimate for the corresponding  linear solution of the gravity waves system (\ref{waterwave}).
\begin{lemma}\label{lineardecay}
For $f\in L^1(\R^2)$, the following $L^\infty$ type estimates hold,
\begin{equation}\label{highdecay}
\| e^{it \sqrt{\d\tanh\d}} P_{k} f\|_{L^\infty} \lesssim (1+|t|)^{-1} 2^{3k/2} \| f\|_{L^1},  \quad \textup{if $k\geq 0$}.
\end{equation}
\begin{equation}\label{lowdecay}
\| e^{it \sqrt{\d\tanh\d}} P_{k} f\|_{L^\infty} \lesssim (1+|t|)^{-\frac{1+\theta}{2}}2^{\frac{(3-3\theta)k}{2} } \| f\|_{L^1},\quad 0\leq \theta \leq 1, \quad \textup{if $k\leq 0$}.
\end{equation}

\end{lemma}
\begin{proof}
After  checking the expansion of the phase,    we can apply the main result in \cite{Guoz}[Theorem 1:(a)\&(b)] directly to derive the  above estimates.
\end{proof}
\subsection{Reduction of the gravity waves system (\ref{waterwave})}

In this subsection, we reformulate the gravity waves system (\ref{waterwave}), which is a coupled system, into a quasilinear dispersive equation, which is diagonalized and has explicit quadratic terms. 

Recall (\ref{noveqn99}) and (\ref{noveqn100}). Based on the order of nonlinearities, we can rewrite the gravity waves system (\ref{waterwave}) as follows, 
\begin{equation}\label{wwreformulate}
\left\{\begin{array}{l}
\p_t h= \Lambda^2 \psi  -\nabla_x\cdot(h\nabla_x\psi) - \d\tanh\d(h \d\tanh\d\psi)+ \Lambda_{\geq 3}[G(h)\psi],\\
\\
\p_t \psi =\displaystyle{ -h - \frac{1}{2} |\nabla \psi|^2 + \h |\Lambda^2 \psi|^2   + \h \Lambda_{\geq 3}\Big[(1+|\nabla h|^2)(B(h)\psi)^2\Big]},
\end{array}\right.
\end{equation}
where
\be\label{noveqn102}
B(h)\psi= \frac{G(h)\psi + \nabla_x h \cdot \nabla_x \psi}{{1+|\nabla_x h|^2}}.
\ee
Define $u= h + i \Lambda \psi$, where $\Lambda:=\sqrt{\d\tanh\d}$. Very naturally, we have
\begin{equation}\label{eqn1200}
h = \frac{u+\bar{u}}{2}, \quad \psi =\Lambda^{-1} \big({c_{+}u + c_{-}\bar{u}} \big), \quad \textup{where\,\,}c_{+}:= -i/2, c_{-}:= i/2.
\end{equation}
From (\ref{wwreformulate}), we can  derive the equation satisfied by ``$u$''  from (\ref{wwreformulate}) as follows, 
\begin{equation}\label{complexversion}
(\p_t + i\Lambda)u = \sum_{\mu, \nu\in\{+,-\}} Q_{\mu, \nu}(u^{\mu}, u^{\nu}) + \mathcal{R}, 
\end{equation}
where
\begin{equation}\label{quadraticterms}
Q_{\mu, \nu}(u^{\mu}, u^{\nu})= -\frac{c_{\nu}}{2}\nabla_x\cdot(u^{\mu}\nabla \Lambda^{-1} u^{\nu})- \frac{c_{\nu}}{2}\d\tanh\d( u^{\mu} \Lambda u^{\nu})+ \frac{i c_{\mu}c_{\nu}}{2}\Lambda \big[-\frac{\nabla}{\Lambda} u^{\mu}\cdot \frac{\nabla}{\Lambda} u^{\nu} + \Lambda u^{\mu} \Lambda u^{\nu}\big],
\ee
\begin{equation}\label{remainderterms}
\mathcal{R}= \Lambda_{\geq 3}[\p_t h] + i \Lambda \Lambda_{\geq 3}[\p_t \psi]= \Lambda_{\geq 3}[G(h)\psi] + i \Lambda\big( \Lambda_{\geq 3}[(1+|\nabla h|^2 ) (B(h)\psi)^2]\big).
\end{equation}
Note that, in (\ref{quadraticterms}), we used the  notation defined in (\ref{signnotation}).

For any fixed $k\in \mathbb{Z}$, we define   
\[
\chi_k^1:=\{(k_1,k_2): k_1, k_2\in \mathbb{Z}, k\leq \max\{k_1,k_2\}-5\leq \min\{k_1,k_2\}\},
\]  
\[
\chi_k^2:=\{(k_1,k_2): k_1, k_2\in \mathbb{Z}, k_2\leq k-5, |k_1-k|\leq 4\},
\]
\[
\chi_k^3:=\{(k_1,k_2): k_1, k_2\in \mathbb{Z}, k-5< \min\{k_1,k_2\}\leq \max\{k_1,k_2\}<k+5\},
\]
\[
\chi_k^4:=\{(k_1,k_2): k_1, k_2\in \mathbb{Z}, k_1\leq k-5, |k_2-k|\leq 4\},
\]
where $\chi_k^1$ corresponds to the High $\times$ High type interaction with the  output frequency relatively small, $\chi_k^2$ and $\chi_k^4$ corresponds to the High $\times$ Low type  and the Low $\times$ High type interactions with a  relatively small input frequency, and $\chi_k^3$ corresponds to  the case when  two inputs frequencies  and the output frequency are all comparable. 

 When $(k_1,k_2)\in \chi_k^4$, we can do change of coordinates to switch the roles of  $k_1$ and $k_2$. As a result, we have
\be\label{noveqn611}
\sum_{\mu, \nu\in\{+,-\}} Q_{\mu, \nu}(u^{\mu}, u^{\nu})= \sum_{\mu, \nu\in\{+,-\}}\sum_{k\in\mathbb{Z}} \sum_{(k_1,k_2)\in \chi_k^1\cup \chi_k^2\cup \chi_k^3} \tilde{Q}_{\mu, \nu}( u^{\mu}_{k_1} ,  u^{\nu}_{k_2}),
\ee
for some bilinear operator $\tilde{Q}_{\mu, \nu}(u^{\mu}, u^{\nu})$. Let   $q_{\mu, \nu}(\xi-\eta, \eta)$  denotes the symbol of quadratic term  $\tilde{Q}_{\mu, \nu}(u_{
\mu}, u_{\nu})$.
 
\begin{lemma}\label{sizeofsymbol}
For $k,k_1, k_2\in \mathbb{Z}$ and any $\mu, \nu\in\{+,-\}$, the following estimates holds for any $k\in \mathbb{Z}, (k_1,k_2)\in \chi_k^1\cup \chi_k^2\cup \chi_k^3$,
\begin{equation}\label{sizeofsymboluniform}
\|q_{\mu, \nu}(\xi-\eta, \eta)\|_{\mathcal{S}_{k,k_1,k_2}^\infty}\lesssim 2^{k/2+k_{-}/2+k_{1,+}}.
\end{equation}
\end{lemma}
\begin{proof}
From (\ref{quadraticterms}) and (\ref{noveqn611}), it is easy to compute the symbol $q_{\mu, \nu}(\xi-\eta, \eta)$ explicitly.  Therefore  the desired estimate  (\ref{sizeofsymboluniform}) holds directly  from the estimate (\ref{eqn61001}) in Lemma \ref{Snorm}.
\end{proof}

\subsection{Bootstrap assumption and proof of the main theorem} We   prove   Theorem \ref{maintheorem} via the standard bootstrap argument. The  bootstrap assumption is stated as follows,
\begin{equation}\label{smallness}
\sup_{t\in[0,T]} (1+t)^{-\delta}\| (h, \Lambda\psi)(t)\|_{H^{N_0}}   + \| e^{it\Lambda}( h + i \Lambda \psi)\|_Z \lesssim \epsilon_1:=\epsilon_0^{5/6}.
\end{equation}
As a result of the new energy estimate  (\ref{energyestimatetheorem}) in Theorem \ref{newenergyestimatetheorem}, the following Proposition holds.  
\begin{proposition}\label{propositon1}
Under the bootstrap assumption \textup{(\ref{smallness})}, we have the following estimate,
\begin{equation}\label{energyestimate}
\sup_{t\in[0,T]} (1+t)^{-\delta}\|(h, \Lambda\psi)(t)\|_{H^{N_0}} \lesssim \epsilon_0.
\end{equation}
\end{proposition}
\begin{proof}
Note that the following estimate holds  under the bootstrap assumption (\ref{smallness}),\[
\|(h, \Lambda \psi)(t)\|_{W^{4}}\lesssim \|(h, \Lambda \psi)(t)\|_{H^{N_0}}^{1/4}\| (h, \Lambda\psi)(t)\|_{W^{4, 2\alpha}}^{3/4}\lesssim (1+t)^{-3/4+2 \delta} \epsilon_0.
\]
 From the estimate (\ref{energyestimatetheorem}) in Theorem \ref{newenergyestimatetheorem}, the following estimate holds for any $t\in[0,T]$, 
\[
\|(h,\Lambda \psi)(t)\|_{H^{N_0}}^2 \lesssim \epsilon_0^2 +\int_{0}^{t} \frac{\epsilon_1^3}{(1+s)^{1-2\delta}}  d s \lesssim (1+t)^{2\delta}\epsilon_0^2.
\]
Hence the desired estimate (\ref{energyestimate}) holds. 
\end{proof}
 The rest of this paper is devoted to proving the following Proposition, which is sufficient to close the bootstrap argument. 
\begin{proposition}\label{proposition2}
Under the bootstrap assumption \textup{(\ref{smallness})} and the energy estimate \textup{(\ref{energyestimate})}, we have
\begin{equation}\label{improvedZnormestimate}
\sup_{t\in[0,T]}  \| e^{it\Lambda}( h + i \Lambda \psi)\|_Z\lesssim \epsilon_0.
\end{equation}
Therefore, from the linear decay estimates \textup{(\ref{highdecay}) and (\ref{lowdecay})} in Lemma \textup{\ref{lineardecay}}, the following decay estimate holds, 
\be
\sup_{t\in[0,T]} (1+t)\|(h, \Lambda\psi)(t)\|_{W^{4,2\alpha}} \lesssim \epsilon_0.
\ee
\end{proposition}
\subsection{Bilinear estimates with the  angle localized}

In the later $Z$-norm estimate, we need to localize the angle between the output frequency and  the smaller input frequency to exploit the observation that the size of the degenerated phase depends on this angle. Unavoidably, we need to estimate some bilinear operators with the angle localized. More precisely, we have the following two Lemmas.

\begin{lemma}\label{bilinearest}
For $l,k, k_1, k_2\in \mathbb{Z}$, $l\leq 2$, $k_2\leq k_1$, and   $f,g \in L^2\cap L^1$, we define a bilinear form  as follows, \[
T(f, g) =   \int_{\R^2}  e^{it \Phi^{\mu, \nu}(\xi, \eta)} \psi_{k}(\xi) \psi_{k_1}(\xi-\eta)\psi_{k_2}(\eta) \psi_{l}(\angle(\xi, \nu \eta)) m(\xi, \eta)\widehat{f^{\mu}}(\xi-\eta) \widehat{g^{\nu}}(\eta) d \eta,
\]
where  $\mu, \nu\in \{+,-\}$, $m(\xi, \eta) \in  S^{\infty}$, and the phase ``$\Phi^{\mu, \nu}(\xi,\eta)$'' is defined as follows, 
\be\label{definitionofphase}
\Phi^{\mu, \nu}(\xi, \eta):=\Lambda(|\xi|)-\mu\Lambda(|\xi-\eta|)- \nu\Lambda(|\eta|), \quad \Lambda(|\xi|):=\sqrt{|\xi|\tanh|\xi|}.
\ee
Then  the  following estimates hold,  
\begin{equation}\label{eqn293}
\| T(f,g)\|_{L^2} \lesssim \| m\|_{L^\infty_{\xi, \eta}} \min\big\{ 2^{2k_2 +l} \| f\|_{L^2} \|\widehat{g}\|_{L^\infty_\xi}, 2^{k_2+l/2}\| f\|_{L^2} \| g\|_{L^2}, 2^{k_2+k_1+l}\| \widehat{f}\|_{L^\infty_{\xi}} \| g\|_{L^2}\big\},\end{equation}
\begin{equation}\label{eqn298}
\| T(f,g)\|_{L^2} \lesssim \| m\|_{L^\infty_{\xi, \eta}}  \min\{2^{k+k_1+l}\| \widehat{f}\|_{L^\infty_{\xi}}\|g\|_{L^2},2^{k+l/2}\| f\|_{L^2} \|g\|_{L^2}, 2^{k_2+k+l} \| \widehat{g}\|_{L^\infty_\xi}\| f\|_{L^2}  \},\end{equation}
\[
\| T(f,g)\|_{L^2} \lesssim  \| m\|_{\mathcal{S}^\infty_{k,k_1,k_2}} \min\{ \| f_{k_1}\|_{L^2} \| \mathcal{F}^{-1}[e^{-it \Lambda(\xi)} \widehat{g}(\xi)\psi_{k_2}(\xi)]\|_{L^\infty_x},  
\]
\begin{equation}\label{equation6884}
2^{(k_1-k_2)/2} \| g(\eta)\psi_{k_2}(\eta)\|_{L^2} \| \mathcal{F}^{-1}[e^{-it \Lambda(\xi)} \widehat{f}(\xi)\psi_{k_1}(\xi)]\|_{L^\infty_x} \}. 
\end{equation}
\end{lemma}
\begin{proof}
$\bullet$\quad We first prove the desired estimates (\ref{eqn293}) and (\ref{eqn298}).

Note that  for any given small number $0< 2^{n}\lesssim 1$, we can decompose the unit  circle $\mathbb{S}^1$ into the union of angular sections with bounded (with upper bound given by an absolute constant) overlaps, where each sector has angular size $2^{n}$.  These cutoff functions  form a partition of unity.  We label those sectors by their angles $\omega = \xi/|\xi|$, use $|\omega|$ to denotes the size of angle and use $b^{\omega}_{n}(\xi)$ to denote a fixed standard bump function that supported in this sector and form a partition of unity.

With  the above defined notation, we  use the angular partition of unity  for  ``$\xi$'',  ``$\xi-\eta$'' and ``$\eta$''. Because of the localized angle between $\xi$ and $\nu \eta$,  the following decomposition holds, 
\[
T(f, g) =  \sum_{\begin{subarray}{c}
|\omega_1|, |\omega_2|\sim 2^{l+k_2-k_1}\\
|\omega_1\pm\omega_2|\sim 2^{l+k_2-k_1}\\
|\omega_3|\sim 2^{l},|\omega_1-\nu \omega_3|\sim 2^{l}\\
\end{subarray}} \int_{\R^2}  e^{it \Phi^{\mu, \nu}(\xi, \eta)} \psi_{k}(\xi) \psi_{k_1}(\xi-\eta)\psi_{k_2}(\eta) \psi_{l}(\angle(\xi, \nu \eta)) m(\xi, \eta)
\]
\be\label{noveeqn141}
\times b^{\omega_1}_{l+k_2-k_1}(\xi)b^{\omega_2}_{l+k_2-k_1}(\xi-\eta)   b^{\omega_3}_{l}(\eta) \widehat{f^{\mu}}(\xi-\eta) \widehat{g^{\nu}}(\eta) d \eta.
\ee
From the $L^2$-orthogonality of the localized angle in ``$\xi$'', the following estimate holds, 
\[
\| T(f,g)\|_{L^2}^2 \lesssim  \sum_{\begin{subarray}{c}
\omega_1,\omega_2,\omega_3\\
\textup{same as in (\ref{noveeqn141})}\\
\end{subarray}}  \big\| \int_{\R^2}  e^{it \Phi^{\mu, \nu}(\xi, \eta)} \psi_{k}(\xi) \psi_{k_1}(\xi-\eta)\psi_{k_2}(\eta)m(\xi, \eta) 
\]
\begin{equation}\label{equation6850}
\times b^{\omega_1}_{l+k_2-k_1}(\xi)b^{\omega_2}_{l+k_2-k_1}(\xi-\eta)  \widehat{f^{\mu}}(\xi-\eta) \widehat{g^{\nu}}(\eta) b^{\omega_3}_{l}(\eta) d \eta  \big\|_{L^2_{\xi}}^2
\end{equation}
\[
\lesssim \| m \|_{L^\infty_{\xi, \eta}}^2 \min \{ \sum_{\begin{subarray}{c}
\omega_1,\omega_2,\omega_3\\
\textup{same as in (\ref{noveeqn141})}\\
\end{subarray}}  2^{4k_2+2l} \| \widehat{f}(\xi- \eta) b^{\omega_2}_{l+k_2-k_1}(\xi-\eta) \psi_{k_1}(\xi-\eta) \|_{L^2}^2 \|\widehat{g}(\eta)\|_{L^\infty_{\eta}}^2,
\]
\begin{equation}\label{eqn280}
\sum_{\begin{subarray}{c}
\omega_1,\omega_2,\omega_3\\
\textup{same as in (\ref{noveeqn141})}\\
\end{subarray}}   2^{2k_2+l}  \| \widehat{f}(\xi- \eta) b^{\omega_2}_{l+k_2-k_1}(\xi-\eta) \psi_{k_1}(\xi-\eta) \|_{L^2}^2 \|b^{\omega_3}_{l}(\eta)\psi_{k_2}(\eta)\widehat{g}(\eta)\|_{L^2_{\eta}}^2\}.\end{equation}
For the simplicity of notation, in (\ref{equation6850}) and (\ref{eqn280}), $\omega_1$, $\omega_2$, and $\omega_3$ belong to the same set listed under the summation of (\ref{noveeqn141}). 

In  the first estimate of (\ref{eqn280}), we used the volume of support of $\eta$; in the second estimate of (\ref{eqn280}), we used the Cauchy-Schwartz inequality for the integration with respect to ``$\eta$''.

Recall the set of $\omega_1, \omega_2, \omega_3$ under the summation in (\ref{noveeqn141}). Because of the partition of unity, it is easy to check  the following two facts: (i)  the multiplicity (i.e., how many times it has been counted) of the summation with respect to $\omega_2$ is a finite number; (ii) the multiplicity of the summation with respect to $\omega_3$ is $2^{k_1-k_2}$ because there are $2^{k_1-k_2}$  sectors $\omega_2$ satisfy $|\omega_1- \nu \omega_3|\sim 2^l$ and $|\omega_1\pm \omega_2|\sim 2^{l+k_2-k_1}$ for a fixed sector $\omega_3$.

 Therefore,  the following estimate holds after summarizing with respect to $\omega_2$, 
\[
\sum_{\begin{subarray}{c}
\omega_1,\omega_2,\omega_3\\
\textup{same as in (\ref{noveeqn141})}\\
\end{subarray}}  2^{4k_2+2l} \| \widehat{f}(\xi- \eta) b^{\omega_2}_{l+k_2-k_1}(\xi-\eta) \psi_{k_1}(\xi-\eta) \|_{L^2}^2 \|\widehat{g}(\eta)\|_{L^\infty_{\eta}}^2
\lesssim 2^{4k+2l} \| f\|_{L^2} \| \widehat{g}\|_{L^\infty_\xi}
\]
\[
\sum_{\begin{subarray}{c}
\omega_1,\omega_2,\omega_3\\
\textup{same as in (\ref{noveeqn141})}\\
\end{subarray}}   2^{2k_2+l}  \| \widehat{f}(\xi- \eta) b^{\omega_2}_{l+k_2-k_1}(\xi-\eta) \psi_{k_1}(\xi-\eta) \|_{L^2}^2 \|b^{\omega_3}_{l}(\eta)\psi_{k_2}(\eta)\widehat{g}(\eta)\|_{L^2_{\eta}}^2 \lesssim 2^{2k_2+l}\| f\|_{L^2}\|g\|_{L^2}.
\]

 Alternatively,  the following estimate holds after summarizing with respect to $\omega_3$ and using the volume of support of $\widehat{f}(\cdot)$,
\[
  \sum_{\begin{subarray}{c}
\omega_1,\omega_2,\omega_3\\
\textup{same as in (\ref{noveeqn141})}\\
\end{subarray}}  2^{2k_2+l}  \| \widehat{f}(\xi-\eta) b^{\omega_2}_{l+k_2-k_1}(\xi-\eta) \psi_{k_1}(\xi-\eta) \|_{L^2}^2 \|b^{\omega_3}_{l}(\eta)\psi_{k_2}(\eta)\widehat{g}(\eta)\|_{L^2_{\eta}}^2\]
\[
 \lesssim 2^{2k_2+l} 2^{k_1-k_2} 2^{2k_1+k_2-k_1+l}\| \widehat{f}\|_{L^\infty_{\xi}}^2 \sum_{\omega_3} \|b^{\omega_3}_{l}(\eta)\psi_{k_2}(\eta)\widehat{g}(\eta)\|_{L^2_{\eta}}^2 
 \lesssim 2^{2k_1+2k_2+2l}\|\widehat{f}\|_{L^\infty_{\xi}}^2\| g\|_{L^2}^2.
\]
Hence finishing the  proof of the desired estimate (\ref{eqn293}).

On the other hand, if we   use the  size of support of $\xi$ first, then the following estimate holds in the same spirit as the proof of (\ref{eqn293}), 
\[
\textup{L.H.S. of (\ref{eqn280}) }\lesssim \| m\|_{L^\infty_{\xi, \eta}}^2 \sum_{\begin{subarray}{c}
\omega_1,\omega_2,\omega_3\\
\textup{same as in (\ref{noveeqn141})}\\
\end{subarray}}  2^{2k + l} \| \widehat{f}(\xi-\eta) b^{\omega_2}_{l+k_2-k_1}(\xi-\eta) \psi_{k_1}(\xi-\eta) \|_{L^2}^2  \]
\[
\times \|b^{\omega_3}_{l}(\eta)\psi_{k_2}(\eta)\widehat{g}(\eta)\|_{L^2_{\eta}}^2\lesssim \| m\|_{L^\infty_{\xi, \eta}}^2 \min\big\{  \sum_{\begin{subarray}{c}
\omega_1,\omega_2,\omega_3\\
\textup{same as in (\ref{noveeqn141})}\\
\end{subarray}} 2^{2k+2l+ k_1+k_2} \| \widehat{f}\|_{L^\infty_{\xi}}^2 \|b^{\omega_3}_{l}(\eta)\psi_{k_2}(\eta)\widehat{g}(\eta)\|_{L^2_{\eta}}^2 , \]
\[
 2^{2k+l}\| f\|_{L^2}^2
\|g\|_{L^2}^2, \sum_{\begin{subarray}{c}
\omega_1,\omega_2,\omega_3\\
\textup{same as in (\ref{noveeqn141})}\\
\end{subarray}} 2^{2k+2k_2+2l}\| \widehat{f}(\xi-\eta) b^{\omega_2}_{l+k_2-k_1}(\xi-\eta) \psi_{k_1}(\xi-\eta) \|_{L^2}^2\|\widehat{g}\|_{L^\infty_\xi}^2\big\}\]
\[
\lesssim \| m\|_{L^\infty_{\xi, \eta}}^2  \min\{2^{2k+2k_1+2l}\| \widehat{f}\|_{L^\infty_{\xi}}^2\|g\|_{L^2}^2, 2^{2k+l}\|f\|_{L^2}^2\|g\|_{L^2}^2, 2^{2k_2+2k+2l} \| \widehat{g}\|_{L^\infty_\xi}^2\|f\|_{L^2}^2 \}.
\]
Hence finishing  the proof of  (\ref{eqn298}).

$\bullet$\quad Now we proceed to prove the desired estimate (\ref{equation6884}). From the $L^2-L^\infty$ type bilinear estimate (\ref{bilinearesetimate}) in Lemma \ref{multilinearestimate}, the following estimate holds,
\[
\textup{(\ref{equation6850})} \lesssim \| m\|_{\mathcal{S}^\infty_{k,k_1,k_2}}^2 \min\big\{  \sum_{\begin{subarray}{c}
\omega_1,\omega_2,\omega_3\\
\textup{same as in (\ref{noveeqn141})}\\
\end{subarray}}\| \widehat{f}(\xi- \eta) b^{\omega_2}_{l+k_2-k_1}(\xi-\eta) \psi_{k_1}(\xi-\eta) \|_{L^2}^2
\]
\[
\times  \|\mathcal{F}^{-1}[e^{-it \Lambda(\eta)} \widehat{g}(\eta) b^{\omega_3}_{l}(\eta)\psi_{k_2}(\eta)]\|_{L^\infty_{x}}^2,  \sum_{\begin{subarray}{c}
\omega_1,\omega_2,\omega_3\\
\textup{same as in (\ref{noveeqn141})}\\
\end{subarray}} \| \widehat{g}(\eta) b^{\omega_3}_{l}(\eta)\psi_{k_2}(\eta)\|_{L^2}^2
\]
\[
 \times \| \mathcal{F}^{-1}[ e^{-it \Lambda(\xi-\eta)} \widehat{f}(\xi- \eta) b^{\omega_2}_{l+k_2-k_1}(\xi-\eta) \psi_{k_1}(\xi-\eta) \|_{L^\infty_{x}}^2\big\}
\]
\[
\lesssim \| m\|_{\mathcal{S}^\infty_{k,k_1,k_2}}^2\min\big\{ \|f_{k_1}\|_{L^2}^2 \|\mathcal{F}^{-1}[e^{-it \Lambda(\eta)} \widehat{g_{k_2}}(\eta)]\|_{L^\infty_{x}}^2, 2^{(k_1-k_2)} \| {g_{k_2}}\|_{L^2}^2 \|\mathcal{F}^{-1}[e^{-it \Lambda(\eta)} \widehat{f_{k_1}}(\eta)]\|_{L^\infty_{x}}^2\big\}.
\]
Hence finishing the proof of (\ref{equation6884}). In the  above estimate, we used the facts that the multiplicity of summation with respect to $\omega_3$ is $2^{k_1-k_2}$ and the kernel of symbol $b^\omega_{l}(\xi)\psi_k(\xi)$ belongs $L^1$, where $l, k\in \mathbb{Z}.$
\end{proof}

To take   advantage of the oscillation in time when the frequencies are localized away from the time  resonance set, we do integration by parts in time. Unavoidably, we will confront a bilinear operator with the small divisor issue because of the presence of the degenerated phase in the denominator. Hence, we provide the $L^2$-estimates of  bilinear operators of this type  in the following Lemma. 

\begin{lemma}\label{angularbilinear}
For $m,k,k_1,k_2,\kappa,l \in \mathbb{Z}$,  $k_2\leq k_1$,  $t\in[2^{m-1}, 2^{m} ]$, $\kappa \geq -m +\delta m $, and $f,g\in L^2\cap L^1$, we define two bilinear operators as follow,  
\[
T_1(f,g) =  \int_{\R^2}  e^{it \Phi^{\mu, \nu}(\xi, \eta)} \frac{\psi_{[-10, 10]}(2^{-\kappa}\Phi^{\mu, \nu}(\xi, \eta))}{\Phi^{\mu, \nu}(\xi, \eta)} \psi_{k}(\xi)  m(\xi, \eta) \widehat{f^{\mu}_{k_1}}(\xi-\eta) \widehat{g^{\nu}_{k_2}}(\eta) d \eta,
\]
\[
T_2(f,g) =  \int_{\R^2}  e^{it \Phi^{\mu, \nu}(\xi, \eta)} \frac{\psi_{[-10, 10]}(2^{-\kappa}\Phi^{\mu, \nu}(\xi, \eta))}{\Phi^{\mu, \nu}(\xi, \eta)} \psi_{k}(\xi)  m(\xi, \eta) \widehat{f^{\mu}_{k_1}}(\xi-\eta) \widehat{g^{\nu}_{k_2}}(\eta) \psi_l(\angle(\xi, \nu\eta)) d \eta,
\]
where the phase $\Phi^{\mu, \nu}(\xi,\eta)$  is defined in \textup{(\ref{definitionofphase})}. Then the following estimates hold,
\[
\| T_1(f,g)\|_{L^2}\lesssim_{} 2^{-\kappa} \| m\|_{\mathcal{S}^\infty_{k,k_1,k_2}} \sup_{|\lambda| \leq 2^{\delta m/2}} \min\{\| e^{-i (t+\lambda2^{-\kappa})\Lambda} f_{k_1}\|_{L^\infty} \| g_{k_2}\|_{L^2} , 
\]
\begin{equation}\label{equation6910}
\| e^{-i (t+\lambda2^{-\kappa})\Lambda} g_{k_2}\|_{L^\infty} \| f_{k_1}\|_{L^2}\} + 2^{-10m-\kappa+k} \| m\|_{L^\infty_{\xi, \eta}} \| f_{k_1}\|_{L^2} \| g_{k_2}\|_{L^2},
\end{equation}
\[
\| T_2(f,g)\|_{L^2}\lesssim 2^{-\kappa} \| m\|_{\mathcal{S}^\infty_{k,k_1,k_2}}\sup_{|\lambda| \leq 2^{\delta m/2}} \min\{2^{(k_1-k_2)/2}\| e^{-i (t+\lambda2^{-\kappa})\Lambda} f_{k_1}\|_{L^\infty} \| g_{k_2}\|_{L^2} ,\]
\begin{equation}\label{equation6911}
\| e^{-i (t+\lambda2^{-\kappa})\Lambda} g_{k_2}\|_{L^\infty} \| f_{k_1}\|_{L^2}\} + 2^{-10m-\kappa+k} \| m\|_{L^\infty_{\xi, \eta}} \| f_{k_1}\|_{L^2} \| g_{k_2}\|_{L^2}.
\end{equation}
\end{lemma}
\begin{proof}
To prove (\ref{equation6910}), we use the inverse    Fourier transform to reformulate $T_1(f,g)$ as follows, 
\[
T_1(f,g) = \frac{1}{4\pi^2} \int_{\R}  \int_{\R^2} 2^{-\kappa} e^{i (t+ \lambda 2^{-\kappa}) \Phi^{\mu, \nu}(\xi, \eta)} \widehat{\chi}(\lambda)\psi_{k}(\xi)  m(\xi, \eta) \widehat{f}^{\mu}_{k_1}(\xi-\eta) \widehat{g}^{\nu}_{k_2}(\eta) d \eta d \lambda,
\]
where
\begin{equation}\label{equation6912}
\widehat{\chi}(\lambda)= \int_{\R} e^{-i \lambda x} \frac{\psi_{[-10, 10]}(x)}{x} d x, \Longrightarrow |\widehat{\chi}(\lambda)| \lesssim (1+|\lambda|)^{-N_0/\delta}.
\end{equation}
Hence, when $|\lambda|\leq 2^{\delta m/2}$, we use the $L^2-L^\infty$ type bilinear estimate (\ref{bilinearesetimate}) in Lemma \ref{multilinearestimate}, which gives the first part of estimate (\ref{equation6910}). When $|\lambda|\geq 2^{\delta m/2} $, from (\ref{equation6912}), $\widehat{\chi}(\lambda)$ provides fast decay. After using the size of support of $\xi$ first and then use the $L^2-L^2$ type estimate, we derive the second part of estimate (\ref{equation6910}).

With minor modifications in the above estimates and the proof of (\ref{equation6884}), we can prove the desired estimate (\ref{equation6911}) very similarly. Hence, we omit details here. 
\end{proof}

\section{The Set-up of  the $Z$-norm Estimate}\label{improvedZnorm}

Note that our goal is reduced to prove the Proposition \ref{proposition2}. In other words, we will prove that the $Z$-norm of the profile ``$e^{it\Lambda}(h+ i \Lambda\psi)=e^{it\Lambda} u $ '' doesn't grow over time.  In this section, we first introduce the set-up of the $Z$-norm estimate and  then reduce the proof of Proposition \ref{proposition2} into the proof of three Propositions.

\subsection{The first reduction}

Recall the equation satisfied by $u$ in (\ref{complexversion}) and (\ref{noveqn611}), we define the   profile of $u$ as $f(t):= e^{i t \Lambda} u(t)$ and then rewrite the equation (\ref{complexversion}) in terms of profile $f(t)$ as follows,
\begin{equation}\label{eqn1400}
\p_t f = \sum_{\mu, \nu\in
\{+,-\}}\sum_{k\in \mathbb{Z}} \sum_{(k_1,k_2)\in \chi_k^1\cup \chi_k^2\cup\chi_k^3} T^{\mu, \nu}(f^{\mu}_{k_1}, f^{\nu}_{k_2}) + \mathcal{R}',
\end{equation}
where $\mathcal{R}'= e^{it \Lambda} \mathcal{R}$ and the bilinear operator $T^{\mu, \nu}(\cdot, \cdot)$  is defined as follows,
\[
T^{\mu, \nu}(g, h):=\mathcal{F}^{-1}\Big[ \int_{\R^2}  e^{it \Phi^{\mu, \nu}(\xi, \eta)}\widehat{g}(\xi-\eta)\widehat{h}(\eta) q_{\mu, \nu}(\xi-\eta, \eta) d \eta \Big], 
\]
where ``$g$'' and ``$h$'' are two localized $L^2$ functions.

Because of the presence of the space resonance but not time resonance set, which is a small neighborhood of $(\underbrace{\xi}_{\text{output frequency}}, \underbrace{\xi/2}_{\text{input frequency}},\underbrace{\xi/2}_{\text{input frequency}})$,  instead of estimating the $Z$-norm of the profile $f(t)$ directly, we will estimate a good substitution variable instead. 

We first identify this good substitution variable by utilizing the normal form transformation. More precisely, we define 
\begin{equation}\label{equation000010}
v(t):= u(t) + \sum_{\mu, \nu\in\{+,-\}}\sum_{k \in \mathbb{Z}, (k_1,k_2)\in \chi_k^3 } A_{\mu,\nu} (u^{\mu}_{k_1}(t),u^{\nu}_{k_2}(t)),
\ee
where the symbol $m_{\mu,\nu}(  \cdot, \cdot)$ of $A_{\mu,\nu} (\cdot, \cdot)$  is defined as follows, 
\be\label{neweqn90}
m_{\mu,\nu}( \xi-\eta, \eta) := - \frac{q_{\mu, \nu}(\xi-\eta, \eta)}{i \Phi^{\mu, \nu}(\xi, \eta)}.
\ee
From the estimate (\ref{sizeofsymboluniform}) in Lemma \ref{sizeofsymbol}  and the estimate (\ref{noveqn519}) in Lemma \ref{roughestimatephase2}, the following estimate holds, 
\be\label{sizeofnormalform}
\|m_{\mu,\nu}( \xi-\eta, \eta)\psi_k(\xi) \psi_{k_1}(\xi-\eta)\psi_{k_2}(\eta)\|_{L^\infty_{\xi, \eta}}\lesssim 2^{k-\min\{k,k_2\}-2k_{1,-}+3k_{1,+}/2}.
\ee
Since the phases are alway bounded from below (see the estimate (\ref{noveqn519}) in Lemma \ref{roughestimatephase2}), for simplicity, instead of remove a very small neighborhood of $(\xi, \xi/2, \xi/2)$, we removed the case when the output frequnecy and the two inputs frequencies are all comparable, i.e., the case when $(k_1,k_2)\in\chi_k^3$, in (\ref{equation000010}).

For the $Z$-norm estimate of the normal form transformation, we have the following Lemma. 
 \begin{lemma}\label{L2Znormestimate}
Under the bootstrap assumption  \textup{(\ref{smallness})}, the following estimate holds,
\begin{equation}\label{equation24334}
\sup_{t\in[0,T ]} \big\|  \sum_{\mu, \nu\in\{+,-\}}\sum_{k\in \mathbb{Z}, (k_1,k_2)\in \chi_k^3 }  e^{i t\Lambda} A_{\mu,\nu} (u^{\mu}_{k_1}(t),u^{\nu}_{k_2}(t))\big\|_{Z} \lesssim \epsilon_0.
\end{equation}
\end{lemma}
\begin{proof}
Postponed to  subsection \ref{znormnormalform}.
\end{proof}

Define the profile of the good substitution variable ``$v(t)$'' as  $g(t):= e^{ it \Lambda} v(t)$. Recall (\ref{equation000010}). From the estimate (\ref{equation24334}) in    Lemma \ref{L2Znormestimate}, it is easy to see that the $Z$-norm of $f(t)$ and $g(t)$ are comparable, Hence,  it would be sufficient to prove  the  following estimate to close the argument,
\begin{equation}\label{desiredZnorm}
\sup_{t_1, t_2\in[2^{m-1}, 2^{m+1}]} \| g(t_2)  - g(t_1)\|_{Z} \lesssim  2^{- \delta m} \epsilon_0,\quad [2^{m-1}, 2^{m+1}]\subset[0,T], \quad m \gg 1. 
\end{equation}
In the rest of this paper,  time ``$t$" will be naturally restricted inside the time interval $[2^{m-1}, 2^{m+1}]$, where ``$m$" is a fixed and sufficiently large number.

From (\ref{eqn1400}) and (\ref{equation000010}), we can derive the equation satisfied by the profile $g(t)$ as follows, 
\[
\p_t g_k(t)= \sum_{\mu, \nu \in\{+,-\} }  \sum_{(k_1,k_2)\in \chi_k^1\cup \chi_k^2} T^{\mu, \nu} (f_{k_1}, f_{k_2}) + P_{k}[\mathcal{R}'] + \sum_{\mu, \nu\in\{+,-\}}\sum_{(k_1, k_2)\in \chi_k^3} \mathcal{F}^{-1}\big[\int_{\R^2} e^{  i t \Phi^{\mu,\nu}(\xi,\eta)} 
\]
\be\label{noveqn251}
 \times  
   {m_{\mu, \nu}(\xi-\eta, \eta)}   \p_t \big(\widehat{f^{\mu}_{k_1}}(t,\xi-\eta) \widehat{f^{\nu}_{k_2}}(t,\eta) \big)     d \eta\big].
 \ee

\subsection{The second reduction} In this subsection, based on the properties of  the associated phases,  we classify the quadratic terms in (\ref{noveqn251}) into two types: good type and bad type. Moreover, we reduce the proof of the desired estimate (\ref{desiredZnorm}) into the proof of three propositions.

 \begin{definition}\label{goodphase}
 We call the  phase $\Phi^{\mu, \nu}(\xi, \eta)$  a \textit{good phase} if and only if  
 \be\label{goodtypephase}
(k_1,k_2, \mu, \nu)\in \mathcal{P}_{\textup{good}}^k = \chi_k^1 \times \{(-,-),(+,+)\} \cup \chi_k^2\times \{(-,-),(-,+)\}.
 \ee
\end{definition}
 Recall that the phase $\Phi^{\mu,\nu}(\xi, \eta)$ is defined as follows,
\[
\Phi^{\mu, \nu}(\xi, \eta)=\Lambda(|\xi|)-\mu \Lambda(|\xi-\eta|) -\nu \Lambda(|\eta|),\, \quad \mu, \nu\in\{+,-\}.
\]
It is easy to verify that the following estimate holds,
\begin{equation}\label{hhchi1}
|\Phi^{\mu,\nu}(\xi, \eta)| \sim 2^{k_1/2+k_{1,-}/2}, \quad \textup{if\,\,} |\xi|\sim 2^k,\quad |\xi-\eta|\sim 2^{k_1}, \quad |\eta|\sim 2^{k_2},\,\, (k_1, k_2, \mu, \nu)\in \mathcal{P}_{\textup{good}}^k.
\end{equation}
 
From (\ref{hhchi1}), it is easy to see that the sizes of phases  are not highly degenerated, which is why we refer the phases in the scenarios mentioned above  as good type phases.
 \begin{definition}\label{badphase}
 We call the  phase $\Phi^{\mu, \nu}(\xi, \eta)$  a \textit{bad phase} if and only if  
 \be\label{badtypephase}
(k_1,k_2, \mu, \nu)\in \mathcal{P}_{\textup{bad}}^k= \chi_k^1 \times \{(+,-),(-,+)\} \cup \chi_k^2\times \{(+,-),(+,+)\}.
 \ee
\end{definition}
We refer the phases in the  scenarios   mentioned above as bad type phases because the associated phases are of cubic level smallness  in the worst scenario,   see  the estimate (\ref{dece31}) in Lemma \ref{roughestimatephase2}.

Recall (\ref{noveqn251}). We can rewrite the equation satisfied by the frequency localized profile $g(t) $ as follows, 
\begin{equation}\label{equation6410}
\p_t g_k(t)= \textup{good}_k(t) +\textup{bad}_k(t)+ P_{k}[\mathcal{R}'],\end{equation}
where
\begin{equation}\label{goodtypeterms}
\textup{good}_{k}(t) = \sum_{(k_1,k_2,\mu, \nu)\in \mathcal{P}_{\textup{good}}^k} T^{\mu, \nu}(f^{\mu}_{k_1},f^{\nu}_{k_2}),   
\end{equation}
\begin{equation}\label{badtypeterms}
\textup{bad}_{k}(t) = \sum_{(k_1,k_2,\mu, \nu)\in \mathcal{P}_{\textup{bad}}^k} T^{\mu, \nu}(f^{\mu}_{k_1},f^{\nu}_{k_2}) +  \sum_{(k_1,k_2)\in \chi_{k}^3}\sum_{\mu, \nu \in\{+,-\}} \mathcal{F}^{-1}\big[ K^{\mu, \nu}
(f^{\mu}_{k_1}, {f}^{\nu}_{k_2})\big], 
\end{equation}
where the bilinear operator   $K^{\mu,\nu}(\cdot, \cdot)$ is defined as follows, 
 \be\label{neweqn95}
K^{\mu,\nu}(f_{k_1}^{\mu},f_{k_2}^{\nu})=    \int e^{i t\Phi^{\mu,\nu}(\xi, \eta)}  m_{\mu,\nu}(\xi-\eta, \eta) \p_t\big(\widehat{f_{k_1 }^{\mu}}(t,\xi-\eta) \widehat{f_{k_2}^{\nu}}(t,\eta)\big)    d \eta.
\ee
Hence, from (\ref{equation6410}),  the following identity holds,
\[
P_{k}g (t_2)- P_{k}g (t_1) =\int_{t_1}^{t_2} \textup{good}_k(t) +\textup{bad}_k(t)+ P_k[\mathcal{R}'(t)] d t. \]
Hence to prove the desired estimate (\ref{desiredZnorm}), recall the definition of $Z$-norm in (\ref{definitionofZnorm}),  it would  be sufficient if we can prove the following three Propositions,
\begin{proposition}\label{propZnorm2}
Under the bootstrap assumption \textup{(\ref{smallness})}, for any $\theta\in[0,1]$,  the following estimates hold for the remainder term $\mathcal{R}'$:
\begin{equation}
\sup_{t_1, t_2\in[2^{m-1, m+1}]}\sup_{k\in \mathbb{Z}, j\geq -k_{-} } 2^{\delta  j} \| \int_{t_1}^{t_2} P_k[\mathcal{R}'(t)] d t\|_{B_{k,j}}   \lesssim 2^{- \delta m}\epsilon_0,
\end{equation}
\begin{equation}\label{equation7760}
\sup_{t\in[2^{m-1, m+1}]}  \| \mathcal{R}'(t)\|_{Z}  +2^{-(1-\theta) k + \theta m}\| P_k\big(\mathcal{R}'(t)\big)\|_{L^2} + \| \widehat{\mathcal{R}'}(t, \xi) \|_{L^\infty_\xi}   \lesssim 2^{-m} \epsilon_0.
\end{equation}
\end{proposition}
\begin{proof}
Postponed to section \ref{remainderZnorm}.
\end{proof}
\begin{proposition}\label{propZnorm1} 
Under the bootstrap assumption \textup{(\ref{smallness})}  and the assumption that Proposition \textup{\ref{propZnorm2}} holds, 
the following $Z$-norm estimate holds for any $t_1, t_2\in[2^{m-1}, 2^{m+1}]$, $m\in\mathbb{Z}_{+}$,
 \begin{equation}\label{goodtypeZ}
\sup_{k\in \mathbb{Z}, j\geq -k_{-}} 2^{\delta  j} \| \int_{t_1}^{t_2}  \textup{good}_k(t)  d  t\|_{B_{k,j}} \lesssim 2^{-\delta m}\epsilon_0.
 \end{equation}
\end{proposition}
\begin{proof}
Postponed to section \ref{goodimprovedZnorm}. 
\end{proof}
\begin{proposition}\label{propZnorm4bad}
Under the bootstrap assumption \textup{(\ref{smallness})}  and the assumption that Proposition \textup{\ref{propZnorm2}} holds, 
the following $Z$-norm estimate holds for any $t_1, t_2\in[2^{m-1}, 2^{m+1}]$, $m\in\mathbb{Z}_{+}$,
\begin{equation}\label{badtypeZ}
\sup_{k\in \mathbb{Z}, j\geq -k_{-}} 2^{\delta  j} \| \int_{t_1}^{t_2}  \textup{bad}_k(t)  d  t\|_{B_{k,j}} \lesssim 2^{-\delta m}\epsilon_0.
\end{equation}
\end{proposition}
\begin{proof}
Postponed to section \ref{badimprovedZnorm}. 
\end{proof}

\subsection{The size of profile under the bootstrap assumption}
In this subsection, we estimate the size of the profile in different function spaces under the bootstrap assumption.  These estimates will give us a good sense of what the profile $f(t)$  looks like with respect to the localized frequencies  over time.

From the definition of $Z$-norm,  the bootstrap assumption  (\ref{smallness}), the improved energy estimate (\ref{energyestimate}),  and the linear decay estimates  in Lemma \ref{lineardecay}, we have the following estimates,  
\be\label{L2}
\| P_{k} f(t)\|_{L^2}\lesssim 2^{-N_0 k_{+} +\delta m }\epsilon_0,  
\ee
\be\label{lowfrequencyestimateZ}
\| P_{k} f(t)\|_{L^2}\lesssim 2^{(1-\alpha)k-6k_{+}}\epsilon_1, \quad \| \widehat{f_{k}}(\xi)\|_{L^\infty_\xi} \lesssim 2^{-\alpha k - 6 k_+}   \epsilon_1, 
\ee
\begin{equation}\label{Linfty}
  \| e^{ -i t \Lambda} P_{k} f \|_{L^\infty} \lesssim 2^{-\alpha k -m-4.5k_{+}} \epsilon_1,
\end{equation}
where  estimate (\ref{L2}) is derived from the energy estimate,   estimate (\ref{lowfrequencyestimateZ}) is derived from the $Z$-norm estimate, estimate (\ref{Linfty}) is derived from the linear decay estimates (\ref{highdecay}) and (\ref{lowdecay}) in Lemma  \ref{lineardecay}.

Note that the $L^\infty$- estimate in (\ref{Linfty}) is not sharp when $k$ is sufficiently small. Alternatively, after choosing  $\theta = 1- \alpha $ in the estimate  (\ref{lowdecay}) in Lemma \ref{lineardecay}, the following estimate holds, 
\[
\| e^{i t \Lambda }P_{k} f \|_{L^\infty} \lesssim 2^{-m+\alpha m/2} 2^{3\alpha k/2} \| f_k\|_{L^1}\lesssim 2^{-m+\alpha m/2+\alpha k/2  }\epsilon_1,\quad \textup{if $k\leq 0$}. \]
To sum up, we have the following linear decay estimate at the low-frequency part,
\begin{equation}\label{Linfinity3}
 \| e^{ -i t \Lambda} P_{k} f \|_{L^\infty}\lesssim  \min\{2^{-m-\alpha k }, 2^{- m+\alpha m/2+\alpha k/2}\}\epsilon_1,\quad \textup{if $k\leq 0$}. 
\end{equation}

Note that the $L^\infty_\xi$-estimate and the $L^2$-estimate of the profile in (\ref{lowfrequencyestimateZ}) is derived directly from the size of $Z$-norm. When ``$k$'' is extremely small, the upper bound provided by the $Z$-norm is not sharp. It turns out that, under the bootstrap assumption (\ref{smallness}),  the $L^\infty_\xi$ estimate of the profile grows at most at rate $(1+t)^{3\delta}$. More precisely, we summarize those improved estimates in the following Lemma. 
\begin{lemma}\label{L2estimatelemma}
Under the bootstrap assumption \textup{(\ref{smallness})}  and the assumption that Proposition \textup{\ref{propZnorm2}} holds, the following estimates hold, 
\begin{equation}\label{L2estimate}
\sup_{t\in[0,T]} \sup_{k\in \mathbb{Z}} 2^{-k }  (1+t)^{ -\delta} \|f_k(t,x)\|_{L^2}    \lesssim   \epsilon_0 ,\quad 
\end{equation}
\begin{equation}\label{L2derivativeestimate}
\sup_{t\in[2^{m-1},2^{m+1}]} \sup_{k\in \mathbb{Z}}2^{-k} \| \p_t f_k(t,x)\|_{L^2}\lesssim 2^{-m}\epsilon_0,
\end{equation}
\begin{equation}\label{equation6726}
\sup_{t\in[2^{m-1}, 2^{m+1}]} \sup_{k\leq 0}\| \widehat{f_k}(t,\xi)\|_{L^\infty_\xi}\lesssim 2^{3\delta m }\epsilon_0.
\end{equation}
\end{lemma}
\begin{proof}
Recall (\ref{eqn1400}).  From the $L^2-L^\infty$ type bilinear estimate (\ref{bilinearesetimate}) in Lemma \ref{multilinearestimate},   (\ref{sizeofsymboluniform}) in Lemma \ref{sizeofsymbol}, and  (\ref{equation7760}) in Proposition \ref{propZnorm2}, the following estimate holds, 
\[
\sup_{t\in[2^{m-1},2^{m+1}]} \sup_{k\in \mathbb{Z}}2^{-k} \| \p_t f_k(t,x)\|_{L^2} \lesssim \sup_{t\in[2^{m-1},2^{m+1}]} 2^{-k} \| P_k[\mathcal{R}'](t)\|_{L^2} + \sum_{k_2\leq k_1} 2^{-k+k+k_{1,+}} \| f_{k_2}(t) \|_{L^2}
\]
\begin{equation}\label{equation7090}
 \times \| e^{-it \Lambda} f_{k_1}(t)\|_{L^\infty} \lesssim 2^{-m}\epsilon_0 + \sum_{k_2\leq k_1} 2^{-m+(1-2\alpha)k_2-3k_{1,+}}\epsilon_1^2 \lesssim 2^{-m}\epsilon_0 .
 \end{equation}
 Recall that the  $L^\infty_\xi$-norm of the initial data  is of size $\epsilon_0$, see (\ref{initialcondition}). Therefore, the following estimate holds from the volume of support of the frequency variable ``$\xi$'', 
\be\label{newequation310}
\sup_{k\in \mathbb{Z}} 2^{-k}\| f_k(0,x)\|_{L^2} \lesssim \|\widehat{f}(t,\xi)\psi_k(\xi)\|_{L^\infty_\xi}\lesssim \epsilon_0.
\ee
Combining (\ref{newequation310}) and (\ref{equation7090}), it is easy to see that our desired estimate (\ref{L2estimate}) holds. 

Recall again  (\ref{eqn1400}). From  the   estimate   (\ref{equation7760}) in Proposition \ref{propZnorm2}, the following estimate holds straightforwardly for any $k\in \mathbb{Z}, k\leq 0$, and $t\in[2^{m-1}, 2^m]$,
\be\label{noveqn401}
\| \widehat{f}_k(t, \xi)\|_{L^\infty_\xi} \lesssim 2^{\delta m }\epsilon_0 + \sum_{
\begin{subarray}{c}
k_1,k_2\in \mathbb{Z},
\mu, \nu \in\{+,-\}\\
k_2\leq k_1\\
\end{subarray}
}   \| \int_0^t e^{i s \Phi^{\mu, \nu}(\xi, \eta)} q_{\mu, \nu}(\xi-\eta, \eta)\widehat{f_{k_1}^\mu}(s,\xi-\eta) \widehat{f_{k_2}^{\nu}}(s,\eta)\psi_k(\xi) d \eta d s \|_{L^\infty_\xi}.
\ee
From the estimate  (\ref{sizeofsymboluniform})  in Lemma \ref{sizeofsymbol} and the estimate (\ref{L2estimate}), the following estimate holds when $k_1, k_2\notin [-2m, 2\beta m ]$, 
\[
\sum_{
\begin{subarray}{c}
 k_1, k_2\notin [-2m, 2\beta m ]\\
\mu, \nu \in\{+,-\}, k_1,k_2\in \mathbb{Z},k_2\leq k_1\\
\end{subarray}
}   \| \int_0^t e^{i s \Phi^{\mu, \nu}(\xi, \eta)} q_{\mu, \nu}(\xi-\eta, \eta)\widehat{f_{k_1}^\mu}(s,\xi-\eta) \widehat{f_{k_2}^{\nu}}(s,\eta)\psi_k(\xi) d \eta d s \|_{L^\infty_\xi}
\]
\be\label{noveqn310}
\lesssim \sum_{k_2\leq -2m, \textup{\, or\,} k_1\geq 2\beta m } 2^{m +k+k_{1,+}} \sup_{0\leq t\leq 2^m} \| f_{k_1}(t)\|_{L^2}\| f_{k_2}(t)\|_{L^2} \lesssim 2^{-m+2\delta m}\epsilon_1^2.
\ee

When $k_1, k_2\in[-2m, 2\beta m ]$, we do integration by parts in time once. As a result, the following estimate holds from the estimate (\ref{sizeofnormalform}),
\[
\sum_{k_1, k_2\in[-2m, 2\beta m ], k_2\leq k_1}\| \int_0^t e^{i s \Phi^{\mu, \nu}(\xi, \eta)} q_{\mu, \nu}(\xi-\eta, \eta)\widehat{f_{k_1}^\mu}(s,\xi-\eta) \widehat{f_{k_2}^{\nu}}(s,\eta)\psi_k(\xi) d \eta d s \|_{L^\infty_\xi}
\]
\[
\lesssim \sup_{0\leq t\leq 2^m}\sum_{k_1, k_2\in[-2m, 2\beta m ], k_1, k_2\in \mathbb{Z}, k_2\leq k_1} 2^{k-\min\{k,k_2\}-2k_{1,-}+3k_{1,+}/2}\big[  \| f_{k_1}(t)\|_{L^2}  \| f_{k_2}(t)\|_{L^2} 
\]
\be\label{noveqn412}
+\int_0^t \big( \| \p_t f_{k_1}(s)\|_{L^2} \| f_{k_2}(s)\|_{L^2} 
+  \| f_{k_1}( s)\|_{L^2} \| \p_t f_{k_2}(s)\|_{L^2}\big) d s\big]\lesssim m^2 2^{2\delta m}\epsilon_1^2 \lesssim 2^{3\delta m}\epsilon_0.
\ee
From the estimates (\ref{noveqn401}), (\ref{noveqn310}), and (\ref{noveqn412}), it is easy to see that our desired estimate (\ref{equation6726}) holds. 
\end{proof}

Therefore, \textit{under the bootstrap assumption} \textup{(\ref{smallness})} \textit{and the   assumption that Proposition} \ref{propZnorm2} \textit{holds}, the following estimates hold for the profile ``$f(t)$'' at the low-frequency part,
\be\label{sizeofprofilelowfrequency}
\| P_{k} f(t)\|_{L^2}\lesssim \min\{2^{(1-\alpha)k },2^{k+3\delta m }\}\epsilon_1,  \quad \| \widehat{f_{k}}(\xi)\|_{L^\infty_\xi} \lesssim \min\{2^{-\alpha k  }  ,2^{3\delta m }\} \epsilon_1, \quad \textup{when\,\,} k\leq 0.
\ee
\begin{remark}
Please note that the results in Lemma \ref{L2estimatelemma} cannot and will not be used in the proof of Proposition \ref{propZnorm2}, i.e., the estimate of the cubic and higher order terms, because the validity of Proposition \ref{propZnorm2} is part of the assumptions in Lemma \ref{L2estimatelemma}.
\end{remark}
\section{The Improved $Z$-norm Estimate: Good  Type Phases}\label{goodimprovedZnorm}

The main goal of this section is to prove the desired Proposition \ref{propZnorm1}. In other words, we will prove the desired estimate (\ref{goodtypeZ}) under the bootstrap assumption (\ref{smallness}) and the assumption that Proposition \ref{propZnorm2} holds. Note that the estimate (\ref{sizeofprofilelowfrequency}) is valid in this section. 

 Recall (\ref{goodtypeterms}) and (\ref{goodtypephase}). In   subsection \ref{highhighgood}, we consider the case $(k_1,k_2, \mu, \nu)\in \chi_k^1\times\{(+,+), (-,-)\}$.  In   subsection \ref{lowhighgood}, we consider the case  $(k_1,k_2, \mu, \nu)\in \chi_k^2\times$ $\{(-,+), (-,-)\}$. Hence finishing the proof.

\subsection{ When $(k_1, k_2, \mu, \nu)\in \chi_k^1\times \{(-,-), (+,+)\} $}\label{highhighgood}

For simplicity, we first rule out the relatively high-frequency case and the very-low-frequency case. More precisely, the following Lemma holds.

\begin{lemma}\label{ruleoutlowhigh}
For any fixed  $j,m\in \mathbb{Z}_{+}$, 
under the bootstrap assumption \textup{(\ref{smallness})}  and the assumption that Proposition \textup{\ref{propZnorm2}} holds, the following estimate  holds for any $k\in \mathbb{Z}$, $t_1, t_2\in [2^{m-1}, 2^m]$, 
\be\label{nove2}
\sum_{\begin{subarray}{c}
|k_1-k_2|\leq 10\\
k_1 \geq  (1+\delta)(j+m)/(N_0-8)\\
\end{subarray}
}\sum_{ \mu, \nu\in\{+,-\}} 2^{\delta j}  \|
\int_{t_1}^{t_2}T^{\mu, \nu} (f_{k_1}, f_{k_2})  d t \|_{B_{k,j}}\lesssim 2^{-\delta m }\epsilon_0.
\ee
Moreover, for any $k\in \mathbb{Z}$ s.t.,  $k\leq  - (1+10\delta)(m+j)/(2+\alpha) $,   the following estimate holds, 
\be\label{nove3}
\sum_{\begin{subarray}{c}
|k_1-k_2|\leq 10\\
\end{subarray}
}\sum_{ \mu, \nu\in\{+,-\}} 2^{\delta j}  \|
\int_{t_1}^{t_2}T^{\mu, \nu} (f_{k_1}, f_{k_2})  d t \|_{B_{k,j}}\lesssim 2^{-\delta m }\epsilon_0.
\ee
\end{lemma}
\begin{proof} 
For any $\mu, \nu\in\{+,-\}$, from the $L^2_x\rightarrow L^1_x$ type Sobolev embedding,  the  $L^2_x-L^2_x$-type bilinear estimate and the estimate (\ref{sizeofsymboluniform}) in Lemma \ref{sizeofsymbol}, the following estimate holds for any $(k_1,k_2)\in \chi_k^1$,
\[
 \|
\int_{t_1}^{t_2}T^{\mu, \nu} (f_{k_1}, f_{k_2})  d t \|_{B_{k,j}}  \lesssim   2^{\alpha k + 6k_+ + m+ j +k } \| q_{\mu, \nu}(\xi-\eta, \eta)\|_{\mathcal{S}^{\infty}_{k,k_1,k_2}}  \|  P_{k_1}f\|_{L^2}   \| P_{k_2} f\|_{L^2} \]
\be\label{nove1}
\lesssim  2^{(2+\alpha) k + 6 k_{+}+ m+  j +k_{1,+}  -2N_0 k_{1,+}+2\delta m } \epsilon_1^2. 
 \ee
 From the estimate (\ref{nove1}), it is easy to see that the desired estimates (\ref{nove2}) and (\ref{nove3}) hold. 
\end{proof}

To prove the desired estimate (\ref{goodtypeZ}),  from the estimates (\ref{nove2}) and (\ref{nove3}) in Lemma \ref{ruleoutlowhigh}, it is easy to see that  it would be  sufficient to prove the following estimate,
\be\label{reduceddesired1}
   2^{\delta j}  \|
\int_{t_1}^{t_2}T^{\mu, \nu} (f_{k_1}, f_{k_2})  d t \|_{B_{k,j}}\lesssim 2^{-2\delta m -2\delta j}\epsilon_0,\quad   (\mu, \nu)\in\{(-,-), (+,+)\},
\ee
where fixed $k$,  $k_1$, and $k_2$ satisfy the following estimate 
\begin{equation}\label{eqn40000}
 - (1+10\delta)(m+j)/(2+\alpha)  \leq k  \leq k_1-5\leq k_2\leq k_1 +5 \leq (1+\delta)(j+m)/(N_0-8)+10.
\end{equation}
Note that we used the fact that there are  at most  $(m+j)^2$ cases in total in (\ref{eqn40000}), which is only a logarithmic loss. 

 Based on the possible size of the fixed $j$, we separate   into two cases, see Lemma \ref{hhgoodj1} and Lemma \ref{hhgoodj2}.

\begin{lemma}\label{hhgoodj1}
Under the bootstrap assumption \textup{(\ref{smallness})}  and the assumption that Proposition \textup{\ref{propZnorm2}} holds, if $j\geq (1+20\delta)m$, then the desired estimate (\textup{\ref{reduceddesired1}}) holds  for any fixed $k$,  $k_1$, and $k_2$ that satisfy  \textup{(\ref{eqn40000})}.
\end{lemma}
\begin{proof}
Firstly, we do spatial localizations for two inputs. As a result,   the following decomposition holds, 
\be\label{noveq9}
\int_{t_1}^{t_2}T^{\mu, \nu} (f_{k_1}(t), f_{k_2}(t))  d t = \sum_{j_1\geq -k_{1,-}, j_2\geq -k_{2,-}} \int_{t_1}^{t_2}T^{\mu, \nu} (f_{k_1,j_1}(t), f_{k_2,j_2}(t))  d t.
\ee
  If  $\min\{j_1,j_2\}\geq j-\delta j-\delta m$, then the following estimate holds after using the $L^2_x-L^\infty_x$ type bilinear estimate, the estimate (\ref{sizeofsymboluniform}) in Lemma \ref{sizeofsymbol}, and the $L^\infty_x\rightarrow L^2_x$ type Sobolev embedding, 
\[
\sum_{\min\{j_1,j_2\}\geq j-\delta j-\delta m} 2^{\delta j}\| \int_{t_1}^{t_2}T^{\mu, \nu} (f_{k_1,j_1}(t), f_{k_2,j_2}(t))  d t\|_{B_{k,j}} \lesssim \sum_{\min\{j_1,j_2\}\geq j-\delta j-\delta m}2^{\alpha k + m + (1+\delta)j}
\]
\be\label{noveq7}
 \times 2^{k+(1-2\alpha)k_2 -j_1-j_2-3k_{1,+}}\| f_{k_1,j_1}\|_Z \| f_{k_2,j_2}\|_Z \lesssim 2^{m +2\delta m-(1-3\delta)j}\epsilon_0\lesssim 2^{-2\delta j -2\delta m}\epsilon_0,
\ee
where we used the assumption that $j\geq (1+20\delta)m.$

Now we consider the case 
$\min\{j_1, j_2\}\leq j-\delta j-\delta m$. Since $|k_1-k_2|\leq 5$, from the symmetry between inputs, we assume that $j_1=\min\{j_1,j_2\}$. Otherwise, we can simply do change of variables to switch the roles of $\xi-\eta$ and $\eta$.

  For this case,  we can do integration by parts in ``$\xi$''  to see rapidly decay. More precisely, after integration by parts in ``$\xi$'' once, we have
\[
\int_{t_1}^{t_2}T^{\mu, \nu} (f_{k_1,j_1}(t), f_{k_2,j_2}(t))  d t \]
\[=\int_{t_1}^{t_2} \int_{\R^2}\int_{\R^2} e^{i x\cdot \xi + i t \Phi^{\mu, \nu}(\xi, \eta)}  \nabla_{\xi}\cdot \big(\widehat{f^{\mu}_{k_1,j_1}}(t,\xi-\eta)\widehat{f^{\nu}_{ k_2,j_2}}(t,\eta)a_{\mu, \nu}(t,x,\xi, \eta)\big) d \eta  d \xi dt, 
\] 
where
\begin{equation}\label{equation6840}
a_{\mu, \nu}(t,x,\xi, \eta) = i q_{\mu, \nu}(\xi-\eta, \eta) \frac{x+ t \nabla_{\xi}\Phi^{\mu, \nu}(\xi, \eta)}{|x+ t \nabla_{\xi}\Phi^{\mu, \nu}(\xi, \eta)|^2}.
\end{equation}
Note that 
\begin{equation}\label{eqn3320}
|\nabla_{\xi}\Phi^{\mu, \nu}(\xi, \eta)|\lesssim 1\Longrightarrow \frac{x+ t \nabla_{\xi}\Phi^{\mu, \nu}(\xi, \eta)}{|x+ t \nabla_{\xi}\Phi^{\mu, \nu}(\xi, \eta)|^2}\varphi_{j}^{k}(x)\sim 2^{-j}.
\end{equation}
Hence, we can gain $2^{-j}$ if  doing integration by parts in ``$\xi$'' once. 

In the meantime, we need to find out what the maximal loss is. If $\nabla_{\xi}$ hits input $\widehat{f^{\mu}_{k_1,j_1}}(\cdot)$,  then we at most lose $2^{j_1}$. If $\nabla_{\xi}$ hits the  cutoff functions or the symbol $a_{\mu, \nu}(t,x,\xi, \eta)$,    it is easy to see that we at most  lose $\max\{2^{-k_1},2^{-k}, 1\}$.   Note that $j_1\geq -k_{1,-}$ and $k_1\geq k-10$.  Therefore,  the net gain of  doing integration by parts in ``$\xi$'' once is at least $\max\{2^{-j+j_1}, 2^{-j-k}\}\lesssim 2^{-\delta j-\delta m}$. Note that we used the fact that $j+k\geq \delta j+\delta m$, which can be derived from the estimate (\ref{eqn40000}) and the assumption that $j\geq (1+20\delta)m$.

We can do this process as many times as we want to see rapidly decay. As a result,  the following point-wise estimate holds after using  the estimate (\ref{sizeofsymboluniform}) in Lemma \ref{sizeofsymbol} and the $L^2-L^2$ type bilinear estimate, 
\begin{equation}\label{eqn3321}
\Big|\int_{t_1}^{t_2}T^{\mu, \nu} (f_{k_1,j_1}(t), f_{k_2,j_2}(t))  d t\Big| \varphi_{j}^{k}(x)\lesssim 2^{-10j}\| f_{k_1,j_1}\|_{L^2} \| f_{k_2,j_2}\|_{L^2},
\end{equation}
which further implies the following estimate, 
\be\label{noveq6}
\sum_{\min\{j_1,j_2\}\leq j-\delta j-\delta m}2^{\delta m}\| \int_{t_1}^{t_2}T^{\mu, \nu} (f_{k_1,j_1}(t), f_{k_2,j_2}(t))  d t\|_{B_{k,j}} \lesssim 2^{m+2j-10j}\epsilon_1^2\lesssim 2^{-2\delta m -2\delta j}\epsilon_0.
\ee
From (\ref{noveq9}), (\ref{noveq7}) and (\ref{noveq6}), it is easy to see that our desired estimate (\ref{reduceddesired1}) holds if $j\geq (1+20\delta)m$.
\end{proof}

\begin{lemma}\label{hhgoodj2}
Under the bootstrap assumption \textup{(\ref{smallness})}  and the assumption that Proposition \textup{\ref{propZnorm2}} holds, if $j\leq (1+20\delta)m$, then the desired estimate (\textup{\ref{reduceddesired1}}) holds  for any fixed $k$,  $k_1$, and $k_2$ that satisfy \textup{(\ref{eqn40000})}.
\end{lemma}
\begin{proof}
If $j\leq (1+20\delta )m$, from the estimate (\ref{eqn40000}), we have
\[
-(2+100\delta)m/(2+\alpha)\leq k   \leq k_1-10\leq k_1\leq   \beta m, \quad \beta:=1/980.
\]
Note that the following equality holds, 
\[
\mathcal{F}[\int_{t_1}^{t_2}T^{\mu, \nu} (f_{k_1}(t), f_{k_2}(t))  d t](\xi) = \int_{t_1}^{t_2} \int_{\R^2} e^{it \Phi^{\mu, \nu}(\xi, \eta)}  \widehat{ f^{\mu}_{k_1}}(t,\xi-\eta)\widehat{f^{\nu}_{k_2}}(t,\eta) q_{\mu, \nu}(\xi-\eta, \eta) d \eta d t.
\]
Recall (\ref{hhchi1}). To take   advantage of the high oscillation in time for the good type phases,  we  do integration by parts in time once for the above integral. As a result, we  have, 
\begin{equation}\label{eqn450}
\mathcal{F}[\int_{t_1}^{t_2}T^{\mu, \nu} (f_{k_1}(t), f_{k_2}(t))  d t](\xi)= \sum_{i=1,2} \textup{End}_{k_1, k_2}^{\mu, \nu, i}+  J^{\mu, \nu,i}_{k_1, k_2} ,
\end{equation}
where
\begin{equation}\label{eqn453}
\textup{End}_{k_1, k_2}^{\mu, \nu, i } = (-1)^{i-1}  \int_{\R^2} e^{i t_i \Phi^{\mu, \nu}(\xi, \eta)}\widehat{ f^{\mu}}(t_i,\xi-\eta)\widehat{f^{\nu}}(t_i,\eta) m_{\mu, \nu}(\xi-\eta, \eta) \psi_{k_1}(\xi-\eta)\psi_{k_2}(\eta ) d \eta,
\end{equation}
\begin{equation}\label{eqn451}
J^{\mu, \nu,1}_{k_1, k_2}= \int_{t_1}^{t_2} \int_{\R^2} e^{i t \Phi^{\mu, \nu}(\xi, \eta)}\widehat{f^{\mu}}(t,\xi-\eta)\widehat{\p_t f^{\nu}}(t,\eta) m_{\mu, \nu}(\xi-\eta, \eta) \psi_{k_1}(\xi-\eta)\psi_{k_2}(\eta ) d \eta d t, 
\end{equation}
\begin{equation}\label{eqn452}
J^{\mu, \nu,2}_{k_1, k_2}=  \int_{t_1}^{t_2} \int_{\R^2} e^{i t \Phi^{\mu, \nu}(\xi, \eta)}\widehat{\p_t f^{\mu}}(t,\xi-\eta)\widehat{f^{\nu}}(t,\eta) m_{\mu, \nu}(\xi-\eta, \eta) \psi_{k_1}(\xi-\eta)\psi_{k_2}(\eta ) d \eta d t,
\end{equation}
where the symbol $m_{\mu, \nu}(\cdot, \cdot)$ is defined in (\ref{neweqn90}).
 
From Lemma \ref{Snorm},  (\ref{sizeofsymboluniform})  in Lemma \ref{sizeofsymbol}, and (\ref{hhchi1}),  it is easy to check that the following estimate holds for any $(k_1,k_2)\in \chi_k^1\cup \chi_k^2$,
\begin{equation}\label{equation5667}
 \|m_{\mu, \nu}(\xi-\eta, \eta) \|_{\mathcal{S}^\infty_{k,k_1,k_2}}\lesssim 2^{k/2+k_{-}/2+3k_{1,+}-k_1}
\lesssim 2^{3\beta m}.
\end{equation}
 From the $L^2-L^\infty$ type bilinear estimate (\ref{bilinearesetimate}) in Lemma \ref{multilinearestimate}, (\ref{L2derivativeestimate}) in Lemma \ref{L2estimatelemma} and  (\ref{equation5667}),  the following estimate holds if $k_1\leq -4\beta m$  or $k\leq - 40\beta m$,
\[
\sum_{i=1,2}2^{\delta j}\| \mathcal{F}^{-1}\big[\textup{End}_{k_1, k_2}^{\mu, \nu, i}\big]\|_{B_{k,j}} + 2^{\delta j} \| \mathcal{F}^{-1}\big[J^{\mu, \nu,i}_{k_1, k_2}\big]\|_{B_{k,j}} \lesssim 2^{\alpha k +6k_{+} +j +3\beta m+2\delta m} \| f_{k_1} \|_{L^2} \| e^{-i t\Lambda} f_{k_2}\|_{L^\infty}\]
\[
+ 2^{\alpha k +6k_+ + j +m +3\beta m +2\delta m}\big( \|  \p_t f_{k_1} \|_{L^2} \| e^{-i t\Lambda} f_{k_2}\|_{L^\infty}  + \| e^{-it \Lambda } f_{k_1}\|_{L^\infty} \|\p_t  f_{k_2}\|_{L^2}\big) \]
\be\label{noveqn74}
\lesssim 2^{\alpha k +(1-2\alpha)k_1 + 3\beta m +2\delta m}\epsilon_1^2 \lesssim 2^{-2\delta m -2\delta j}\epsilon_0.
\ee
 From the estimate (\ref{noveq41}) in Lemma \ref{hhgoodj3} and the estimate (\ref{noveqn65}) in Lemma \ref{hhgoodj7}, we know that the desired estimate (\ref{noveqn74}) also holds when $k_1 \geq - 4\beta m$ and $k\geq -40\beta m $. Hence finishing the proof. 
\end{proof}
\begin{lemma}\label{hhgoodj3}
Under the bootstrap assumption \textup{(\ref{smallness})}  and the assumption that Proposition \textup{\ref{propZnorm2}} holds, if $j\leq (1+20\delta)m$, then the following estimate holds for fixed $k, k_1\in \mathbb{Z}, k\in[-40\beta m , \beta m], k_1\in [-4\beta m, \beta m]$, 
\be\label{noveq41}
\sum_{i=1,2}2^{\delta j}\| \mathcal{F}^{-1}\big[\textup{End}_{k_1, k_2}^{\mu, \nu, i }\big]\|_{B_{k,j}}\lesssim 2^{-2\delta m-2\delta j}\epsilon_0.
\ee  
\end{lemma}
\begin{proof}
Recall (\ref{eqn453}). After doing  spatial localizations for two inputs, we have 
\begin{equation}
\textup{End}_{k_1, k_2}^{\mu, \nu, i }= \sum_{j_1\geq -k_{1,-},j_2\geq -k_{2,-}}\textup{End}_{k_1,j_1, k_2,j_2}^{\mu, \nu, i },
\end{equation}
\[
\textup{End}_{k_1,j_1, k_2,j_2}^{\mu, \nu, i } = (-1)^{i-1} \int_{\R^2} e^{i t_i \Phi^{\mu, \nu}(\xi, \eta)}\widehat{ f^{\mu}_{k_1,j_1}}(t_i,\xi-\eta)\widehat{f^{\nu}_{k_2,j_2}}(t_i,\eta) m_{\mu, \nu}(\xi-\eta, \eta)  d \eta.
\]

Firstly, let's consider the case ``  $\max\{j_1,j_2\}\geq m + k +k_1-4\beta m $''. From the $L^2-L^\infty$ type bilinear estimate (\ref{bilinearesetimate}) in Lemma \ref{multilinearestimate} and the estimate  (\ref{equation5667}), we can put the input with larger spatial concentration in $L^2$ and the other input in $L^\infty$. As a result, the following estimate holds,
\[
\sum_{\max\{j_1,j_2\}\geq m + k +k_1-4 \beta m}  2^{\delta j}\| \mathcal{F}^{-1}\big[\textup{End}_{k_1,j_1, k_2,j_2}^{\mu, \nu, i }\big]\|_{B_{k,j}}\lesssim \sum_{\max\{j_1,j_2\}\geq m + k +k_1-4\beta m}  2^{\alpha k+ 6k_{+} +(1+\delta)j+4\beta m} 
\]
\be\label{noveq81}
 \times 2^{  -\max\{j_1,j_2\} -m-\alpha (k_1+k_2)}\|f_{k_1,j_1}\|_{Z} \|f_{k_2,j_2}\|_{Z} \lesssim  2^{-m-(1+\alpha)(k+k_1)+9\beta m}\epsilon_1^2 \lesssim 2^{-2\delta m -2\delta j}\epsilon_0.
\ee

Lastly, let's consider the case when $\max\{j_1,j_2\}\leq m+k+k_1-4\beta m$. For this case, we can do integration by parts in ``$\eta$'' many times  to see rapidly decay. More precisely, after integration by parts in $\eta$ once, we have the following identity,
\begin{equation}\label{equation6510}
\textup{End}_{k_1,j_1, k_2,j_2}^{\mu, \nu, i} = \frac{(-1)^{i-1}}{t_i}\int_{\R^2} e^{i t_i \Phi^{\mu, \nu}(\xi, \eta)} \nabla_{\eta}\cdot\big( \widehat{ f^{\mu}_{k_1,j_1}}(t_i,\xi-\eta)\widehat{f^{\nu}_{k_2,j_2}}(t_i,\eta)  \tilde{m}_{\mu, \nu}(\xi-\eta, \eta)\big) d \eta,
\end{equation}
\begin{equation}\label{integrationeta}
 \tilde{m}_{\mu, \nu}(\xi-\eta, \eta)  =- \frac{ {m}_{\mu, \nu}(\xi-\eta, \eta) \nabla_{\eta} \Phi^{\mu, \nu}(\xi, \eta)  }{i|\nabla_{\eta}\Phi^{\mu, \nu}(\xi, \eta)|^2}.
\end{equation}
If $\nabla_{\eta}$ hits  $\widehat{f}_{k_1,j_1}$ and $\widehat{f}_{k_2,j_2}$, we at most lose $2^{\max\{j_1,j_2\} }$. If $\nabla_{\eta}$ hits the symbol $\tilde{m}_{\mu, \nu}(\cdot, \cdot)$, then from the estimates (\ref{dece1}) and (\ref{dece2}) in Lemma \ref{phasesize}, it is easy to see that the maximal loss is  $2^{-3k_{ -}+k_{1,+}}$. Therefore,  the net gain of doing integration by parts in ``$\eta$'' once is at least $2^{-m}\max\{ 2^{\max\{j_1,j_2\}-k-k_1+5k_{1,+}/2},2^{-5k_{ -}+k_{1,+}} \}$, which is less than $2^{-\beta m}$. Therefore, after repeating   this process many times, it is easy to see that the following estimate holds,
 \[
\sum_{\max\{j_1,j_2\}\leq  m + k +k_1-4 \beta m}  2^{\delta j}\| \mathcal{F}^{-1}\big[\textup{End}_{k_1,j_1, k_2,j_2}^{\mu, \nu, i }\big]\|_{B_{k,j}} \lesssim \sum_{\max\{j_1,j_2\}\leq  m + k +k_1-4 \beta m} 2^{-10 m} \| f_{k_1,j_1}\|_{L^2}  \]
\be\label{noveq82}
\times \| f_{k_2,j_2}\|_{L^2}\lesssim 2^{-2\delta m -2\delta j}\epsilon_0.
 \ee
From the estimates (\ref{noveq81}) and (\ref{noveq82}), it is easy to see that our desired estimate (\ref{noveq41}) holds.  
 \end{proof}
 \begin{lemma}\label{hhgoodj7}
 Under the bootstrap assumption \textup{(\ref{smallness})}  and the assumption that Proposition \textup{\ref{propZnorm2}} holds, if $j\leq (1+20\delta)m$, then the following estimate holds for fixed $k, k_1\in \mathbb{Z}, k\in[-40\beta m , \beta m], k_1\in [-4\beta m, \beta m]$, 
\be\label{noveqn65}
\sum_{i=1,2}2^{\delta j}\| \mathcal{F}^{-1}\big[J^{\mu, \nu,i}_{k_1, k_2}\big]\|_{B_{k,j}}\lesssim 2^{-2\delta m-2\delta j}\epsilon_0.
\ee  
 \end{lemma}
 \begin{proof}
 Recall (\ref{eqn451}) and (\ref{eqn452}). 
   After plugging in the equation satisfied by $\p_t f$ in (\ref{eqn1400}) and doing dyadic decompositions for the quadratic terms of $\p_t f$, we have
\[
 J^{\mu, \nu, i}_{k_1,k_2}= \sum_{(k_1',k_2')\in \chi_{k_{3-i}}^1\cup \chi_{k_{3-i}}^2\cup \chi_{k_{3-i}}^3}  \sum_{\tau, \kappa\in \{+, -\}} J^{\mu, \nu,\tau, \kappa, i}_{k_1',k_2'} + \mbox{JR}^{i}_{k_1,k_2},\]
 \[
 J^{\mu, \nu,\tau, \kappa, i}_{k_1',k_2' }:= \sum_{j_1'\geq -k_{1,-}', j_2'\geq -k_{2,-}', j_i \geq -k_{i,-}} H_{j_i;j_1',j_2'}^{\mu,\nu,\tau, \kappa,i},
\]
\begin{equation}\label{tridecom}
H_{j_1',j_2'}^{\mu,\nu,\tau, \kappa,i}=\sum_{ j_i \geq -k_{i,-}} H_{j_i;j_1',j_2'}^{\mu,\nu,\tau, \kappa,i}, \quad H_{j_i;j_1'} ^{\mu,\nu,\tau, \kappa,i}=\sum_{ j_2' \geq -k_{2,-}'} H_{j_i;j_1',j_2'}^{\mu,\nu,\tau, \kappa,i},\quad i \in \{1,2\},
\end{equation}
where
\[
H_{j_i;j_1',j_2'}^{\mu,\nu,\tau, \kappa,1} =   \int_{t_1}^{t_2}\int_{\R^2} \int_{\R^2} e^{i t \Phi_1^{\mu, \tau, \kappa}(\xi, \eta, \sigma)} m^{ \tau, \kappa}_{\mu, \nu,1}(\xi, \eta, \sigma)\widehat{
f_{k_1,j_1}^{\mu}
}(t,\xi-\eta)\widehat{f_{k_1', j_1'}^{\tau}}(t,\eta-\sigma) \widehat{f_{k_2',j_2'}^{\kappa}}(t,\sigma)  d \eta d\sigma   d t, 
\]
\[
H_{j_i;j_1',j_2'}^{\mu,\nu,\tau, \kappa,2} =   \int_{t_1}^{t_2} \int_{\R^2}\int_{\R^2} e^{i t \Phi_2^{\tau, \kappa, \nu}(\xi, \eta, \sigma)} m^{ \tau, \kappa}_{\mu, \nu,2}(\xi, \eta, \sigma)\widehat{
f_{k_1',j_1'}^{\tau}
}(t,\xi-\sigma)\widehat{f_{k_2',j_2'}^{\kappa}}(t,\sigma-\eta) \widehat{f_{k_2,j_2}^{\nu}}(t,\eta)  d \eta d\sigma   d t, 
\]
\[
\mbox{JR}^{1}_{k_1,k_2}= \int_{t_1}^{t_2} \int_{\R^2} e^{i t \Phi^{\mu, \nu}(\xi, \eta)}\widehat{f^{\mu}}(t,\xi-\eta)\widehat{\mathcal{R'}^{\nu}}(t,\eta) m_{\mu, \nu}(\xi-\eta, \eta) \psi_{k_1}(\xi-\eta)\psi_{k_2}(\eta ) d \eta d t, \]
\[
\mbox{JR}^{2}_{k_1,k_2}= \int_{t_1}^{t_2} \int_{\R^2} e^{i t \Phi^{\mu, \nu}(\xi, \eta)}\widehat{\mathcal{R'}^{\mu}}(t,\xi-\eta)\widehat{f^{\nu}}(t,\eta) m_{\mu, \nu}(\xi-\eta, \eta) \psi_{k_1}(\xi-\eta)\psi_{k_2}(\eta ) d \eta d t, \]
\be\label{neweqn1}
\Phi_1^{\mu, \tau,\kappa}(\xi, \eta, \sigma)= \Lambda(\xi)- \mu \Lambda(\xi-\eta)-\tau\Lambda(\eta-\sigma)-\kappa\Lambda(\sigma),
\ee
\be\label{neweqn2}
\Phi_2^{\tau,  \kappa,\nu}(\xi, \eta, \sigma)= \Lambda(\xi)- \tau\Lambda(\xi-\sigma)-\kappa\Lambda(\sigma-\eta)-\nu\Lambda(\eta),
\ee
\begin{equation}\label{eqn4400}
m^{ \tau, \kappa}_{\mu, \nu,1}(\xi, \eta, \sigma)= m_{\mu, \nu}(\xi-\eta, \eta) \big(q_{\tau\nu, \kappa\nu}(\eta-\sigma, \sigma)\big)^{\nu}\psi_{k}(\xi)\psi_{k_1}(\xi-\eta)\psi_{k_2}(\eta),
\end{equation}
\begin{equation}\label{eqn4401}
m^{ \tau, \kappa}_{\mu, \nu,2}(\xi, \eta, \sigma)= m_{\mu, \nu}(\xi-\eta, \eta) \big(q_{\tau\mu, \kappa\mu}(\xi-\sigma, \sigma-\eta)\big)^{\mu}\psi_{k}(\xi)\psi_{k_1}(\xi-\eta)\psi_{k_2}(\eta).
\end{equation}
Here, we remind readers that ``$\tau \nu$ '' is understood as the product of signs, e.g., $+-=-$. 

From the estimate (\ref{productofsymbol}) in Lemma \ref{multilinearestimate}, the estimate (\ref{equation5667}),  and the estimate  (\ref{sizeofsymboluniform})  in Lemma \ref{sizeofsymbol}, the following estimate holds, 
\begin{equation}\label{equation7860}
\| m^{\tau, \kappa}_{\mu, \nu,1}(\xi, \eta, \sigma)\psi_{k_1'}(\eta-\sigma)\psi_{k_2'}(\sigma)\|_{\mathcal{S}^\infty} + \| m^{\tau, \kappa}_{\mu, \nu,2}(\xi, \eta, \sigma)\psi_{k_1'}(\xi-\sigma)\psi_{k_2'}(\sigma-\eta)\|_{\mathcal{S}^\infty} \lesssim 2^{k_1+k_{1,+}'+4\beta m}. 
\end{equation}

From the estimate (\ref{equation7760}) in Proposition \ref{propZnorm2}, we know that the $Z$-norm of $\mathcal{R}'(t)$ decays at rate $2^{-m  }$, which compensates the loss from  the integration with respect to time. With minor modifications,  the method used in the estimate of $\textup{End}_{k_1, k_2}^{\mu, \nu, i }$ can be applied directly to the estimate of $\mbox{JR}^{1}_{k_1,k_2}$ and $\mbox{JR}^{2}_{k_1,k_2}$. We omit details here.

Now let's proceed to estimate $J^{\mu, \nu,\tau, \kappa, 1}_{k_1',k_2'}$ and $J^{\mu, \nu,\tau, \kappa,2}_{k_1',k_2'}$.  From the $L^\infty-L^\infty-L^2$ type trilinear estimate (\ref{trilinearesetimate}) in Lemma \ref{multilinearestimate} and the estimate (\ref{equation7860}), the following estimate holds for fixed $k_1'$ and $k_2'$, 
\[
\sum_{i=1,2}2^{\delta j}\| \mathcal{F}^{-1}\big[J^{\mu, \nu,\tau, \kappa, i}_{k_1', k_2'}\big]\|_{B_{k,j}}
  \lesssim \sum_{i=1,2}2^{\alpha k +6k_+ +m +j+k_1+k_{1,+}'+4\beta m }   \| e^{-i t\Lambda} f_{k_i}\|_{L^\infty} \| e^{-it \Lambda} f_{k_1'}\|_{L^\infty}\|f_{k_2'}\|_{L^2}
\] 
\begin{equation}\label{equation6820}
\lesssim 2^{\alpha k +(1-\alpha)k_1+(1-2\alpha)k_2'} \min\{2^{ -2 k_{1,+}'+4\beta m +20\delta m },  2^{m  -(N_0-6) k_{1,+}'+4\beta m+30\delta m }\} \epsilon_0.
\end{equation}
In the above estimate, we used the fact that $\| e^{-it \Lambda} f_{k_1'}\|_{L^\infty}\lesssim 2^{k_1'} \| f_{k_1'}\|_{L^2}\lesssim 2^{k_1'-N_0 k_{1,+}' +\delta m}\epsilon_1$. From (\ref{equation6820}), we can rule out the case when $k_1'\geq 2\beta m$ or $k_2'\leq -6\beta m$. 

It remains to consider the case when  $k_1'$  and $k_2'$ are fixed and  $-6\beta m \leq k_2'\leq k_1'\leq 2\beta m.$ Recall that $|k_1-k_2|\leq 10$. With minor modifications,  we can  estimate $J^{\mu, \nu,\tau, \kappa, 2}_{k_1',k_2' }$ and  $J^{\mu, \nu,\tau, \kappa, 1}_{k_1',k_2' }$ in the same way. Hence,  we only show the estimate of $J^{\mu, \nu,\tau, \kappa, 1}_{k_1',k_2' }$ in details here.

 Firstly, we consider the case when $\kappa=-\tau$. For $J^{\mu, \nu,\tau, -\tau, 1}_{k_1' ,k_2' }$, we can first rule out the case when $\max\{j_1',j_2'\}\leq m  -20\beta m $ by doing integration by parts in $\sigma$ many times to see rapidly decay. More precisely, after doing integration by parts in ``$\sigma$'', we have
 \[
H_{ j_1',j_2'}^{\mu,\nu,\tau, -\tau,1}  = \int_{t_1}^{t_2}\int_{\R^2} \int_{\R^2} \frac{1}{t} e^{i t \Phi_1^{\mu, \tau,-\tau}(\xi, \eta, \sigma)} \nabla_{\sigma}\cdot \big( \tilde{m}^{ \tau, -\tau}_{\mu, \nu,1}(\xi, \eta, \sigma) \]
\be\label{neweqn45}
\times \widehat{
f_{k_1}^{\mu}
}(t,\xi-\eta)\widehat{f_{k_1',j_1'}^{\tau}}(t,\eta-\sigma) \widehat{f_{k_2',j_2'}^{-\tau}}(t,\sigma) \big) d\eta d \sigma d t, 
\ee
where
\be\label{neweqn79}
\tilde{m}^{ \tau, -\tau}_{\mu, \nu,1}(\xi, \eta, \sigma):= - \frac{{m}^{ \tau, -\tau}_{\mu, \nu,1}(\xi, \eta, \sigma)\nabla_\sigma \Phi_1^{\mu, \tau,-\tau}(\xi, \eta, \sigma)   }{i|\nabla_\sigma \Phi_1^{\mu, \tau,-\tau}(\xi, \eta, \sigma)|^2}.
\ee
From the estimate  (\ref{dece1})  in Lemma \ref{phasesize} and (\ref{equation7860}), we have
\be\label{newequation09903}
|\tilde{m}^{ \tau, -\tau}_{\mu, \nu,1}(\xi, \eta, \sigma)|\lesssim 2^{15\beta m},  \quad |\nabla_\sigma \tilde{m}^{ \tau, -\tau}_{\mu, \nu,1}(\xi, \eta, \sigma)| \lesssim 2^{-2(k_{2,-}+k_{1,-}') -k_{2}' + 9\beta m } \lesssim 2^{m/2}.
\ee
From the estimate (\ref{newequation09903}), it is easy to see  that  the net gain of doing integration by parts in ``$\sigma$'' once  is at least $ \max\{ 2^{-m/2}, 2^{-m+15\beta m-\max\{j_1',j_2'\}}\}$,  which is less than $2^{-\beta m}$. Therefore, we can repeat this process many times to see rapidly decay. 

From (\ref{equation7860}) and the $L^2-L^\infty-L^\infty$ type trilinear estimate (\ref{trilinearesetimate}) in Lemma \ref{multilinearestimate}, the following estimate holds after putting the input with the higher spatial concentration in $L^2$ and other inputs in $L^\infty$,
\[
\sum_{\max\{j_1',j_2'\}\geq m -20\beta m} 2^{\delta j} \| H_{ j_1',j_2'}^{\mu,\nu,\tau, -\tau,1} \|_{B_{k,j}} \lesssim \sum_{\max\{j_1',j_2'\}\geq m -20\beta m} 2^{\alpha k + 6k_+ +m +(1+\delta)j }2^{k_1+4\beta m +k_{1,+}' }\]
\be\label{neweqn141}
\times  2^{-\alpha k_1-\alpha k_1'-\alpha k_2'} 2^{-2m-\max\{j_1', j_2'\} } \| f_{k_1',j_1'}\|_{Z}\| f_{k_2',j_2'}\|_{Z}\| f_{k_1}\|_Z \lesssim 2^{-m/2}\epsilon_0.
\ee

Lastly,  we consider the case $\kappa=\tau$. Note that  $\nabla_\sigma\Phi_1^{\mu, \tau, \tau}(\xi, \eta, \sigma)=0$ if $\sigma=\eta/2$. Hence,  we localize around a small neighborhood of $(\xi, \eta, \eta/2)$ and split $H_{j_1;j_1',j_2'}^{\mu,\nu,\tau, \tau,1}$ into two parts as follows,
\[
H_{j_1;j_1',j_2'}^{\mu,\nu,\tau, \tau,1}=   \int_{t_1}^{t_2}\int_{\R^2} \int_{\R^2} e^{i t \Phi_1^{\mu, \tau, \tau}(\xi, \eta, \sigma)} m^{ \tau, \tau}_{\mu, \nu,1}(\xi, \eta, \sigma)\widehat{
f_{k_1,j_1}^{\mu}
}(t,\xi-\eta)\widehat{f_{k_1',j_1' }^{\tau}}(t,\eta-\sigma) 
\]
\[
 \times  \widehat{f_{k_2', j_2'}^{\tau}}(t,\sigma) \psi_{\leq k_2-10}(\eta-2\sigma) d \eta d\sigma   d t +\int_{t_1}^{t_2}\int_{\R^2} \int_{\R^2} e^{i t \Phi_1^{\mu, \tau,\tau}(\xi, \eta, \sigma)} m^{ \tau, \tau}_{\mu, \nu,1}(\xi, \eta, \sigma)\]
 \be\label{neweqn27}
 \times \widehat{
f_{k_1}^{\mu}
}(t,\xi-\eta)\widehat{f_{k_1' }^{\tau}}(t,\eta-\sigma) \widehat{f_{k_2' }^{\tau}}(t,\sigma )\psi_{> k_2-10}(\eta-2\sigma) \big)  d \eta d\sigma   d t.
\ee

Note that $\big||\eta|-|\xi-\eta|\big|\leq 2^{-4}|\eta|$ because $(k_1,k_2)\in \chi_k^1$. It is easy to see that $|\xi-\eta|-|\eta-\sigma|\sim |\eta|$ when $\sigma$ locates inside a small neighborhood of $\eta/2$.  Hence, we can take the advantage of the high oscillation in ``$\eta$''
by doing integration by parts in ``$\eta$'' when $\sigma$ is close to $\eta/2$.

Therefore, we do integration by parts in ``$\eta$'' for the first integral in (\ref{neweqn27}) and   do integration by parts in ``$\sigma$'' for the second integral in (\ref{neweqn27}).
 As a result, we have
\[
H_{j_1;j_1',j_2'}^{\mu,\nu,\tau, \tau,1}= \int_{t_1}^{t_2}\int_{\R^2} \int_{\R^2} \frac{1}{t} e^{i t \Phi_1^{\mu, \tau, \tau}(\xi, \eta, \sigma)} \nabla_{\sigma}\cdot \big( \tilde{m}^{ \tau, \tau}_{\mu, \nu,1}(\xi, \eta, \sigma)\widehat{
f_{k_1,j_1}^{\mu}
}(t,\xi-\eta)\widehat{f_{k_1',j_1' }^{\tau}}(t,\eta-\sigma) \widehat{f_{k_2',j_2'}^{\tau}}(t,\sigma) \big) \]
\be\label{neweqn47}
+\frac{1}{t} e^{i t \Phi_1^{\mu, \tau, \tau}(\xi, \eta, \sigma)} \nabla_{\eta}\cdot \big( \tilde{m}^{ \tau, \tau}_{\mu, \nu,2}(\xi, \eta, \sigma)\widehat{
f_{k_1,j_1}^{\mu}
}(t,\xi-\eta)\widehat{f_{k_1',j_1'}^{\tau}}(t,\eta-\sigma) \widehat{f_{k_2',j_2' }^{\tau}}(t,\sigma) \big) d \eta d\sigma   d t, 
\ee
where
\begin{equation}\label{equation6920}
\tilde{m}^{ \tau, \tau}_{\mu, \nu,1}(\xi, \eta, \sigma):=- \frac{{m}^{ \tau, \tau}_{\mu, \nu,1}(\xi, \eta, \sigma)\nabla_\sigma \Phi_1^{\mu, \tau,\tau}(\xi, \eta, \sigma) \psi_{> k_2-10}(\eta-2\sigma)  }{i|\nabla_\sigma \Phi_1^{\mu, \tau,\tau}(\xi, \eta, \sigma)|^2},
\end{equation}
\be\label{newequation240543}
\tilde{m}^{ \tau, \tau}_{\mu, \nu,2}(\xi, \eta,\sigma):= -\frac{{m}^{\mu, \tau, \tau}_1(\xi, \eta, \sigma)\nabla_\eta \Phi_1^{\mu, \tau,\tau}(\xi, \eta, \sigma) \psi_{\leq k_2-10}(\eta-2\sigma)  }{i|\nabla_\eta \Phi_1^{\mu, \tau,\tau}(\xi, \eta, \sigma)|^2}.
\end{equation}
From the estimates (\ref{dece1}) and (\ref{dece2}) in Lemma \ref{phasesize} and (\ref{equation7860}), we have
\be\label{noveqn101}
|  \tilde{m}^{ \tau, \tau}_{\mu, \nu,1}(\xi, \eta, \sigma)| + |  \tilde{m}^{ \tau, \tau}_{\mu, \nu,2}(\xi, \eta, \sigma)|  \lesssim 2^{-k_2-k_{2}'+10\beta m }.
\ee
\be\label{neweqn130}
|\nabla_\sigma \tilde{m}^{ \tau, \tau}_{\mu, \nu,1}(\xi, \eta, \sigma)| + |\nabla_\eta \tilde{m}^{ \tau, \tau}_{\mu, \nu,2}(\xi, \eta, \sigma)|\lesssim 2^{-2(k_{2,-}+k_{1,-}') -k_{2}' + 9\beta m } \lesssim 2^{m/2}.
\ee
Hence,     from estimates (\ref{noveqn101}) and (\ref{neweqn130}), we can see that  the net gain of doing integration by parts in ``$\sigma$''(when $\sigma$ is away from $\eta/2$) and ``$\eta$'' (when $\sigma$ is close to $\eta/2$) once  is at least $ \max\{2^{-m/2},$  $ 2^{-m} $  $\times 2^{ \max\{j_1,j_1',j_2'\}-k_2-k_{2}' +10\beta m }\}$, which is less than $2^{-\beta m}$ when $\max\{j_1, j_1', j_2'\}\leq m +k_2+k_2'-16\beta m$. Therefore, we can repeat this process many times to see rapidly decay, hence ruling out  the case when $\max\{j_1, j_1', j_2'\}\leq m +k_2+k_2'-16\beta m$.
 
It remains to consider the case when $\max\{j_1, j_1', j_2'\}\geq   m +k_2+k_2'-16\beta m $. From the estimate (\ref{equation7860}) and the $L^2-L^\infty-L^\infty$ type trilinear estimate (\ref{trilinearesetimate}) in Lemma \ref{multilinearestimate}, the following estimate holds after putting the input with the largest spatial concentration in $L^2$ and other two inputs in $L^\infty$, 
\[
\sum_{\max\{j_1, j_1', j_2'\}\geq m  +k_2+k_2'-16\beta m} 2^{\delta j}\| H_{j_i;j_1',j_2'}^{\mu,\nu,\tau, \tau,1}\|_{B_{k,j}} \lesssim \sum_{\max\{j_1, j_1', j_2'\}\geq m +k_2+k_2'-16\beta m}  2^{\alpha k +6k_+ + m +(1+\delta) j} 
\]
\be\label{neweqn150}
\times  2^{k_1+4\beta m+k_{1,+}'}  2^{-\alpha k_1'-\alpha k_2'-\alpha k_2} 2^{-2m -\max\{j_1,j_1',j_2'\}} \| f_{k_1,j_1}\|_Z \| f_{k_1',j_1'}\|_Z \| f_{k_2',j_2'}\|_Z \lesssim 2^{-m/2}\epsilon_0.
\ee
Hence finishing the proof.
\end{proof}
\subsection{When $(k_1, k_2, \mu, \nu)\in \chi_k^2\times   \{(-,-), (-,+)\}$}\label{lowhighgood}   As  we did before, we   first rule out the very-low-frequency case and the relatively-high-frequency case. More precisely, the following Lemma holds. 

\begin{lemma}\label{highlowhl}
For any fixed $j,m\in \mathbb{Z}_{+}$, 
under the bootstrap assumption \textup{(\ref{smallness})}  and the assumption that Proposition \textup{\ref{propZnorm2}} holds, then the following estimate  holds for any $k\in \mathbb{Z}$, $t_1, t_2\in [2^{m-1}, 2^m]$, 
\be\label{noveqn110} 
 \sum_{\begin{subarray}{c}
(k_1,k_2)\in \chi_k^2\\
 k_2\leq -(1+10\delta)(m+j )/(2-\alpha)\\
\end{subarray}
} \sum_{
 \mu, \nu  \in \{+,-\}}2^{\delta j}  \|
\int_{t_1}^{t_2}T^{\mu, \nu} (f_{k_1}, f_{k_2})  d t \|_{B_{k,j}}\lesssim 2^{-\delta m }\epsilon_0.
\ee
Moreover, for any $k\in \mathbb{Z}$ s.t., $k\notin[  -(1+10\delta)(m+j )/(4-\alpha), (1+10\delta)(m+j )/(N_0-10)]   $,   the following estimate holds,
\be\label{noveqn111} 
\sum_{(k_1,k_2)\in \chi_k^2}\sum_{ \mu, \nu  \in \{+,-\}} 2^{\delta j}  \|
\int_{t_1}^{t_2}T^{\mu, \nu} (f_{k_1}, f_{k_2})  d t \|_{B_{k,j}}\lesssim 2^{-\delta m }\epsilon_0.
\ee
\end{lemma}
\begin{proof}
From the $L^2-L^\infty$ type bilinear estimate (\ref{bilinearesetimate}) in Lemma \ref{multilinearestimate} and the estimate (\ref{sizeofsymboluniform}) in Lemma \ref{sizeofsymbol}, the following estimate holds, 
\begin{equation}\label{equation10020}
\|
\int_{t_1}^{t_2}T^{\mu, \nu} (f_{k_1}, f_{k_2})  d t \|_{B_{k,j}}\lesssim 2^{\alpha k+6k_{+} + m + j} \| q_{\mu, \nu}(\xi-\eta, \eta)\|_{\mathcal{S}^\infty_{k,k_1,k_2}}
\| P_{k_1} f \|_{L^2} \| e^{-it \Lambda} P_{k_2} f\|_{L^\infty}
\end{equation}
\begin{equation}\label{eqn420}
\lesssim 2^{   m+ j+k_1 
+  k_1+(2-\alpha)k_2-(N_0-7)k_{1,+}-4 k_{2,+} +\delta m}\epsilon_1^2.
\end{equation}
From the above estimate (\ref{eqn420}), it is easy to see that our desired estimates (\ref{noveqn110}) and (\ref{noveqn111}) hold. 
\end{proof}
Hence, from the estimates (\ref{noveqn110}) and (\ref{noveqn111}) in Lemma \ref{highlowhl}, we know that it would be sufficient  to  prove the following estimate, 
\be\label{noveqn116}
2^{\delta j}  \|
\int_{t_1}^{t_2}T^{\mu, \nu} (f_{k_1}, f_{k_2})  d t \|_{B_{k,j}}\lesssim 2^{-\delta m -\delta j }\epsilon_0,\quad  (\mu, \nu)\in\{(-,+), (-,-)\},
\ee
where fixed $k, k_1$ and $k_2$ satisfies the following estimates, 
\begin{equation}\label{range1}
-(1+10\delta)(m+j )/(2-\alpha) \leq k_2\leq k-10, |k_1-k|\leq 10,
\end{equation}
\begin{equation}\label{range2}
\, -(1+10\delta)(m+j )/(4-\alpha)\leq k \leq (1+10\delta)(m+j  )/(N_0-10). 
\end{equation}
Same as we did in the previous subsection, we separate into two cases based on the possible size of $j$.
\begin{lemma}\label{hlgoodj10}
Under the bootstrap assumption \textup{(\ref{smallness})}  and the assumption that Proposition \textup{\ref{propZnorm2}} holds, if $j\geq (1+20\delta)m$, then the desired estimate (\textup{\ref{noveqn116}}) holds  for any fixed $k $, $k_1,$ and $k_2$ that satisfy the estimates \textup{(\ref{range1})} and \textup{(\ref{range2})}.
\end{lemma}
 \begin{proof} Note that the rough estimate (\ref{eqn3320}) 
still holds for the case $(k_1,k_2)\in \chi_k^2$. The   sizes of frequencies of inputs do not play a role there. With minor modifications in the proof of Lemma \ref{hhgoodj1},  we can prove the desired estimate very similar if $j\geq (1+20\delta)m$. We omit details for this case here.
\end{proof}

\begin{lemma}\label{hlgoodj2}
Under the bootstrap assumption \textup{(\ref{smallness})}  and the assumption that Proposition \textup{\ref{propZnorm2}} holds, if $j\leq (1+20\delta)m$, then the desired estimate (\textup{\ref{noveqn116}}) holds  for any fixed $k $, $k_1,$ and $k_2$ that satisfy the estimates \textup{(\ref{range1})} and \textup{(\ref{range2})}.
\end{lemma}
\begin{proof}
 Since $j\leq (1+20\delta)m$, from estimates (\ref{range1}) and (\ref{range2}), we know that fixed $k, k_1,$ and $k_2$ satisfy the following estimates, 
 \begin{equation}\label{noveqn130}
-2(1+100\delta)m/(2-\alpha) \leq k_2\leq k-10, |k_1-k|\leq 10,
\end{equation}
\begin{equation}\label{noveqn131}
\, -2(1+100\delta)m/(4-\alpha)\leq k \leq 2\beta m, \quad \beta=1/980. 
\end{equation}

 Same as we did  in the proof of Lemma \ref{hhgoodj2}, we do integration by parts in time to take advantage of the high oscillation in time. As a result, we have the same identity as in (\ref{eqn450}). For simplicity, we use the same notations used there. Note that    the only difference is that now $(k_1, k_2)\in\chi_k^2$ instead of belongs to $\chi_k^1$.

 From the $L^2-L^\infty$ type bilinear estimate (\ref{bilinearesetimate}) in Lemma \ref{multilinearestimate} and the estimate (\ref{equation5667}), 
the following estimate holds when  $k_2\leq -\alpha m$,
\[
\sum_{i=1,2}2^{\delta j}\| \mathcal{F}^{-1}\big[\textup{End}_{k_1, k_2}^{\mu, \nu, i}\big]\|_{B_{k,j}} + 2^{\delta j} \| \mathcal{F}^{-1}\big[J^{\mu, \nu,i}_{k_1, k_2}\big]\|_{B_{k,j}} \lesssim 2^{\alpha k+ 6k_{+} + (1+\delta)j +3\beta m} \| e^{-it \Lambda} f_{k_1} \|_{L^\infty} \|  f_{k_2}\|_{L^2} \]
\[ 
+ 2^{\alpha k +6k_{+}+ (1+\delta)j +m +3\beta m } \big( \| e^{-i t\Lambda} \p_t f_{k_1} \|_{L^\infty} \|  f_{k_2}\|_{L^2} + \| e^{-it \Lambda } f_{k_1}\|_{L^\infty} \|\p_t  f_{k_2}\|_{L^2}\big)
\]
\be\label{noveqn145}
\lesssim 2^{(1-\alpha)k_2 + \alpha m/2+ 10\beta m }\epsilon_1^2 \lesssim 2^{-2\delta m-\delta j} \epsilon_0.
\ee
Note that in the above estimate, we used the following estimate 
\[
 \| e^{-i t\Lambda} \p_t f_{k_1} \|_{L^\infty} \lesssim \sum_{\mu, \nu\in\{+,-\}, k_2'\leq k_1'} \| e^{-it \Lambda} T^{\mu, \nu }(f_{k_1'}^{\mu}, f_{k_2'}^{\nu})\|_{L^\infty} + 2^{-m} \| P_{k_1}(\mathcal{R}')\|_{Z}\]
 \[\lesssim   \sum_{  k_2'\leq k_1'}  2^{k_1'+k_{1,+}'} \| e^{-it\Lambda} f_{k_1'}\|_{L^\infty} \| e^{-it \Lambda} f_{k_2'}\|_{L^\infty}+ 2^{-2m}\epsilon_0\lesssim  2^{-2m+\alpha m /2}\epsilon_0,
\]
 which can be derived from $L^\infty-L^\infty$ type bilinear estimate (\ref{bilinearesetimate}) in Lemma \ref{multilinearestimate}, the estimate (\ref{Linfinity3}), and the estimate (\ref{equation7760}) in Proposition \ref{propZnorm2}. From the estimate (\ref{noveq331}) in Lemma \ref{hlgoodj3} and the estimate (\ref{noveq332}) in Lemma \ref{hlgoodj7}, it is easy to see that the desired estimate (\ref{noveqn116}) also holds for the case when $k_2\geq -\alpha m$. Hence finishing the proof. 
\end{proof}
\begin{lemma}\label{hlgoodj3}
Under the bootstrap assumption \textup{(\ref{smallness})}  and the assumption that Proposition \textup{\ref{propZnorm2}} holds, if $j\leq (1+20\delta)m$, then the following estimate holds for fixed $k, k_1, k_2\in \mathbb{Z}, $ s.t., $(k_1, k_2, \mu, \nu)\in \chi_k^2\times   \{(-,-), (-,+)\}$, and $k_1, k_2\in[-\alpha m , 2\beta m ],$
\be\label{noveq331}
\sum_{i=1,2}2^{\delta j}\| \mathcal{F}^{-1}\big[\textup{End}_{k_1, k_2}^{\mu, \nu, i }\big]\|_{B_{k,j}}\lesssim 2^{-2\delta m-2\delta j}\epsilon_0.
\ee  
\end{lemma}
\begin{proof}
 Firstly, we consider the case when $\max\{j_1,j_2\}\geq m  +2k_2-4\beta m$. From  the $L^2-L^\infty$ type bilinear estimate (\ref{bilinearesetimate}) in Lemma \ref{multilinearestimate} and the estimate (\ref{equation5667}), the following estimate holds, 
\[
\sum_{\max\{j_1,j_2\}\geq m + 2k_2-4 \beta m} 2^{\delta j} \|  \mathcal{F}^{-1}\big[\textup{End}_{k_1,j_1, k_2,j_2}^{\mu, \nu, i }\big]\|_{B_{k,j}} \lesssim \sum_{\max\{j_1,j_2\}\geq m +2k_2-4\beta m}  2^{\alpha k +6k_+ + (1+\delta)j+4\beta m }\]
\[\times 2^{ -\max\{j_1,j_2\}-m-\alpha(k_1+k_2)} \|f_{k_1,j_1}\|_{Z} \|f_{k_2,j_2}\|_{Z} \lesssim
2^{-m-(2+2\alpha)k_2+15\beta m } \epsilon_1^2\lesssim 2^{-2\delta m -2\delta j}\epsilon_0.
\]
For the case when $\max\{j_1, j_2\}\leq m +2k_2-4\beta m$, we can do integration by parts in`` $\eta$'' to see rapidly decay. If $\nabla_{\eta}$ hits  $\widehat{f}_{k_1,j_1}$ and $\widehat{f}_{k_2,j_2}$, we at most lose $2^{\max\{j_1,j_2\}-2k_2+5k_{1,+}/2}$, which is less than $2^{m-\beta m}$. If $\nabla_{\eta}$ hits the symbol $\tilde{m}_{\mu, \nu}(\cdot, \cdot)$, then from the  estimates  (\ref{dece1}) and (\ref{dece2})  in Lemma \ref{phasesize}, it is easy to see that the maximal loss  is   $2^{-k_2-4k_{1,-}+k_{1,+}}$, which is less than $2^{m-\beta m }$. Hence the net gain is at least $2^{-\beta m }$ from integration by parts in ``$\eta$'' once. We can repeat  this process many times  to see rapidly decay.  Hence finishing the proof. 
\end{proof}
 \begin{lemma}\label{hlgoodj7}
Under the bootstrap assumption \textup{(\ref{smallness})}  and the assumption that Proposition \textup{\ref{propZnorm2}} holds, if $j\leq (1+20\delta)m$, then the following estimate holds for fixed $k, k_1, k_2\in \mathbb{Z}$,  s.t., $(k_1, k_2, \mu, \nu)\in \chi_k^2\times   \{(-,-), (-,+)\}$, and $k_1, k_2\in[-\alpha m , 2\beta m ],$
\be\label{noveq332}
 \sum_{i=1,2}2^{\delta j}\| \mathcal{F}^{-1}\big[J^{\mu, \nu,i}_{k_1, k_2}\big]\|_{B_{k,j}} \lesssim 2^{-2\delta m-2\delta j}\epsilon_0.
\ee  
\end{lemma}
\begin{proof}

 Same as before, the method used in the estimate of $\textup{End}_{k_1, k_2}^{\mu, \nu, i }$ can be applied directly to the estimate of $\mbox{JR}^{1}_{k_1,k_2}$ and $\mbox{JR}^{2}_{k_1,k_2}$. Now we  proceed to estimate $J^{\mu, \nu,\tau, \kappa, 1}_{k_1' ,k_2' }$ and $J^{\mu, \nu,\tau, \kappa, 2}_{k_1', k_2' }$. 

Recall (\ref{eqn4400}) and (\ref{eqn4401}). From the estimate  (\ref{productofsymbol}) in Lemma \ref{multilinearestimate},  the estimate (\ref{sizeofsymboluniform}) in Lemma \ref{sizeofsymbol}, and   the estimate (\ref{equation5667}), the following estimate holds,  
\begin{equation}\label{equation8860}
\| m^{ \tau, \kappa}_{\mu, \nu,1}(\xi, \eta, \sigma)\psi_{k_1'}(\eta-\sigma)\psi_{k_2'}(\sigma)\|_{\mathcal{S}^\infty} + \| m^{ \tau, \kappa}_{\mu, \nu,2}(\xi, \eta, \sigma)\psi_{k_1'}(\xi-\sigma)\psi_{k_2'}(\sigma-\eta)\|_{\mathcal{S}^\infty} \lesssim 2^{k_1'+4\beta m}. 
\end{equation}
From the estimate (\ref{equation8860}) and the   $L^2-L^\infty-L^\infty$ type trilinear estimate (\ref{trilinearesetimate}) in Lemma \ref{multilinearestimate}, the following estimate holds for fixed $k_1'$ and $k_2'$, 
\[
\sum_{i=1,2}2^{\delta j}\| \mathcal{F}^{-1}\big[J^{\mu, \nu,\tau, \kappa,i}_{k_1', k_2' }\big]\|_{B_{k,j}} \lesssim \sum_{i=1,2} 2^{2\delta m + \alpha k +6k_+ +m +j+k_1'+4\beta m }   \| e^{-i t\Lambda} f_{k_i}\|_{L^\infty}  \| e^{-it \Lambda} f_{k_1'}\|_{L^\infty}\|f_{k_2'}\|_{L^2}
\] 
\[
\lesssim \min\{2^{(1-\alpha)(k_1'+k_2')+11\beta m},  2^{(1-\alpha)k_2'+m+4\beta m-(N_0-8)k_{1,+}'}\}\epsilon_0.
\]
From the above estimate, we can rule out the case when  $k_1', k_2'\notin[-14\beta m, \beta m ]$. Now, it is sufficient to consider the case when $k_1'$ and $k_2'$ are  fixed and  $-14\beta m \leq k_2'\leq k_1'\leq \beta m$.

Recall that $|\eta|\leq 2^{-4}|\xi-\eta|$ for  the case we are considering. It is easy to see that  ``$\xi-\eta$'' and ``$\eta/2$'' are still not close. Hence, the methods used in the proof of Lemma \ref{hhgoodj7} for the $\chi_k^1$ case can be applied directly here. Therefore we can do integration by parts in ``$\sigma$'' when ``$\sigma$'' is far away from $\eta/2$ and do integration by parts in $\eta$ when $\sigma$ is close to $\eta/2$ to take the advantage of the high oscillation in $\sigma$ or $\eta$. As a result, we can rule out the case when $\max\{j_i, j_1', j_2'\}\leq m+k_2+k_2'-16\beta m $. 

For the case when $\max\{j_i, j_1', j_2'\}\geq m+k_2+k_2'-16\beta m $, it is easy to verify that a similar estimate as in the estimate  (\ref{neweqn150}) is still valid for the case when $k_1, k_2\in [-\alpha m , 2\beta m]$ and $k_1', k_2'\in[-14\beta m , \beta m]$.  Hence finishing the proof.

 \end{proof}

\section{The Improved $Z$-Norm Estimate: Bad Type Phases }\label{badimprovedZnorm}

The main goal of this section is to prove the desired Proposition \ref{propZnorm4bad}. In other words, we will prove the desired estimate (\ref{badtypeZ}) under the bootstrap assumption (\ref{smallness}) and the assumption that Proposition \ref{propZnorm2} holds. Note that the estimate (\ref{sizeofprofilelowfrequency}) is valid in this section. 

 Recall (\ref{badtypeterms}) and (\ref{badtypephase}). In   subsection \ref{rangehighhigh}, we consider the case $(k_1,k_2, \mu, \nu)\in \chi_k^1\times\{(+,-), (-,+)\}$.  In   subsection  \ref{hlbad}, we consider the case  $(k_1,k_2, \mu, \nu)\in \chi_k^2\times$ $\{(+,+), (+,-)\}$. In subsection \ref{comparablebad}, we estimate $K^{\mu, \nu}(f_{k_1}^{\mu}, f_{k_2}^{\nu})$, where $\mu, \nu\in\{+, -\}$ and $(k_1,k_2)\in \chi_k^3$.   Hence finishing the proof.

\subsection{When $(k_1, k_2)\in \chi_k^1, (\mu, \nu)\in \{(-,+), (+,-)\} $}\label{hhbad}  

Note that the estimates (\ref{nove2}) and (\ref{nove3}) in Lemma \ref{ruleoutlowhigh} holds regardless the sign of $\mu$ and $\nu$.  Moreover, the proof of the Lemma \ref{hhgoodj1}  is also valid regardless the sign of $\mu$ and $\nu$. Hence,  we can rule out the very-low-frequency case,  the relatively-high-frequency case, and the case $j\geq (1+20\delta)m$  as in subsection \ref{highhighgood}.  Moreover, from the estimate (\ref{noveqn519}) in Lemma \ref{roughestimatephase2}, it is easy to see that the phase is not degenerated if $k_1\geq 0$. Therefore, the case when $k_1\geq 0$ can be handled in the same way as in the subsection \ref{highhighgood}.
To sum up, in this subsection,  it would be sufficient to consider 
 fixed  $k$,  $k_1$, $k_2,$ and $j$ that satisfy the following estimate, 
\begin{equation}\label{rangehighhigh}
 - (1+\delta)(m+j)/(2+\alpha)  \leq k\leq k_1- 10 \leq k_2 \leq k_1+10\leq 10,\quad  j\leq m +20\delta m.  
\end{equation}

From the estimate of bad type phases in (\ref{noveqn519}), we know that the size of phases highly depends on the angle between $\xi$ and $\nu \eta$. This fact motivates us to do dyadic decomposition for the angle between $\xi$ and $\nu\eta$   with  a threshold $\bar{l}$ chosen to be $ 2k_{1,-}$  as follows, 
\begin{equation}\label{equation6601}
\mathcal{F}[\int_{t_1}^{t_2}T^{\mu, \nu} (f_{k_1}(t), f_{k_2}(t))  d t](\xi)=\sum_{ \bar{l}\leq l \leq 2} I_{l}^{\mu, \nu},\quad I_{l}^{\mu,\nu} = \sum_{j_1\geq -k_{1,-}, j_2\geq -k_{2,-}} I_{l;j_1, j_2}^{\mu, \nu}, \end{equation}
where
\begin{equation}\label{equation7002}
I_{l;j_1, j_2}^{\mu, \nu}= \int_{t_1}^{t_2} \int e^{i t\Phi^{\mu, \nu}(\xi, \eta)} \varphi_{\bar{l};l}(\angle(\xi, \nu\eta)) {q}_{\mu, \nu}(\xi-\eta, \eta)\widehat{f^{\mu}_{k_1,j_1}}(t,\xi-\eta)\widehat{f^{\nu}_{k_2,j_2}}(t,\eta)  d \eta  dt, 
\end{equation}
and $\varphi_{\bar{l};l}(\cdot)$ is defined as follows, 
\begin{equation}\label{angularcutoff}
\varphi_{\bar{l};l}(x) =\left\{\begin{array}{ll}
 \psi_{l}(x) & \textup{if $ \bar{l}< l \leq 2$}\\
 \psi_{\leq \bar{l}}(x) & \textup{if\, } l= \bar{l}:=2k_{1,-}.
\end{array}\right.
\end{equation}
Note that there are at most  $m^2$ cases in total for $l$, which is only a logarithmic loss. Hence we will also let $l$ to be fixed in the rest of this subsection.

 To sum up, 
it would be sufficient to prove the following estimate in this subsection, 
\be\label{baddesired1}
  2^{\delta j}  \| \mathcal{F}^{-1}\big[
I_{l}^{\mu, \nu}\big] \|_{B_{k,j}}\lesssim 2^{-2\delta m -2\delta j}\epsilon_0,\quad   (\mu, \nu)\in\{(-,+), (+,-)\},
\ee
where  fixed  $k$,  $k_1$, and $j$ satisfy the estimate (\ref{rangehighhigh}) and fixed $l\in [-2k_{1,-}, 2]$.  Based on the possible sizes of $j$, $k+2k_1$, and $k+2l$, we separate into five cases. As a result, the desired estimate (\ref{baddesired1}) follows from Lemma \ref{badhhlemma1}, Lemma \ref{badhhlemma2}, Lemma \ref{badhhlemma3}, Lemma \ref{badhhlemma4}, Lemma \ref{badhhlemma5}.

\begin{lemma}\label{badhhlemma1}
Under the bootstrap assumption \textup{(\ref{smallness})}  and the assumption that Proposition \textup{\ref{propZnorm2}} holds, if $  \max\{m+  l , -k-l\}+ 100 \delta m \leq j \leq m+20\delta m $, then the desired estimate  \textup{(\ref{baddesired1})}  holds  for any fixed $k$,  $k_1$, and $k_2$ that satisfy \textup{(\ref{rangehighhigh})}.
\end{lemma}
\begin{proof}
Note that the assumption on the size of $j$ implies that \emph{$k\leq -30\delta m$} because otherwise $\max\{m+  l , -k-l\}+ 100 \delta m >  m+20\delta m $.    

We first consider the case when 
  $\min\{j_1, j_2\}\geq j-\delta m$.  From  (\ref{eqn298}) in Lemma \ref{bilinearest} and (\ref{sizeofsymboluniform}) in Lemma \ref{sizeofsymbol}, the following estimate holds, 
\[
\sum_{ \min\{j_1, j_2\}\geq j-\delta m}2^{\delta j}\|\mathcal{F}^{-1}[ I_{l;j_1, j_2}^{\mu, \nu}]\|_{B_{k,j}} \lesssim \sum_{ \min\{j_1, j_2\}\geq j-\delta m}  2^{(2+\alpha) k  +m+(1+\delta)j+k +k+l/2}   \|   f_{k_1,j_1}\|_{L^2}  \]
\begin{equation}\label{equation6603}
\times  \|  f_{k_2,j_2}\|_{L^2} \lesssim  2^{(2+\alpha) k +m+2\delta m -(1-\delta)j-2\alpha k_1+l/2} \epsilon_1^2 \lesssim 2^{(1-\alpha)k +2\delta m }\epsilon_0\lesssim 2^{-2\delta m -2\delta j}\epsilon_0.
\end{equation}

Now we consider the case when $\min\{j_1, j_2\}\leq j -\delta m$. For this case, we will   do integration by parts in ``$\xi$" repeatedly to see rapidly decay. Recall that $\mu \nu=-$. From the estimate (\ref{angularrelation}) in Lemma \ref{angularrelationlemma} and the estimate (\ref{eqn140}) in Lemma \ref{phasesize}, it is easy to see that the following estimate holds,
 \[
 |\nabla_{\xi} \Phi^{\mu, \nu}(\xi, \eta)|\varphi_{\bar{l};l}(\angle(\xi, \nu \eta))= \Big|\Lambda'(|\xi|)\frac{\xi}{|\xi|}-\mu \Lambda'(|\xi-\eta|) \frac{\xi-\eta}{|\xi-\eta|} \Big|\varphi_{\bar{l};l}(\angle(\xi, \nu \eta))\lesssim 2^{l}.
  \]
From the above estimate and the assumption that $j\geq m+l +100\delta m$,   we have  the following estimate, 
\begin{equation}\label{eqn12200}
|\nabla_{\xi}[x\cdot \xi + t \Phi^{\mu, \nu}(\xi,\eta)]|\varphi_{\bar{l};l}(\angle(\xi, \nu\eta))\varphi_{j}^k(x)  \sim 2^{j}.
\end{equation}

Hence, after doing integration by part in $\xi$ once, we can gain $2^{-j}$ by paying the price of at most  $\max\{2^{\min\{j_1, j_2\}},$ $  2^{-k-l}\}$, where $2^{-k-l}$ comes from the fact that $\nabla_{\xi}$ might hit the angular cutoff function $\varphi_{\bar{l};l}(\angle(\xi, \nu \eta))$ or the symbol $a_{\mu, \nu}(t,x,\xi, \eta)$ (see (\ref{equation6840})).  As $j\geq -k-l+ 100\delta m$ and $\min\{j_1,j_2\}\leq j-\delta m$, we can see that the net gain of  doing  integration by parts in ``$\xi$'' once is at least $2^{-\delta m}$. Hence, we can keep doing this process to see rapidly decay. More precisely, the following estimate holds.
\[
\big|\mathcal{F}^{-1}[ I_{l;j_1, j_2}^{\mu, \nu}](x)\big|\varphi_{j}^k(x)\lesssim 2^{-10m} \| f_{k_1,j_1}\|_{L^2}  \| f_{k_2,j_2}\|_{L^2},\]
Hence, it's easy to see that the following estimate holds
\begin{equation}\label{dece4}
 \sum_{\min\{j_1, j_2\}\leq j-\delta m}2^{\delta j}\|\mathcal{F}^{-1}[ I_{l;j_1, j_2}^{\mu, \nu}]\|_{B_{k,j}} \lesssim 2^{-2\delta m-2\delta j}\epsilon_0.
\end{equation}
Hence  finishing the proof  from the estimates (\ref{equation6603}) and (\ref{dece4}). 
\end{proof}

\begin{lemma}\label{badhhlemma2}
Under the bootstrap assumption \textup{(\ref{smallness})}  and the assumption that Proposition \textup{\ref{propZnorm2}} holds, if $  j\leq \max\{m+l, -k-l\}+100 \delta m $, then the desired estimate  \textup{(\ref{baddesired1})} holds  for any fixed $k$,  $k_1$, and $k_2$ that satisfy \textup{(\ref{rangehighhigh})} under the assumption that  $k+ 2k_1\leq-m+\beta m$, and $k +2l\leq -m +2\beta m $.
\end{lemma}
\begin{proof}
From the assumptions that $j\leq \max\{m+l, -k-l\}+100 \delta m$, $k+2k_1\leq -m+\beta m $ and $k+2l\leq -m +2\beta m $, it is easy to see that  $j\leq -k-l+2\beta m +100\delta m $ and  $k\leq -m/3+\beta m $.  

From the estimate (\ref{eqn298}) in Lemma \ref{bilinearest}, the estimate (\ref{sizeofsymboluniform}) in Lemma \ref{sizeofsymbol}, and the estimates (\ref{L2estimate}) and (\ref{equation6726}) in Lemma \ref{L2estimatelemma}, the following estimate holds, 
\[
  2^{\delta j}\|\mathcal{F}^{-1}[ I_{l}^{\mu, \nu}]\|_{B_{k,j}}\lesssim  2^{2\delta m+ \alpha k   + m + (1+\delta) j +k + k+k_1+l}  \| f_{k_1}\|_{L^2}\|  \widehat{f_{k_2}}(t, \xi)\|_{L^\infty_\xi} 
\]
\[ 
\lesssim 2^{(1+\alpha) k+m + 2 k_1+3\beta m} \epsilon_0 \lesssim 2^{ \alpha k +4\beta m }\epsilon_0 \lesssim 2^{-2\delta m -2\delta j}\epsilon_0.
\]
Hence finishing the proof.
\end{proof}
 \begin{lemma}\label{badhhlemma3}
Under the bootstrap assumption \textup{(\ref{smallness})}  and the assumption that Proposition \textup{\ref{propZnorm2}} holds, if $  j\leq \max\{m+l, -k-l\}+100 \delta m $, then the desired estimate  \textup{(\ref{baddesired1})} holds  for any fixed $k$,  $k_1$, and $k_2$ that satisfy \textup{(\ref{rangehighhigh})} under the assumption that  $k+ 2k_1\leq-m+\beta m$, and $k+2l\geq -m+2\beta m$.
\end{lemma}
 \begin{proof}
From the assumptions on $j$, $k+2k_1$ and $k+2l$, it is easy to see that $j\leq m + l+100\delta m$ and $k\leq -m/3+\beta m $.

\noindent $\bullet$\quad We first consider the case when $\max\{j_1,j_2\}\geq m +k - k_1+l -4\beta m$.   From the estimate (\ref{equation6884}) in Lemma \ref{bilinearest} and the estimate (\ref{sizeofsymboluniform}) in Lemma \ref{sizeofsymbol}, the following estimate holds, 
\[
\sum_{\max\{j_1, j_2\}\geq m+k-k_1+l-4\beta m }2^{\delta j} \| \mathcal{F}^{-1}[I_{l;j_1, j_2}^{\mu, \nu}]\|_{B_{k,j}}\lesssim \sum_{\max\{j_1, j_2\}\geq m+k-k_1+l-4\beta m } 2^{\alpha k  + m +(1+\delta)j}
\]
\[
\times 2^{k +\beta m -m-2\alpha k_1 -\max\{j_1,j_2\} } \| f_{k_1,j_1}\|_Z \| f_{k_2,j_2}\|_Z \lesssim 2^{\alpha k +(1-2\alpha)k_1   +6\beta m }\epsilon_1^2  \lesssim 2^{-2\delta m -2\delta j}\epsilon_0.
\]
 
\noindent $\bullet$\quad Now we consider the case when $\max\{j_1,j_2\}\leq m +k-k_1 +l -4 \beta m $.  For this case, we can keep doing  integration by parts in `` $\eta$'' to see rapidly decay.   Because $k+2l \geq -m +2\beta m$, which means that we are away from the space resonance set,  there is no problem   when $\nabla_{\eta}$ hits the symbol $\tilde{m}_{\mu, \nu}(\xi-\eta, \eta)$ when doing integration by parts in $\eta$, see (\ref{integrationeta}) and (\ref{eqn134}) in Lemma \ref{phasesize}. From   (\ref{eqn134}) in Lemma \ref{phasesize}, we can see that the net gain of doing integration by parts in ``$\eta$'' once is at least $2^{-m}\max\{2^{-k-2l+\beta m},$ $ 2^{\max\{j_1, j_2\}-k+k_1-l+3\beta m}\} $, which is less than $2^{-\beta m}$. Therefore, we can keep doing this process to see rapidly decay.   
\end{proof}
 \begin{lemma}\label{badhhlemma4}
Under the bootstrap assumption \textup{(\ref{smallness})}  and the assumption that Proposition \textup{\ref{propZnorm2}} holds, if $  j\leq \max\{m+l, -k-l\}+100 \delta m $, then the desired estimate  \textup{(\ref{baddesired1})} holds  for any fixed $k$,  $k_1$, and $k_2$ that satisfy \textup{(\ref{rangehighhigh})} under the assumption that  $k+ 2k_1\geq-m+\beta m$, and $k+2l\leq -m+2\beta m$.
\end{lemma}
\begin{proof}
From the assumptions on $j$, $k+2k_1$ and $k+2l$, it is easy to see that  $j\leq -k-l+\beta m +100\delta m$, $k\leq -m/5+\beta m$ and $l\leq k_1 + \beta m/2$.

Because $k+2k_1\geq -m +\beta m$, which means that we are away from the time resonance set,  we  do integration by parts in time once to take   advantage of the high oscillation in time. The formulas are very similar to (\ref{eqn450}) and (\ref{tridecom}).   For the sake of readers, we still state them in details as follows, 
\begin{equation}\label{equation7000}
I_{l}^{\mu, \nu}= \sum_{i=1,2} \textup{End}_{l;k_1,k_2}^{\mu, \nu, i} +  J^{\mu, \nu,i}_{l;k_1,k_2}, \quad  J^{\mu, \nu,i}_{l;k_1,k_2} =   \sum_{k_1', k_2'\in \mathbb{Z}} \sum_{\mu', \kappa'\in \{+, -\}} J^{\mu, \nu,\tau, \kappa, i}_{l;k_1',k_2'} + \mbox{JR}^{i}_{l;k_1,k_2},\quad i\in\{1,2\},
\end{equation}
\begin{equation}\label{equation7001}
 J^{\mu, \nu,\tau, \kappa, i}_{l;k_1',k_2' }:= \sum_{j_1'\geq -k_{1,-}', j_2'\geq -k_{2,-}', j_i \geq -k_{i,-}} H_{l,j_i;j_1',j_2'}^{\mu,\nu,\tau, \kappa,i},, \quad \textup{End}_{l;k_1,k_2}^{\mu, \nu, i}= \sum_{j_1\geq -k_{1,-}, j_2\geq -k_{2,-}}\textup{End}_{l;k_1,j_1,k_2,j_2}^{\mu, \nu, i},
\end{equation}
\begin{equation}\label{tridecomwithangle}
H_{l;j_1',j_2'}^{\mu,\nu,\tau, \kappa,i}=\sum_{ j_i \geq -k_{i,-}} H_{l,j_i;j_1',j_2'}^{\mu,\nu,\tau, \kappa,i}, \quad H_{l,j_i;j_1'} ^{\mu,\nu,\tau, \kappa,i}=\sum_{ j_2' \geq -k_{2,-}'} H_{l,j_i;j_1',j_2'}^{\mu,\nu,\tau, \kappa,i},\quad i \in \{1,2\},
\end{equation}
where 
\be\label{noveqn810}
\textup{End}_{l;k_1,j_1,k_2,j_2}^{\mu, \nu, i} =  (-1)^{i-1} \int_{\R^2} e^{i t_i \Phi^{\mu, \nu}(\xi, \eta)}\widehat{ f^{\mu}_{k_1,j_1}}(t_i,\xi-\eta)\widehat{f^{\nu}_{k_2,j_2}}(t_i,\eta) m_{\mu, \nu}(\xi-\eta, \eta) \varphi_{\bar{l};l}(\angle(\xi, \nu\eta)) d \eta,
\ee
\[
H_{l,j_1;j_1',j_2'}^{\mu,\nu,\tau, \kappa,1}=   \int_{t_1}^{t_2}\int_{\R^2} \int_{\R^2} e^{i t \Phi_1^{\mu, \tau, \kappa}(\xi, \eta, \sigma)} m^{ \tau, \kappa}_{\mu, \nu,1}(\xi, \eta, \sigma) \varphi_{\bar{l};l}(\angle(\xi, \nu\eta))\]
\be\label{noveqn811}
\times \widehat{
f_{k_1,j_1}^{\mu}
}(t,\xi-\eta)\widehat{f_{k_1',j_1'}^{\tau}}(t,\eta-\sigma) \widehat{f_{k_2',j_2'}^{\kappa}}(t,\sigma)  d \eta d\sigma   d t, 
\ee
\[
H_{l,j_2;j_1',j_2'}^{\mu,\nu,\tau, \kappa,2}= \int_{t_1}^{t_2} \int_{\R^2}\int_{\R^2} e^{i t_i \Phi_2^{\tau, \kappa, \nu}(\xi, \eta, \sigma)} m^{ \tau, \kappa}_{\mu, \nu,2}(\xi, \eta, \sigma)\varphi_{\bar{l};l}(\angle(\xi, \nu\eta)) \]
\be\label{noveqn812}
\times \widehat{
f_{k_1',j_1'}^{\tau}
}(t,\xi-\sigma)\widehat{f_{k_2',j_2'}^{\kappa}}(t,\sigma-\eta) \widehat{f_{k_2,j_2}^{\nu}}(t,\eta)  d \eta d\sigma   d t, 
\ee
\be\label{noveqn813}
\mbox{JR}^{1}_{l;k_1,k_2}= \int_{t_1}^{t_2} \int_{\R^2} e^{i t \Phi^{\mu, \nu}(\xi, \eta)}\widehat{f^{\mu}}(t,\xi-\eta)\widehat{\mathcal{R'}^{\nu}}(t,\eta) m_{\mu, \nu}(\xi-\eta, \eta) \psi_{k_1}(\xi-\eta)\psi_{k_2}(\eta )\varphi_{\bar{l};l}(\angle(\xi, \nu\eta)) d \eta d t, \ee
\be\label{noveqn814}
\mbox{JR}^{2}_{l;k_1,k_2}=\int_{t_1}^{t_2} \int_{\R^2} e^{i t \Phi^{\mu, \nu}(\xi, \eta)}\widehat{\mathcal{R'}^{\mu}}(t,\xi-\eta)\widehat{f^{\nu}}(t,\eta) m_{\mu, \nu}(\xi-\eta, \eta) \psi_{k_1}(\xi-\eta)\psi_{k_2}(\eta )\varphi_{\bar{l};l}(\angle(\xi, \nu\eta)) d \eta d t,\ee
where $m_{\mu, \nu}(\xi-\eta, \eta)$, $ m^{ \tau, \kappa}_{\mu, \nu,1}(\xi, \eta, \sigma)$, and $ m^{ \tau, \kappa}_{\mu, \nu,2}(\xi, \eta, \sigma)$ are defined  in (\ref{neweqn90}), (\ref{eqn4400}), and (\ref{eqn4401}). From the estimate (\ref{eqn61001}) in Lemma \ref{Snorm}, the estimate (\ref{sizeofsymboluniform}) in Lemma \ref{sizeofsymbol}, the estimate (\ref{noveqn519}) in Lemma \ref{roughestimatephase2}, and the estimate (\ref{eqn134}) in Lemma \ref{phasesize}, the following estimates hold, 
\begin{equation}\label{equation6860}
\| m_{\mu, \nu}(\xi-\eta, \eta) \varphi_{\bar{l};l}(\angle(\xi, \nu \eta))\|_{L^{\infty}_{\xi, \eta}} \lesssim 2^{k-k-2\max\{k_{1,-}, l\} }  ,
\end{equation}
\[
  \| m_{\mu, \nu}(\xi-\eta, \eta) \varphi_{\bar{l};l}(\angle(\xi, \nu \eta))\|_{\mathcal{S}^{\infty}_{k,k_1,k_2}}
\]
\begin{equation}\label{equation6861}
 \lesssim \max\{2^{k-k-2\max\{k_{1,-}, l\}  - 3k_{1,-}},2^{k-k-2\max\{k_{1,-}, l\} - 3l} \}  \lesssim 2^{-2\max\{k_{1,-}, l\}-6k_{1,-}  }.
\end{equation}
   From the estimate (\ref{eqn298}) in Lemma \ref{bilinearest}, estimates (\ref{L2estimate}), (\ref{L2derivativeestimate}) and  (\ref{equation6726}) in Lemma \ref{L2estimatelemma}, and the estimate (\ref{equation6860}), the following estimate holds, 
\[
\sum_{i=1,2}2^{\delta j}\| \mathcal{F}^{-1}[ \textup{End}_{l;k_1, k_2}^{\mu, \nu, i } ]\|_{B_{k,j}} + \sum_{i=1,2}2^{\delta j}\| \mathcal{F}^{-1}[ J^{\mu, \nu,i}_{l;k_1, k_2}]\|_{B_{k,j}}  \lesssim \sum_{i=1,2} 2^{\alpha k + (1+\delta)j-2\max\{k_{1,-},l\}  +2 \beta m + k+k_1+l}\]
\[
  \times \big(\| \widehat{f_{k_i}}(\xi)\|_{L^\infty_\xi} \|f_{k_{3-i}}\|_{L^2}  + 2^m \| \widehat{f_{k_i}}(t,\xi)\|_{L^\infty_\xi} \| \p_t f_{k_{3-i}}\|_{L^2}\big)
\lesssim 2^{\alpha k  +3\beta m }\epsilon_1^2 \lesssim 2^{-2\delta m -2\delta j}\epsilon_0.
\]
Hence finishing the proof.
\end{proof}
 
 \begin{lemma}\label{badhhlemma5}
Under the bootstrap assumption \textup{(\ref{smallness})}  and the assumption that Proposition \textup{\ref{propZnorm2}} holds, if $  j\leq \max\{m+l, -k-l\}+100 \delta m $, then the desired estimate  \textup{(\ref{baddesired1})} holds  for any fixed $k$,  $k_1$, and $k_2$ that satisfy \textup{(\ref{rangehighhigh})} under the assumption that  $k+ 2k_1\geq-m+\beta m$, and $k+2l\geq -m+2\beta m$.
\end{lemma}
\begin{proof}
From the assumptions on $j$, $k+2k_1$ and $k+2l$, it is easy to see that  $j\leq m+l +100\delta m$ and $k_1\geq -m/3$.  Same as in Lemma \ref{badhhlemma4}, we do integration by parts in time first to take   advantage of high oscillation in time. 

$\bullet$\quad \textit{Estimate of the endpoint case.}\quad Because $k+2l \geq -m +2\beta m $, which means that we are away from the space resonance set,  we can first rule out  the case when $\max\{j_1,j_2\}\leq m +k-k_1 +l -4 \beta m $ by keep doing integration  by parts in `` $\eta$'' many times.

 Now it would be sufficient to consider the case when $\max\{j_1, j_2\}\geq m+k-k_1+l -4\beta m$. From the estimate (\ref{eqn298}) in Lemma \ref{bilinearest} and the estimate (\ref{equation6860}), the following estimate holds after putting the input with higher spatial concentration in $L^2$ and the other input in $L^\infty_\xi$.
\[
\sum_{\max\{j_1, j_2\}\geq m+k-k_1+l-4\beta m }  \sum_{i=1,2} 2^{\delta j}\| \mathcal{F}^{-1}[ \textup{End}_{l;k_1,j_1,  k_2, j_2}^{\mu, \nu, i } ]\|_{B_{k,j}}
\]
\[
 \lesssim \sum_{\max\{j_1, j_2\}\geq m+k-k_1+l-4\beta m } 2^{\alpha k  + (1+\delta)j -2\max\{k_{1,-},l\}    + k +k_1+l -\max\{j_1, j_2\} -2\alpha k_1}  \]
 \[ \times \| f_{k_1, j_1}\|_{Z} \| f_{k_2, j_2}\|_{Z} \lesssim  2^{\alpha k +2k_1 +l -2\max\{k_{1,-},l\} -2\alpha k_1 + 6\beta m } \epsilon_1^2 \lesssim 2^{\alpha k + (1-2\alpha )k_1+6\beta m }\epsilon_0.
\]
From the above estimate, we can rule out the case when $k_1\leq -7\beta m$ or $k\leq-70\beta m$. It remains to consider the case when $k_1\geq -7 \beta m $ and $k\geq -70\beta m.$ From the estimate (\ref{equation6861}), we have
\begin{equation}\label{equation6874}
\|  {m}_{\mu, \nu}(\xi-\eta, \eta)\varphi_{\bar{l};l}(\angle(\xi, \nu \eta))\|_{\mathcal{S}^\infty_{k,k_1,k_2}} \lesssim 2^{-8k_{1,-}+2\beta m } \lesssim 2^{58\beta m}.
\end{equation}
From the $L^2-L^\infty$ type bilinear estimate (\ref{bilinearesetimate}) in Lemma \ref{multilinearestimate} and the estimate (\ref{equation6874}), the following estimate holds after putting the input with higher spatial localization in $L^2$ and the other input in $L^\infty$.
\[
\sum_{\max\{j_1, j_2\}\geq m+k-k_1+l-4\beta m }  \sum_{i=1,2}2^{\delta j}\| \mathcal{F}^{-1}[ \textup{End}_{l;k_1,j_1,  k_2, j_2}^{\mu, \nu, i } ]\|_{B_{k,j}}
\]
\[
\lesssim \sum_{\max\{j_1, j_2\}\geq m+k-k_1+l-4\beta m }  2^{\alpha k  + (1+\delta)j +58 \beta m-\max\{j_1,j_2\}-m-2\alpha k_1}  \| f_{k_1, j_1}\|_{Z} \| f_{k_2, j_2}\|_{Z}
\]
\[
\lesssim 2^{-m-(1-\alpha)k+(1-2\alpha)k_1 + 70 \beta m}\epsilon_1^2\lesssim 2^{-m+140\beta m }\epsilon_0\lesssim 2^{-2\delta m -2\delta j}\epsilon_0.
\]

$\bullet$\quad \textit{Estimate of $J^{\mu, \nu,i}_{l;k_1, k_2}$, $i\in\{1,2\}$.\quad} Recall (\ref{equation7000}), (\ref{equation7001}), and (\ref{tridecomwithangle}). 

Since the decay rate of $Z$-norm of $\mathcal{R}'$ is   $2^{-m}$, with minor modifications, we can  estimate of $\mbox{JR}^{1}_{l;k_1,k_2}$ and $\mbox{JR}^{2}_{l;k_1,k_2}$ in the same way as we did  above for $ \textup{End}_{l;k_1,  k_2}^{\mu, \nu, i }$. We omit details here. 

From the estimate (\ref{noveqn519})  in Lemma \ref{roughestimatephase2}, it is easy to see  that the size of $\Phi^{\mu, \nu}(\xi-\eta, \eta)$ is greater than  $2^{k+2\max\{k_1,l\}} $, which is greater than $2^{-m+\beta m }$. From the estimate (\ref{equation6911}) in Lemma \ref{angularbilinear}, after putting $f_{k_i}(t)$  in $L^\infty$ and $T^{\tau, \kappa}(f_{k_1'}, f_{k_2'})$ in $L^2$, the following estimate holds 
\[
\sum_{i=1,2}2^{\delta j}\| \mathcal{F}^{-1}[J^{\mu, \nu, \tau, \kappa, i}_{l;k_1', k_2'}]\|_{B_{k,j}}\lesssim  \sup_{|\lambda|\leq 2^{\beta m/2} } 2^{\alpha k   + m + (1+\delta) j-2\max\{l,k_{1,-} \}+2\beta m} \| e^{-i (t+2^{-k-2l}\lambda)} f_{k_1}(t)\|_{L^\infty}\]
\[
\times  2^{k_1} \| e^{-it\Lambda} f_{k_1'}\|_{L^\infty} \| f_{k_2'}\|_{L^2}
+2^{-10 m +k+\alpha k+m+j-2\max\{l,k_{1,-}\} }\| f_{k_1}\|_{L^2}\| e^{-it\Lambda} f_{k_1'}\|_{L^\infty} \| f_{k_2'}\|_{L^2}
\]
\begin{equation}\label{equation6880}
\lesssim \min\{ 2^{(1-2\alpha)k_2'+3\beta m}, 2^{(1-2\alpha)k_2'+m+3\beta m -(N_0-8)k_{1,+}'}\}\epsilon_0.
\end{equation}
From the above estimate (\ref{equation6880}), we can rule out the case when $k_1', k_2'\notin [-4\beta m , \beta m ] $. 
Now, it remains  to consider fixed $k_1', k_2'\in [-4\beta m , \beta m ]$.

Firstly, we consider the case when $\max\{j_i, j_1',j_2'\}\geq m + k_1+k_1'-10\beta m $, $i\in\{1,2\}$. From the estimate (\ref{equation6911}) in Lemma \ref{angularbilinear}, the following estimate holds after putting the input with the maximal spatial concentration in $L^2$ and the other two inputs in $L^\infty$, the following estimate holds, 
  \[
\sum_{i=1,2}\sum_{\max\{j_i, j_1', j_2'\}\geq m+k_1+k_1'-10\beta m } 2^{\delta j}\| \mathcal{F}^{-1}[H_{l,j_i;j_1',j_2'}^{\mu,\nu,\tau, \kappa,i}]\|_{B_{k,j}}\]
\[\lesssim  \sum_{\max\{j_i, j_1', j_2'\}\geq m+k_1+k_1'-10\beta m  }
  2^{2\beta m} \big( 2^{\alpha k   + m +j-2\max\{l,k_{1,-} \}  }  2^{-m-\alpha k_1}+ 2^{-10 m +k+\alpha k+m+j-2\max\{l,k_{1,-}\}  } \big)\]
   \[ 
 \times  2^{k_1} 2^{-\max\{j_i,j_1', j_2'\}-m-\alpha(k_1'+k_2')}\| f_{k_1', j_1'}\|_{Z} \| f_{k_i, j_i}\|_Z \| f_{k_2',j_2'}\|_Z
\]
\begin{equation}\label{equation9880}
\lesssim 2^{-m-k_1-(1+\alpha)(k_1'+k_2') +24\beta m } \epsilon_0\lesssim 2^{-2\delta m -2\delta j}\epsilon_0.
\end{equation}
In the above estimate, we used the fact that $k_1\geq -m/3$.

Now, we consider the case when $\max\{j_i, j_1',j_2'\}\leq m + k_1+k_1'-10\beta m $, $i\in\{1,2\}$. We separate further into two cases based on the possible size of $k_1$. If $k_1\geq -10\beta m$, then the cubic degeneracy of the bad type phases doesn't cause much difference, with minor modifications, the argument used in the proof of Lemma \ref{hhgoodj7} for the good type phases also works out for this case.

Lastly, we consider the case when $k_1\leq -10\beta m $. Recall that $-4\beta m\leq k_2'\leq k_1'\leq \beta m$. In other words, we have $|\eta|\ll |\eta-\sigma|\sim |\sigma|$  or $|\xi-\eta|\ll |\xi-\sigma|\sim |\sigma-\eta|$. Hence, we can do integration by parts in $\sigma$ to take the advantage of high oscillation in $\sigma$. More precisely, the net gain of doing integration by parts in ``$\sigma$'' is  at least $\max\{2^{-m+\max\{j_1',j_2'\} -k_1-k_1'+3\beta m},$ $  2^{-m-\min\{k_1,k_2'\}-4k_{1}'+5\beta m} \}$, which is less than $2^{-\beta m}$, see estimates (\ref{dece1}) and (\ref{dece2}) in Lemma \ref{phasesize}. Therefore, we can do integration by parts in ``$\sigma$'' many times to see rapidly decay to rule out the case when $\max\{j_1', j_2'\}\leq m +k_1+k_1'-4\beta m $.  Hence finishing the proof.

\end{proof}

\subsection{When $(k_1, k_2)\in \chi_k^2, (\mu, \nu)\in \{(+,+), (+,-)\} $}\label{hlbad}

Note that the estimates (\ref{noveqn110}) and (\ref{noveqn111}) in Lemma \ref{highlowhl} and Lemma \ref{hlgoodj10} hold regardless the sign of $\mu$ and $\nu$.   Hence,  we can rule out the very-low-frequency case,  the relatively-high-frequency case and the case $j\geq (1+20\delta)m$  as in subsection \ref{lowhighgood}. From the estimate (\ref{noveqn519}) in Lemma \ref{roughestimatephase2}, it is easy to see that the phase is not degenerated when $k_{}\geq 0$. Therefore, there is little difference between the bad type phase and the good type phase and the method used in   subsection \ref{lowhighgood} also works for this case. To sum up,   it would be sufficient to consider 
 fixed  $k$,  $k_1, k_2$, and $j$ that satisfy the following estimate, 
\begin{equation}\label{badrange1}
-2(1+100\delta)m/(2-\alpha) \leq k_2\leq k-10, |k_1-k|\leq 10,
\end{equation}
\begin{equation}\label{badrange2}
\, -2(1+100\delta)m/(4-\alpha)\leq k \leq 0, \quad j\leq (1+20\delta) m.
\end{equation}

For fixed $k_1$ and $k_2$ in the above range, we do dyadic decomposition for   the angle between $\xi$ and $\nu \eta$ with the threshold $\bar{l}$ chosen to be $2k_{1,-}$ and then spatially localize two inputs as in (\ref{equation7002}).  For simplicity, we use the same notations listed in (\ref{equation6601}) and  (\ref{equation7002}) but readers should keep in mind  that now $(k_1,k_2)\in \chi_k^2$ instead of $\chi_k^1$. 

 To sum up, 
it would be sufficient to prove the following estimate in this subsection, 
\be\label{baddesired2}
  2^{\delta j}  \| \mathcal{F}^{-1}\big[
I_{l}^{\mu, \nu}\big] \|_{B_{k,j}}\lesssim 2^{-2\delta m -2\delta j}\epsilon_0,\quad   (\mu, \nu)\in\{(+,-), (+,+)\},
\ee
where  fixed  $k$,  $k_1$, and $j$ satisfy the estimates (\ref{badrange1}) and (\ref{badrange2}) and fixed $l\in [-2k_{1,-}, 2]$. Based on the possible size of $j$, $k_2$ and $l$, we separate the proof of the desired estimate (\ref{baddesired2}) into five cases, see Lemma \ref{highlowbadlemma1}, Lemma \ref{highlowbadlemma2}, Lemma \ref{highlowbadlemma3}, Lemma \ref{highlowbadlemma5}, and Lemma \ref{highlowbadlemma9}.

\begin{lemma}\label{highlowbadlemma1}
Under the bootstrap assumption \textup{(\ref{smallness})}  and the assumption that Proposition \textup{\ref{propZnorm2}} holds, if   $ \max\{m+l,\min\{-k_2-l,m\}\} + 100\delta m \leq j \leq m +20\delta m$, then the desired estimate (\textup{\ref{baddesired2}}) holds  for any fixed $k$,  $k_1$, and $k_2$ that satisfy \textup{(\ref{badrange1})} and \textup{(\ref{badrange2})}.
\end{lemma}
\begin{proof}
Recall that $j\leq m+20\delta m $. From the assumption on $j$, it is easy to see that   we only need to consider the case when $k_2+l\geq -m$, $j\geq \max\{m+l, -k_2-l\}+100\delta m$ and $k_1\leq -20\delta m$.

Recall (\ref{equation7002}). Although $k_1$ and $k_2$ are not comparable in the case we are considering, the following rough estimate always holds,
\[
 |\nabla_{\xi} \Phi^{\mu, \nu}(\xi, \eta)|\varphi_{\bar{l};l}(\angle(\xi, \nu \eta))\psi_{k_1}(\xi-\eta) \psi_{k_2}(\eta) + |\nabla_{\xi} \Phi^{\mu, \nu}(\xi, \xi-\eta)|\varphi_{\bar{l};l}(\angle(\xi, \nu (\xi-\eta)))\psi_{k_1}( \eta) \psi_{k_2}(\xi-\eta) \]
 \[
 \lesssim 2^{l}.
\]

Recall that $j\geq \max\{ m+l, -k_2-l\}+100\delta m$.  Hence,   by doing integration by parts in $\xi$ once, we gain $2^{-j}$ and pay the price of $\max\{2^{\min\{j_1,j_2\}}, 2^{-k_2-l}\}$, where $2^{-k_2-l}$ comes from the fact that $\nabla_{\xi}$ might hit the angular cutoff function or $a_{\mu, \nu}(t,x, \xi-\eta)$ (see (\ref{equation6840})). Hence,  we can rule out the case when $\min\{j_1,j_2\}\leq j-\delta m$ by doing integration by parts in $\xi$ many times. 

It remains to consider the case when $\min\{j_1,j_2\}\geq j-\delta m$, from the estimate (\ref{eqn293}) in Lemma \ref{bilinearest}, the following estimate holds, 
\[
\sum_{ \min\{j_1, j_2\}\geq j-\delta m}2^{\delta j}\|\mathcal{F}^{-1}[ I_{l;j_1, j_2}^{\mu, \nu}]\|_{B_{k,j}} \lesssim \sum_{ \min\{j_1, j_2\}\geq j-\delta m}  2^{ \alpha k  +m+(1+\delta)j+k+k_2+l/2} \|  f_{k_1,j_1}\|_{L^2}    \]
\begin{equation}\label{equation10010}
\times  \|   f_{k_2,j_2}\|_{L^2} \lesssim  2^{k+(1-\alpha)k_2+m+2\delta m -j+l/2} \epsilon_1^2   \lesssim 2^{ (1-\alpha)k_2 +2\delta m}\epsilon_0 \lesssim 2^{-2\delta m -2\delta j}\epsilon_0.
\end{equation}
Hence finishing the proof. 

\end{proof}

\begin{lemma}\label{highlowbadlemma2}
Under the bootstrap assumption \textup{(\ref{smallness})}  and the assumption that Proposition \textup{\ref{propZnorm2}} holds, if $  j\leq \max\{m+l,\min\{-k_2-l,m\}\} + 100\delta m  $, then the desired estimate (\textup{\ref{baddesired2}}) holds  for any fixed $k$,  $k_1$, and $k_2$ that satisfy \textup{(\ref{badrange1})} and \textup{(\ref{badrange2})} under the assumption that  $(2-2\alpha)k_2\leq -m-20\beta m $, and  $k_2+2l\leq -m +4\beta m $.
\end{lemma}
\begin{proof}
Note that the assumptions on $j$, $k_2$, and $l$ implies that  $j\leq \min\{ -k_2-l + 4\beta m, m\} +100\delta m$.   Since $k_1$ and $k_2$
are not comparable for the case we are considering,  whether $j_2$ is the smaller than $j_1$   makes a difference.

$\bullet$\quad  If $j_2\leq j_1$,  then   from the estimate (\ref{eqn293}) in Lemma \ref{bilinearest}, the following estimate holds, 
\[
 \sum_{ j_2\leq j_1} 2^{\delta j}\| \mathcal{F}^{-1}[I_{l;j_1,j_2}^{\mu, \nu}]\|_{B_{k,j}} \lesssim   \sum_{-k_2\leq j_2\leq j_1}  2^{2\delta m + \alpha k+6k_{+} + m +j +k_1  +k_2+l/2} \| { f_{k_1,j_1}}\|_{L^2} \| f_{k_2,j_2}\|_{L^2}
\]
\[
\lesssim 2^{12\beta m +m +k_1-l/2+(2-\alpha)k_2}
\epsilon_0  \lesssim 2^{-2\delta m -\delta j}\epsilon_0.
\]
$\bullet$\quad If $j_1\leq j_2$, then we can improve the upper bound of $j$. More precisely, as $j_1\leq j_2$, there is no need to   switch the role of $\xi-\eta$ and $\eta$. As a result, the following improved estimate holds from the estimate (\ref{eqn134}) in Lemma \ref{phasesize},
\[
 |\nabla_{\xi} \Phi^{+, \nu}(\xi, \eta)|\varphi_{\bar{l};l}(\angle(\xi, \nu \eta)) = |\Lambda'(|\xi|)\frac{\xi}{|\xi|} - \Lambda'(|\xi-\eta|)\frac{\xi-\eta}{|\xi-\eta|}|\varphi_{\bar{l};l}(\angle(\xi, \nu \eta)) \lesssim  2^{k_2-k_1+l}.
\]

With the above observation, we can  redo the argument used in  the proof of Lemma \ref{highlowbadlemma1} to further rule out the case when $\max\{m+(1-\alpha)(k_2-k_1)+l, -k_1-l\}+3\beta m \leq j \leq \max\{m+l,-k_2-l\} +100\delta m$. More precisely, 
a similar estimate as in the estimate (\ref{equation10010}) holds for  the case when $j_1\geq j-\delta m $. Recall (\ref{equation6840}). Note that the price of doing integration by parts in $\xi$ once is  $2^{-k_1-l}$ when $\nabla_{\xi}$ hits $a_{\mu, \nu}(t,x,\xi, \eta)$. Hence, by doing the integration by parts in ``$\xi$'' many times,  we can rule out the case when $j_1 \leq j-\delta m$.

Lastly, it remains to consider the case when $j \leq \max\{m+(1-\alpha)(k_2-k_1)+l, -k_1-l\}+3\beta m$. From the estimate (\ref{eqn293}) in Lemma \ref{bilinearest} and the estimates (\ref{L2estimate})  and (\ref{equation6726}) in Lemma \ref{L2estimatelemma},  the following estimate holds,
\[
\sum_{j_1\leq j_2} 2^{\delta j}\| \mathcal{F}^{-1}[I_{l;j_1,j_2}^{\mu, \nu}]\|_{B_{k,j}}\lesssim \sum_{j_1\leq j_2}   2^{2\delta m + \alpha k+ m +j +k_1   +k_1+k_2+l } \| { f_{k_1,j_1}}\|_{L^1} \| f_{k_2,j_2}\|_{L^2}
\]
\[
\lesssim \max\{2^{7\beta m +2m +(3-2\alpha)k_2+k_1 +2l}\epsilon_0, 2^{7\beta m + m+(2-\alpha)k_2+k_1}\epsilon_0 \} \lesssim 2^{-2\delta m -2\delta j}\epsilon_0.
\]
\end{proof}
\begin{lemma}\label{highlowbadlemma3}
Under the bootstrap assumption \textup{(\ref{smallness})}  and the assumption that Proposition \textup{\ref{propZnorm2}} holds, if $  j\leq \max\{m+l,\min\{-k_2-l,m\}\} + 100\delta m  $, then the desired estimate (\textup{\ref{baddesired2}}) holds  for any fixed $k$,  $k_1$, and $k_2$ that satisfy \textup{(\ref{badrange1})} and \textup{(\ref{badrange2})} under the assumption that  $(2-2\alpha)k_2\leq -m-20\beta m $, and  $k_2+2l\geq -m +4\beta m $.
\end{lemma}
 \begin{proof}

Note that the assumptions on $j$, $k_2$, and $l$ implies that $j\leq m +l +100\delta m$ and $l\geq -m/4$. Moreover, note that    we are away from the space-resonance in ``$\eta$'' set since $k_2+2l\geq -m +4\beta m $.

 We separate into two cases based on whether $j_1$ is   smaller than $j_2$ as follows. 

$\bullet$\quad  We first consider the case when $j_2\leq j_1$.  Note that the net gain of doing integration by parts in $\eta$ once is at least $2^{-m}\max\{2^{\max\{j_1,j_2\}-l+3\beta m}, 2^{-k_2-2l+\beta m}\}$, which is less than $2^{-\beta m}$ if  $j_1\leq m + l-4\beta m$. Hence, we can first rule out the case $j_1\leq m + l-4\beta m$  by doing integration by parts in ``$\eta$'' many times. From (\ref{eqn293}) in Lemma \ref{bilinearest}, the following estimate holds for the case   $j_1\geq  m + l-4\beta m $
\[
\sum_{j_2\leq j_1, m+l-4\beta m \leq j_1} 2^{\delta j}\| \mathcal{F}^{-1}[I_{l;j_1,j_2}^{\mu, \nu}]\|_{B_{k,j}} \lesssim  \sum_{j_2\leq j_1, m+l-4\beta m \leq j_1}  2^{2\delta m + \alpha k   + m + j +k_1  +2k_2+l} 
\]
\[
\times \| f_{k_1,j_1}\|_{L^2} \| f_{k_2, j_2}\|_{L^1} \lesssim 2^{m+(2-\alpha)k_2+k_1+l+6\beta m}\epsilon_0 \lesssim 2^{-2\delta m -2\delta j}\epsilon_0. 
\]

$\bullet$\quad Lastly, we consider the case when $j_1\leq j_2$. For this case, we can  improve the upper bound for $j$. More precisely, following the same argument used in the proof of Lemma \ref{highlowbadlemma2}, we can rule out the case when  $\max\{m+(1-\alpha)(k_2-k_1)+l, -k_1-l\}+3\beta m \leq j \leq \max\{m+l,-k_2-l\} +100\delta m$. 

It remains  to consider the case  when $j \leq \max\{m+(1-\alpha)(k_2-k_1)+l, -k_1-l\}+3\beta m$. Moreover, same as in the case $j_2\leq j_1$ considered previously, we can further rule out the case when $j_2\leq m+l -4\beta m$ by  doing integration by parts in ``$\eta$'' many times to see rapidly decay. 

To sum up, it would be sufficient to consider the case when $j \leq \max\{m+(1-\alpha)(k_2-k_1)+l, -k_1-l\}+3\beta m$ and $j_2\geq m+l -4\beta m$. From the estimate   (\ref{eqn293}) in Lemma \ref{bilinearest} and the estimate  (\ref{sizeofsymboluniform}) in Lemma \ref{sizeofsymbol}, we derive the following estimate,
 \[
 \sum_{j_1\leq j_2, m+l-4\beta m \leq j_2}  2^{\delta j}\| \mathcal{F}^{-1}[I_{l;j_1,j_2}^{\mu, \nu}]\|_{B_{k,j}} \lesssim  \sum_{j_1\leq j_2, m+l-4\beta m \leq j_2}  2^{2\delta m + \alpha k   + m + j +k_1  + k_1+k_2+l}
 \]
\[
\times  \| f_{k_1,j_1}\|_{L^1} \| f_{k_2,j_2}\|_{L^2}\lesssim \max\{2^{m + (2-2\alpha)k_2 +k_1+l+12\beta m }, 2^{12\beta m + (1-\alpha)k_2 +k_1-l}\} \epsilon_1^2 \lesssim 2^{-2\delta m -2\delta j}\epsilon_0. 
\]
In the above estimate, we used the fact that $k_2\leq -m/(2-2\alpha)-12\beta m$ and $l \geq -m/4$. Hence finishing the proof.

\end{proof}

\begin{lemma}\label{highlowbadlemma5}
Under the bootstrap assumption \textup{(\ref{smallness})}  and the assumption that Proposition \textup{\ref{propZnorm2}} holds, if $  j\leq \max\{m+l,\min\{-k_2-l,m\}\} + 100\delta m  $, then the desired estimate (\textup{\ref{baddesired2}}) holds  for any fixed $k$,  $k_1$, and $k_2$ that satisfy \textup{(\ref{badrange1})} and \textup{(\ref{badrange2})} under the assumption that  $(2-2\alpha)k_2\geq -m-20\beta m $, and  $k_2+2l\leq -m +4\beta m $.
\end{lemma}
\begin{proof}

Note that the assumptions on $j$, $k_2$, and $l$ implies that $j\leq -k_2-l +100\delta m+4\beta m $, $l\leq -m/5$ and $k_1\leq - m/10$. For this case,  we first do integration by parts in time. As a result, we have the same equality as in (\ref{equation7000}). 

From the estimate (\ref{eqn61001}) in  Lemma \ref{Snorm}, the estimate (\ref{sizeofsymboluniform}) in Lemma \ref{sizeofsymbol}, the estimate (\ref{noveqn519}) in Lemma \ref{roughestimatephase2}, and the estimates  (\ref{dece1}) and (\ref{dece2}) in Lemma \ref{phasesize}, the following estimates hold, 
\begin{equation}\label{equation7010}
\| m_{\mu, \nu}(\xi-\eta, \eta) \varphi_{\bar{l};l}(\angle(\xi, \nu \eta))\|_{L^{\infty}_{\xi, \eta}} \lesssim 2^{k_1-k_2-2\max\{k_{1}, l\} },
\end{equation}
\[
 \| m_{\mu, \nu}(\xi-\eta, \eta) \varphi_{\bar{l};l}(\angle(\xi, \nu \eta))\|_{\mathcal{S}^{\infty}_{k,k_1,k_2}} \lesssim \max\{2^{k_1-k_2-2\max\{k_{1,-}, l\}  - 3k_{1,-}},
\]
\begin{equation}\label{equation8010}
 2^{k_1-k_2-2\max\{k_{1,-}, l\}  - 3l} \}  \lesssim 2^{k_1-k_2-2\max\{k_{1,-}, l\}-6k_{1,-}  }.
\end{equation}

$\bullet$\quad \textit{Estimate of the endpoint case.}\quad Recall (\ref{noveqn810}).  From the estimate (\ref{eqn293}) in Lemma \ref{bilinearest}, the estimates (\ref{L2estimate}) and (\ref{equation6726}) in Lemma \ref{L2estimatelemma}, and the estimate (\ref{equation7010}), the following estimate holds, 
\[
\sum_{i=1,2}2^{\delta j}\| \mathcal{F}^{-1}[ \textup{End}_{l;k_1, k_2}^{\mu, \nu, i } ]\|_{B_{k,j}}
\lesssim \sum_{i=1,2} 2^{2\delta m + \alpha k   + j-k_1-k_2   + k_1+k_2+l} \| \widehat{f_{k_1}}(t_i,\xi)\|_{L^\infty_\xi} \|f_{k_2}(t_i)\|_{L^2}
\]
\[
\lesssim 2^{\alpha k +4\beta m+200\delta m}\epsilon_0 \lesssim 2^{-2\delta m-2\delta j}\epsilon_0.
\]

\noindent $\bullet$ \quad\textit{Estimate of $J^{\mu, \nu,i}_{l;k_1, k_2}$, $i\in\{1,2\}$.\quad} Recall (\ref{equation7000}).  With minor modifications, we can  estimate of $\mbox{JR}^{1}_{l;k_1,k_2}$ and $\mbox{JR}^{2}_{l;k_1,k_2}$ in the same way as we did for $ \textup{End}_{l;k_1,  k_2}^{\mu, \nu, i }$. We omit details here and proceed to the estimate of $J^{\mu, \nu, \tau, \kappa, i}_{l;k_1', k_2'}$. From (\ref{eqn293}) in Lemma \ref{bilinearest}, (\ref{sizeofsymboluniform}) in Lemma \ref{sizeofsymbol},  (\ref{L2estimate}) and (\ref{equation6726}) in Lemma \ref{L2estimatelemma}, and (\ref{equation7010}),  the following estimate holds after putting  $T^{\tau, \kappa}(f_{k_1'}, f_{k_2'})$ in $L^2$ and the other one in $L^\infty_\xi$, 
\[
\sum_{i=1,2} 2^{\delta j}\| \mathcal{F}^{-1}[J^{\mu, \nu, \tau, \kappa, i}_{l;k_1', k_2'}]\|_{B_{k,j}}\lesssim 2^{2\delta m + \alpha k +m +j -k_1-k_2 +k_1+k_2+l} \|\widehat{f_{k_1}}(t,\xi)\|_{L^\infty_\xi} 2^{k_2+k_{1,+}'} \| e^{-it\Lambda} f_{k_1'}\|_{L^\infty}
\]
\[
\times  \| f_{k_2'}\|_{L^2} + 2^{2\delta m + \alpha k  +m +j -k_1-k_2 +2k_2+l} \|\widehat{f_{k_2}}(t,\xi)\|_{L^\infty_\xi} 2^{k_1+k_{1,+}'} \| e^{-it\Lambda} f_{k_1'}\|_{L^\infty} \| f_{k_2'}\|_{L^2}
\]
\begin{equation}\label{equation7030}
\lesssim \min\{  2^{\alpha k   +(1-\alpha) k_{2}' -4k_{1,+}'+3\beta m}, 2^{m+\alpha k +(1-\alpha) k_{2}' +k_{1}' -N_0k_{1,+}'+3\beta m}\}\epsilon_0.
\end{equation}
 From the  above estimate, we can rule out the case when $k_2'\leq -10\beta m $ or $k_1'\geq\beta m$.

  It remains to consider fixed $k_1'$ and $k_2'$ such that $-10\beta m\leq k_2'\leq k_1'\leq \beta m.$  Recall that $k_1\leq -m/10$. In other words, we have  $|\eta|\ll |\eta-\sigma|\sim |\sigma|$ or $|\xi-\eta|\ll |\xi-\sigma|\sim |\sigma-\eta|$. From the estimates (\ref{dece1}) and  (\ref{dece2}) in Lemma \ref{phasesize}, we know that   $\nabla_\sigma \Phi^{\mu, \tau, \kappa}_i(\xi, \eta, \sigma)$ always has a good lower bound. 
Hence, we can rule out the case when $\max\{j_1',j_2'\}\leq m + k_2+k_1'-4\beta m$ by doing integration by parts in ``$\sigma$'' many times.   It would be sufficient to consider the case when $\max\{j_1',j_2'\}\geq m + k_2+k_1'-4\beta m$.   From (\ref{eqn293}) in Lemma \ref{bilinearest}, the following estimate holds  after first putting  $T^{\tau, \kappa}(f_{k_1',j_1'}, f_{k_2',j_2'})$ in $L^2$ and then putting the input with higher spatial localization in $L^2$ and the other input in $L^\infty_x$,
  \[
\sum_{i=1,2}\sum_{\max\{j_1, j_2\}\geq m+k_2+k_1'-4\beta m }2^{\delta j}\| \mathcal{F}^{-1}[J^{\mu, \nu, \tau, \kappa, i}_{l;k_1',j_1', k_2',j_2'}]\|_{B_{k,j}}\lesssim \sum_{\max\{j_1, j_2\}\geq m+k_2+k_1'-4\beta m } 
2^{2\delta m + \alpha k +m +j} \]
\[
\times 2^{ -k_1-k_2 +k_1+k_2+l} \|\widehat{f_{k_1}}(t,\xi)\|_{L^\infty_\xi} 2^{k_2+k_{1,+}'} 2^{-m-\alpha k_1'-\alpha k_2'-\max\{j_1',j_2'\}} \| f_{k_1',j_1'}\|_Z \| f_{k_2',j_2'}\|_Z
\]
\[
+ 2^{\alpha k  +m +j -k_1-k_2 +2k_2+l} \|\widehat{f_{k_2}}(t,\xi)\|_{L^\infty_\xi} 2^{k_1+k_{1,+}'} 2^{-m-\alpha k_1'-\alpha k_2'-\max\{j_1',j_2'\}}\| f_{k_1',j_1'}\|_Z \| f_{k_2',j_2'}\|_Z
\]
\[
\lesssim 2^{-m-(1+\alpha)k_2-(1+2\alpha) k_2'+4\beta m}\epsilon_0\lesssim 2^{-2\delta m-2\delta j}\epsilon_0.
\]
Hence finishing the proof. 
\end{proof}

\begin{lemma}\label{highlowbadlemma9}
Under the bootstrap assumption \textup{(\ref{smallness})}  and the assumption that Proposition \textup{\ref{propZnorm2}} holds, if $  j\leq \max\{m+l,\min\{-k_2-l,m\}\} + 100\delta m  $, then the desired estimate (\textup{\ref{baddesired2}}) holds  for any fixed $k$,  $k_1$, and $k_2$ that satisfy \textup{(\ref{badrange1})} and \textup{(\ref{badrange2})} under the assumption that  $(2-2\alpha)k_2\geq -m-20\beta m $, and  $k_2+2l\geq -m +4\beta m $.
\end{lemma}

\begin{proof}

Note that the assumptions on $j$, $k_2$, and $l$ implies that  $j\leq m +l +100\delta m $. Moreover, note that   we are away from the space-resonance in ``$\eta$'' set since $k_2+2l\geq -m +4\beta m $. For this case, we do integration by parts in time once and have the same identity as in (\ref{equation7000}).

$\bullet$\quad \textit{Estimate of the endpoint case.}\quad  Recall (\ref{noveqn810}). We separate into two cases based on whether $j_1$ is smaller than $j_2$ as follows.

$(i)$  If $j_2\leq j_1$, then  we can first rule out the case when $j_1\leq m + l -4\beta m$  by doing integration by parts in ``$\eta$'' many times. It would be sufficient to consider the case when $j_1\geq m +l -4\beta m$. From (\ref{eqn293}) in Lemma \ref{bilinearest}, the following estimate holds if $k_2\leq -30\beta m$
\[
\sum_{j_2\leq j_1, m+l-4\beta m\leq j_1} \sum_{i=1,2}2^{\delta j}\| \mathcal{F}^{-1}[ \textup{End}_{l;k_1,j_1, k_2,j_2}^{\mu, \nu, i } ]\|_{B_{k,j}}
\lesssim \sum_{j_2\leq j_1, m+l-4\beta m\leq j_1} 2^{2\delta m + \alpha k  +j +k_1-k_2-2\max\{l,k_{1  }\}}  
\]
\[
\times 2^{2\beta m + 2k_2+l} \| f_{k_1,j_1}\|_{L^2} \| f_{k_2, j_2}\|_{L^1}\lesssim 2^{(1-\alpha)k_2+15\beta m }\epsilon_0\lesssim 2^{-2\delta m -2\delta j}\epsilon_0.
\]

If $k_2\geq -30\beta m$, then  the the following estimate holds from    estimate (\ref{equation8010}), $L^2-L^\infty$ type bilinear estimate (\ref{bilinearesetimate}) in Lemma \ref{multilinearestimate} and (\ref{equation8010}) to derive the following estimate, 
\[
\sum_{j_2\leq j_1, m+l-4\beta m\leq j_1} \sum_{i=1,2}2^{\delta j}\| \mathcal{F}^{-1}[ \textup{End}_{l;k_1,j_1, k_2,j_2}^{\mu, \nu, i } ]\|_{B_{k,j}}
\lesssim \sum_{j_2\leq j_1, m+l-4\beta m\leq j_1}  2^{ \beta  m+ \alpha k  +j -k_2-7k_{1 } }
\]
\[
\times \| f_{k_1,j_1}\|_{L^2} \| e^{-it\Lambda} f_{k_2,j_2}\|_{L^\infty}  \lesssim 2^{-m-(1+\alpha)k_2-7k_{1 } +15\beta m }\epsilon_1^2  \lesssim 2^{-2\delta m-2\delta j}\epsilon_0.
\]

$(ii)$  If $j_1\leq j_2$, then we can first  rule out the case when $\max\{m+(1-\alpha)(k_2-k_1)+l, -k_1-l\}+3\beta m \leq j\leq \max\{m+l, -k_2-l\}+100\delta m $ by redoing   the argument used in the proof of Lemma \ref{highlowbadlemma2}. Moreover, by doing integration by parts in ``$\eta$'' many times, we can further rule out the case when $j_2\leq m+l-4\beta m.$ Therefore, it is sufficient to consider the case when $j\leq \max\{m+(1-\alpha)(k_2-k_1)+l, -k_1-l\}+3\beta m= m+(1-\alpha)(k_2-k_1)+l+3\beta m $ and $j_2\geq m + l -4\beta m .$ 
From (\ref{eqn293}) in Lemma \ref{bilinearest}, the following estimate holds if $k_2\leq -30\beta m$
\[
\sum_{j_1\leq j_2, m+l-4\beta m\leq j_2} \sum_{i=1,2}2^{\delta j}\| \mathcal{F}^{-1}[ \textup{End}_{l;k_1,j_1, k_2,j_2}^{\mu, \nu, i } ]\|_{B_{k,j}}
\lesssim \sum_{j_1\leq j_2, m+l-4\beta m\leq j_2} 2^{\beta m + \alpha k  +j }  
\]
\[
\times 2^{ k_1-k_2-2\max\{l,k_{1,-}\} + 2\beta m + k_1+k_2+l} \| f_{k_1,j_1}\|_{L^1} \| f_{k_2, j_2}\|_{L^2}\lesssim  2^{(1-2\alpha)k_2+10\beta m }\epsilon_0\lesssim 2^{-2\delta m -\delta j}\epsilon_0.
\]

If $k_2\geq -30\beta m$, then the following estimate holds from  the $L^2-L^\infty$ type bilinear estimate (\ref{bilinearesetimate}) in Lemma \ref{multilinearestimate} and the estimate (\ref{equation8010}), 
\[
\sum_{j_1\leq j_2, m+l-4\beta m\leq j_2} \sum_{i=1,2}2^{\delta j}\| \mathcal{F}^{-1}[ \textup{End}_{l;k_1,j_1, k_2,j_2}^{\mu, \nu, i } ]\|_{B_{k,j}}
\lesssim \sum_{j_1\leq j_2, m+l-4\beta m\leq j_2}  2^{\beta m + \alpha k +j  -k_2-7k_{1,-} +2\beta m }
\]
\[
\times \| f_{k_2,j_2}\|_{L^2} \| e^{-it\Lambda} f_{k_1,j_1}\|_{L^\infty} \lesssim 2^{-m-(1+\alpha)k_2-7k_1+15\beta m }\epsilon_1^2  \lesssim 2^{-2\delta m-2\delta j}\epsilon_0.
\]

 $\bullet$ \quad\textit{Estimate of $J^{\mu, \nu,i}_{l;k_1, k_2}$, $i\in\{1,2\}$.\quad} Same as before, we omit details for the estimates of $\mbox{JR}^{1}_{l;k_1,k_2}$ and $\mbox{JR}^{2}_{l;k_1,k_2}$ here and proceed to the estimate of $J^{\mu, \nu, \tau, \kappa, i}_{l;k_1', k_2'}$ directly.  From the estimate (\ref{eqn293}) in Lemma \ref{bilinearest}, it is easy to see that the following estimates hold,
 \[
2^{\delta j}\| \mathcal{F}^{-1}[J^{\mu, \nu, \tau, \kappa,1}_{l;k_1', k_2'}]\|_{B_{k,j}}\lesssim 2^{2\delta m + \alpha k +  m +j+k_1-k_2-2\max\{l, k_{1,-}\} +k_1+k_2+l}
\]
\begin{equation}\label{equation10000}
 \times \| \widehat{f_{k_1}}(t,\xi)\|_{L^\infty_\xi} 2^{k_2} \| e^{-it\Lambda} f_{k_1'}\|_{L^\infty} \| f_{k_2'}\|_{L^2}, 
\end{equation}
 \[
2^{\delta j}\| \mathcal{F}^{-1}[J^{\mu, \nu, \tau, \kappa, 2}_{l;k_1', k_2'}]\|_{B_{k,j}}\lesssim 2^{2\delta m + \alpha k  + m +j+k_1-k_2-2\max\{l, k_{1,-}\} +2k_2+l}
\]
\begin{equation}\label{equation10001}
 \times \| \widehat{f_{k_2}}(t,\xi)\|_{L^\infty_\xi} 2^{k_1} \| e^{-it\Lambda} f_{k_1'}\|_{L^\infty} \| f_{k_2'}\|_{L^2}.
\end{equation}
From (\ref{equation10000}) and (\ref{equation10001}), it is easy to see that the following estimate holds if $k_1',k_2'\notin [-2m , 2\beta m] $,  
\[
\sum_{i=1,2}2^{\delta j}\| \mathcal{F}^{-1}[J^{\mu, \nu, \tau, \kappa, i}_{l;k_1', k_2'}]\|_{B_{k,j}}  \lesssim \min\{ 2^{m +(1-\alpha)k_2'+10\beta m }, 2^{2m+(1-\alpha)k_2'+10\beta m-(N_0-8)k_{1,+}'}\}\epsilon_1^2
\]
\be\label{noveqn841}
\lesssim 2^{-2\delta m -2\delta j}\epsilon_0.
\ee
From the estimate (\ref{noveqn993}) in Lemma \ref{hlowcubiczlemma1} and the estimate (\ref{noveqn862}) in Lemma \ref{hlowcubiczlemma2}, it is easy to see that the desired estimate (\ref{noveqn841}) also holds if  $k_1',k_2'\in [-2m , 2\beta m] $. Hence finishing the proof.  
\end{proof}
\begin{lemma}\label{hlowcubiczlemma1}
Under the bootstrap assumption \textup{(\ref{smallness})}  and the assumption that Proposition \textup{\ref{propZnorm2}} holds, if $  j\leq \max\{m+l,\min\{-k_2-l,m\}\} + 100\delta m  $,   $k$,  $k_1$, and $k_2$   satisfy \textup{(\ref{badrange1})} and \textup{(\ref{badrange2})},  $(2-2\alpha)k_2\geq -m-20\beta m $, and  $k_2+2l\geq -m +4\beta m $, $ k_1', k_2' \in [-2m, 2\beta m ]$, then the following estimate holds, 
\be\label{noveqn993}
 2^{\delta j}\| \mathcal{F}^{-1}[J^{\mu, \nu, \tau, \kappa, 1}_{l;k_1', k_2'}]\|_{B_{k,j}}  \lesssim 2^{-2\delta m -2\delta j}\epsilon_0. 
\ee
\end{lemma}
\begin{proof}
Recall (\ref{equation7001}) and (\ref{noveqn811}). Note that $(k_1',k_2')\in \chi_{k_2}^1\cup \chi_{k_2}^2\cup \chi_{k_2}^3$ and $j\leq m+l+100\delta m $.

\textbf{Case $1$:\quad}  If $k_2'-3\beta m \leq k_2$. From the estimate (\ref{equation6911}) in Lemma \ref{angularbilinear} and the estimate (\ref{L2estimate}) in Lemma \ref{L2estimatelemma}, after putting $T^{\nu\tau, \nu\kappa}(f_{k_1'}, f_{k_2'})$ in $L^2$ and the other input in $L^\infty$,  the following estimate holds,  
\[
 2^{\delta j}\| \mathcal{F}^{-1}[J^{\mu, \nu, \tau, \kappa, 1}_{l;k_1', k_2'}]\|_{B_{k,j}} \lesssim  \sup_{|\lambda|\leq 2^{\beta m}} 2^{2\delta m + \alpha k + m +j+k_1-k_2-2\max\{l,k_{1,-} \}}2^{(k_1-k_2)/2}  2^{k_2} \| e^{-it\Lambda} f_{k_1'}\|_{L^\infty}\]
\[
\times  \| f_{k_2'}\|_{L^2} \| e^{-i (t+2^{-k_2-2\max\{l, k_{1,-}\}} \lambda)} f_{k_1}(t)\|_{L^\infty} 
+2^{-10 m +k+\alpha k+m+j+k_1-k_2-2\max\{l,k_{1,-}\} +k_2}
\]
\[
\times \| f_{k_1}\|_{L^2}\| e^{-it\Lambda} f_{k_1'}\|_{L^\infty} \| f_{k_2'}\|_{L^2}\lesssim  2^{k_2'-\alpha k_1'+(k_1-k_2)/2+\beta m/2} \epsilon_0
\lesssim  2^{(1-2\alpha)k_2'/2 +2\beta m} \epsilon_0.
\]
From above estimate, we can rule out the case when $k_2'\leq -7\beta m$. 

It remains to consider the case when $k_2'\geq -7\beta m$. As $k_2\geq k_2'-3\beta m $, we have $k_2\geq -10\beta m $. That is to say, all frequencies are relatively large, which implies that  the cubic degeneracy of the phases is not an issue. Recall that $|\eta|\ll |\xi-\eta|\sim |\xi|$. From the estimates (\ref{dece1}) and (\ref{dece2}) in Lemma \ref{phasesize}, it is easy to verify that  $\nabla_\sigma\Phi^{\mu,\tau, \kappa}_1(\xi, \eta,\sigma)$ is bounded from below by $2^{k_2+k_2'-4\beta m }$ when $\sigma$ is away from $\eta/2$  and  $\nabla_\eta\Phi^{\mu,\tau, \kappa}_1(\xi, \eta,\sigma)$ is bounded from below by $2^{k_2+k_2'-4\beta m }$ when $\sigma$ is close to  $\eta/2$. As a result, we can do integration by parts in $\sigma$ and $\eta$ many times respectively to rule out the case when $\max\{j_1, j_1',j_2'\}\leq m + k_2+k_2'-10\beta m $.

Now, it's sufficient to consider the case when $\max\{j_1,j_1', j_2'\}\geq m+  k_2+k_2'-10\beta m$.  From the estimate (\ref{equation6911}) in Lemma \ref{angularbilinear} and the  estimate (\ref{sizeofsymboluniform}) in Lemma \ref{sizeofsymbol}, the following estimate holds after putting the input with the maximum spatial concentration in $L^2$ and the other two inputs in $L^\infty$, 
\[
\sum_{\max\{j_1,j_1', j_2'\}\geq m+k_2+k_2'-10\beta m} 2^{\delta j} \| \mathcal{F}^{-1}[H_{l,j_1;j_1',j_2'}^{\mu,\nu,\tau, \kappa,1}]\|_{B_{k,j}}  \lesssim \sum_{\max\{j_1,j_1', j_2'\}\geq m+ k_2+k_2'-10\beta m}   2^{2\delta m + \alpha k +m +j} \]
\[
\times 2^{k_1-k_2-2\max\{l,k_{1,-} \}} 2^{(k_1-k_2)/2} 2^{-m-\alpha k_1} 2^{k_2-m-\max\{j_1,j_1',j_2'\} -2\alpha k_2'} \|f_{k_1', j_1'}\|_Z \|f_{k_2',j_2'}\|_Z  \| f_{k_1,j_1}\|_Z\]
\[
+2^{-10 m +k +m+j+k_1-k_2-2\max\{l,k_{1,-}\} +k_2}\| f_{k_1,j_1}\|_{Z} 2^{-m-\max\{j_1,j_1',j_2'\}-2\alpha k_2'} \|f_{k_1', j_1'}\|_Z \|f_{k_2',j_2'}\|_Z
\]
\[
\lesssim 2^{-m-3k_2'+200\beta m }\epsilon_0\lesssim 2^{-2\delta m -2\delta j}\epsilon_0.
\]

\textbf{Case $2$:\quad} If $k_2'-3\beta m \geq k_2$ and $(k_1', k_2',\nu\tau, \nu\kappa)\in \mathcal{P}_{good}^{k_2}$. Note that the assumption $k_2'-3\beta m \geq k_2$ implies that $(k_1',k_2')\in \chi_{k_2}^1$.

  Recall (\ref{neweqn1}). A key observation for this case is  that the phase  $\Phi_1^{\mu, \tau, \kappa}(\xi, \eta, \sigma)$ is relatively large. More precisely, from the estimate  (\ref{hhchi1}) and the estimate (\ref{dece31}) in Lemma \ref{roughestimatephase2}, we have
\[
2^{k_2'-k_{2,+}'/2}\lesssim 2^{k_2'-k_{2,+}'/2} - 2^{k_2+2\max\{l, k_{1,-}\}  } \leq  | \Phi_{1}^{\mu, \nu, \tau, \kappa}(\xi, \eta, \sigma)| \]
\[ \leq 2^{k_2'-k_{2,+}'/2}  + 2^{k_2+2\max\{l, k_{1,-}\}}\lesssim 2^{k_2'-k_{2,+}'/2}.
\]
Hence, we can take the advantage of the above fact by   doing integration by parts in time again. As a result, we have
\[
J^{\mu, \nu,\tau, \kappa, 1}_{l;k_1',k_2'}= \sum_{i=1,2} (-1)^{i} \textup{E}_i + \textup{H}_1, \quad  \textup{E}_i= \int_{\R^2} \int_{\R^2} e^{i t_i \Phi_1^{\mu, \tau, \kappa}(\xi, \eta, \sigma)} \widetilde{m}^{\tau, \kappa}_{\mu, \nu,1}(\xi, \eta, \sigma) \varphi_{\bar{l};l}(\angle(\xi, \nu\eta)) \widehat{f_{k_1}^{\mu}}(t_i,\xi-\eta) 
\]
\[\times \widehat{
f_{k_1'}^{\tau}
}(t_i,\eta-\sigma)\widehat{f_{k_2'}^{\kappa}}(t_i,\sigma)  d \eta d\sigma, \quad \textup{H}_1= - \int_{t_1}^{t_2}\int_{\R^2} \int_{\R^2} e^{i t \Phi_1^{\mu, \tau, \kappa}(\xi, \eta, \sigma)}   \]
\begin{equation}\label{equation7082}
\times \widetilde{m}^{\tau, \kappa}_{\mu, \nu,1}(\xi, \eta, \sigma)  \varphi_{\bar{l};l}(\angle(\xi, \nu\eta)) \p_t\big(\widehat{f_{k_2}^{\mu}}(t,\xi-\eta) \widehat{
f_{k_1'}^{\tau}
}(t,\eta-\sigma)\widehat{f_{k_2'}^{\kappa}}(t,\sigma) \big) d \eta d\sigma   d t,
\end{equation}
where
\[
\widetilde{m}^{\tau, \kappa}_{\mu, \nu,1}(\xi, \eta, \sigma) = \frac{{m}^{\tau, \kappa}_{\mu, \nu,1}(\xi, \eta, \sigma)}{i \Phi_1^{\mu, \nu,\tau, \kappa}(\xi, \eta,\sigma)}= \frac{ q_{\mu, \nu}(\xi-\eta, \eta)}{-\Phi^{\mu, \nu}(\xi, \eta)}\frac{ (q_{\nu\tau, \nu\kappa}(\eta-\sigma, \sigma))^{\nu}}{\Phi_1^{\mu, \nu,\tau, \kappa}(\xi, \eta,\sigma)}\]
\[=\frac{ q_{\mu, \nu}(\xi-\eta, \eta)}{-\Phi^{\mu, \nu}(\xi, \eta)}\frac{ (q_{\tau\nu, \kappa\nu}(\eta-\sigma, \sigma))^{\nu}}{\Phi_1^{\mu, \nu,\tau, \kappa}(\xi, \eta,\sigma)} \psi_{[-10,10]}(2^{-\kappa_1} \Phi^{\mu, \nu}(\xi, \eta)) \psi_{[-10,10]}(2^{-\kappa_2} \Phi_1^{\mu, \nu, \tau, \kappa}(\xi, \eta, \sigma)), \]
where
\[ \kappa_1:=k_2+2\max\{k_1,l\},  \quad \kappa_2 = {k_2'-k_{2,+}'/2}
.
\]

 Using the inverse Fourier transform twice, we have
\[
 \textup{E}_i= \frac{1}{16 \pi^4} \int_{\R} \int_{\R} \int_{\R^2} \int_{\R^2} 2^{-\kappa_2-\kappa_1} \widehat{\chi}(\lambda_2) \widehat{\chi}(\lambda_1)e^{i(t_i + 2^{-\kappa_2} \lambda_2) \Phi^{\mu, \nu, \tau, \kappa}_1(\xi,\eta, \sigma) + i 2^{-\kappa_1}\lambda_1\Phi^{\mu, \nu}(\xi, \eta)} \varphi_{\bar{l};l}(\angle(\xi, \nu\eta))\]
\[\times q_{\mu, \nu}(\xi-\eta, \eta)  \big(q_{\tau\nu, \kappa\nu}(\eta-\sigma, \sigma)\big)^\nu \widehat{f_{k_1}^{\mu}}(t_i,\xi-\eta) \widehat{
f_{k_1'}^{\tau}
}(t_i,\eta-\sigma)\widehat{f_{k_2'}^{\kappa}}(t_i,\sigma)  d \eta d\sigma d \lambda_2 d\lambda_1
\]
\[
= \frac{1}{16 \pi^4} 2^{-\kappa_1-\kappa_2}\int_{\R} \int_{\R} \int_{\R^2}\widehat{\chi}(\lambda_1)\widehat{\chi}(\chi_2) e^{i(t_i + 2^{-\kappa_2}\lambda_2+2^{-\kappa_1}\lambda_1)\Phi^{\mu, \nu}(\xi, \eta)+i\nu(t_i + 2^{-\kappa_2}\lambda_2
)\Lambda(\eta)} \varphi_{\bar{l};l}(\angle(\xi, \nu\eta))
\]
\[
\times q_{\mu, \nu}(\xi-\eta, \eta) \widehat{f_{k_1}^{\mu}}(\xi-\eta) T^{\tau,\kappa}_{ \lambda_2}(f_{k_1'}, f_{k_2'})(\eta) d \eta d\lambda_1 d\lambda_2, 
\]
where
\[
\widehat{\chi}(\lambda) = \int e^{-i \lambda x} \frac{\psi_{[-10,10]}(x)}{x} d x, \]
\[  T^{\tau,\kappa}_{ \lambda_2}(f_{k_1'}, f_{k_2'})(\eta)  = \int_{\R^2} e^{-i(t_i + 2^{-\kappa_2}\lambda_2)(\tau \Lambda(|\eta-\sigma|) +\kappa \Lambda(|\sigma|))} (q_{\tau\nu, \kappa\nu}(\eta-\sigma, \sigma))^{\nu}  \widehat{
f_{k_1'}^{\tau}
}(t_i,\eta-\sigma)\widehat{f_{k_2'}^{\kappa}}(t_i,\sigma)  d\sigma.
\]
Using the rapidly decay property of $\widehat{\chi}(\lambda)$, very similar to the proof of (\ref{equation6910}) in Lemma \ref{angularbilinear}, we can derive the following estimate, 
\[
2^{\delta j}\| \mathcal{F}^{-1}[ \textup{E}_i]\|_{B_{k,j}} \lesssim \sup_{|\lambda_1|, |\lambda_2|\leq 2^{\beta m/10}} 2^{2\delta m + \alpha k   + j+k_1 -\kappa_1 -\kappa_2} 2^{(k_1-k_2)/2} \| e^{i(t_i + 2^{-\kappa_1}\lambda_1+2^{-\kappa_2}\lambda_2)\Lambda} f_{k_1}\|_{L^\infty}
\]
\[
\times  \| T^{\tau,\kappa}_{ \lambda_2}(f_{k_1'}, f_{k_2'})(\eta)\|_{L^2}+
 2^{-10 m + k_1+k_2-\kappa_1-\kappa_2 +2k_2+2k_2'} \| \widehat{f_{k_1}}\|_{L^\infty_\xi} \| \widehat{f_{k_1'}}\|_{L^\infty_\xi}\| \widehat{f_{k_2'}}\|_{L^\infty_\xi}
\]
\[
\lesssim \sup_{|\lambda_1|, |\lambda_2|\leq 2^{\beta m/10}} 2^{-3k_2/2+k_1/2-k_2'+k_2 +10\beta m} \| e^{i(t_i +2^{-\kappa_2\lambda_2})\Lambda} f_{k_1'}\|_{L^\infty} \| f_{k_2'}\|_{L^2} + 2^{-2\delta m -2\delta j}\epsilon_0
\]
\be\label{neweqn450}
\lesssim 2^{-m-(1+2\alpha)k_2/2+k_1/2+13\beta m }\epsilon_0 + 2^{-2\delta m -2\delta j}\epsilon_0 \lesssim 2^{-2\delta m -2\delta j}\epsilon_0.
\ee
In the above estimate, we used the fact that $k_2\geq -m/(2-2\alpha)-12\beta m $ and also used the estimate (\ref{L2estimate}) in Lemma \ref{L2estimatelemma} and  (\ref{sizeofsymboluniform}) in Lemma \ref{sizeofsymbol}.
 
With minor modifications, we can   estimate   $\textup{H}_1$  very similarly. From (\ref{eqn293}) in Lemma \ref{bilinearest}, and (\ref{L2derivativeestimate}) in Lemma \ref{L2estimatelemma},  the following estimate holds if $k_2\leq -10\beta m $
\[
2^{\delta j}\| \mathcal{F}^{-1}[ \textup{H}_1]\|_{B_{k,j}} \lesssim\sup_{|\lambda_1|, |\lambda_2|\leq  2^{\beta m/10}} 2^{2\delta m }\big[ 2^{\alpha k  +m+ j+k_1 -\kappa_1 -\kappa_2} 2^{(k_1-k_2)/2} \| e^{i(t_i + 2^{-\kappa_1}\lambda_1+2^{-\kappa_2}\lambda_2)\Lambda} f_{k_1}\|_{L^\infty}
\]
\[
\times \big( \| T^{\tau,\kappa}_{ \lambda_2}(\p_t f_{k_1'}, f_{k_2'})(\eta)\|_{L^2}+\| T^{\tau,\kappa}_{ \lambda_2}( f_{k_1'},\p_t f_{k_2'})(\eta)\|_{L^2}\big) + 2^{\alpha k  + m + j+k_1 -\kappa_1-\kappa_2+k_2+l/2} \| \p_t f_{k_1}\|_{L^2}  \]
\[
   \times\| T^{\tau,\kappa}_{ \lambda_2}( f_{k_1'}, f_{k_2'})(\eta)\|_{L^2} \big]+ 2^{-10 m + k_1+k_2-\kappa_1-\kappa_2 +k_2+k_2'}\big( \|\p_t {f_{k_1}}\|_{L^2} \| f_{k_1'}\|_{L^2}\| {f_{k_2'}}\|_{L^2}
\]
\[
+ \|{f_{k_1}}\|_{L^2} \| \p_t {f_{k_1'}}\|_{L^2}\| {f_{k_2'}}\|_{L^2} + \|{f_{k_1}}\|_{L^2} \| {f_{k_1'}}\|_{L^2}\| \p_t {f_{k_2'}}\|_{L^2}\big)
\]
\be\label{neweqn460}
\lesssim 2^{-(1+2\alpha)k_2/2+ k_1/2-m+13\beta m } \epsilon_0 + 2^{-2\delta m -2\delta j}\epsilon_0 + 2^{(1-\alpha)k_2+k_1/2+8\beta m}\epsilon_0 \lesssim 2^{-2\delta m-2\delta j}\epsilon_0.
\ee
If $k_2\geq -10\beta m$, then instead of using the inverse Fourier transform twice, we use the $L^2-L^\infty-L^\infty$ type trilinear estimate directly. From Lemma \ref{Snorm}, (\ref{trilinearesetimate}) in Lemma \ref{multilinearestimate},   and (\ref{L2derivativeestimate}) in Lemma \ref{L2estimatelemma}, the following estimate holds, 
\[
2^{\delta j}\| \mathcal{F}^{-1}[ \textup{H}_1]\|_{B_{k,j}} \lesssim 2^{2\delta m + \alpha k  + m + j -20k_2 }\big( \| \p_t f_{k_1}\|_{L^2} \| e^{-it \Lambda} f_{k_1'}\|_{L^\infty}\| e^{-it \Lambda} f_{k_2'}\|_{L^\infty} \]
\[+ \| \p_t f_{k_1'}\|_{L^2} \| e^{-it \Lambda} f_{k_1}\|_{L^\infty}\| e^{-it \Lambda} f_{k_2'}\|_{L^\infty}+ \| \p_t f_{k_2'}\|_{L^2} \| e^{-it \Lambda} f_{k_1'}\|_{L^\infty}\| e^{-it \Lambda} f_{k_1}\|_{L^\infty}\big)
\]
\be\label{neweqn490}
\lesssim 2^{-m-22k_2+7\beta m } \epsilon_0 \lesssim 2^{-2\delta m -2\delta j}\epsilon_0. 
\ee

\textbf{Case $3$:\quad} If $k_2'-3\beta m \geq k_2$ and $(k_1', k_2', \nu\tau, \nu\kappa)\in \mathcal{P}_{bad}^{k_2}$. Note that the assumption $k_2'-3\beta m \geq k_2$ implies that $(k_1',k_2')\in \chi_{k_2}^1$. 

For this case, we first 
localize the angle between $\eta$ and $\nu\kappa \sigma$ and then decompose $J^{\mu, \nu, \tau, \kappa, 1}_{l;k_1', k_2'}$ as follows, 
\begin{equation}\label{equation7060}
J^{\mu, \nu, \tau, \kappa, 1}_{l;k_1', k_2'} = \sum_{j_1\geq -k_{1,-},j_2\geq -k_{2,-}} \sum_{j_1'\geq -k_{1,-}',j_2'\geq -k_{2,-}'} \sum_{\bar{\tilde{l}}\leq \tilde{l}\leq 2} H^{j_1, j_2}_{l,\tilde{l};j_1',j_2'}
\end{equation}
\begin{equation}\label{equation7070}
  H^{j_1, j_2}_{l,\tilde{l};j_1',j_2'}:=\int_{t_1}^{t_2} e^{it \Phi^{\mu, \nu}(\xi, \eta)} \varphi_{\bar{l};l}(\angle(\xi, \nu \eta)) m_{\mu, \nu}(\xi, \eta) \widehat{f_{k_1, j_1}^{\mu}}(t, \xi-\eta) \widehat{Q_{k_2, j_2}[T^{\nu\tau,\nu \kappa}_{\tilde{l};j_1',j_2'}(t) ]^{\nu}}(\eta)  d \eta dt,
\end{equation}
where $\bar{\tilde{l}}= \max\{l -6\beta m/5, 2k_{1,-}'\} $ and 
\[
T^{\nu\tau,\nu \kappa}_{\tilde{l};j_1',j_2'}(t) = \mathcal{F}^{-1}\big[ \int_{\R} e^{i t \Phi^{\tau, \kappa}(\eta,\sigma)} \widehat{f_{k_1',j_1'}^{\nu\tau}}(t,\eta-\sigma) \widehat{f_{k_2',j_2'}^{\nu\kappa}}(t,\sigma) q_{\nu\tau, \nu\kappa}(\eta-\sigma, \sigma) \varphi_{\bar{\tilde{l}};\tilde{l}}(\angle(\eta, \nu \kappa \sigma))d \sigma \big].
\]
For simplicity, we also use the following notation,
\[
 H_{l,\tilde{l};j_1',j_2'}:= \sum_{j_1\geq-k_{1,-}, j_2\geq-k_{2,-},  }  H^{j_1, j_2}_{l,\tilde{l};j_1',j_2'},\]
 \[  H^{j_1, j_2}_{l,\tilde{l}} =  \sum_{j_1'\geq-k_{1,-}',j_2'\geq-k_{2,-}'  }  H^{j_1, j_2}_{l,\tilde{l};j_1',j_2'}, \quad T^{\nu\tau,\nu \kappa}_{\tilde{l}}(t) = \sum_{j_1'\geq-k_{1,-}',j_2'\geq-k_{2,-}' } T^{\nu\tau,\nu \kappa}_{\tilde{l};j_1',j_2'}(t). 
\]

$\bullet$\quad We first consider the case when either $2k_{1,-}'\geq  l-6\beta m/5$  or $2k_{1,-}'<  l-6\beta m/5$, $\tilde{l}> \bar{\tilde{l}}=l-6\beta m /5$. Recall that $l\in[2k_{1,-},2]$ and $k_2\leq k_2'-3\beta m $, i.e., $|\eta|\ll |\eta-\sigma|\sim |\sigma|$. For the case we are considering,  we have $k_2+2\tilde{l}\geq k_2 +2l -12\beta m/5 \geq -m + 8\beta m/5 $ and $l-\tilde{l} \leq 6\beta m /5$, which means that we are away from the space resonance in ``$\sigma$'' set.  Hence, we can do integration by parts in ``$\sigma$'' many times to rule out the case when $\max\{j_1',j_2'\}\leq m+ k_2-k_{2}' +\tilde{l} -2\beta m$. 

On one hand, from (\ref{equation6910}) in Lemma \ref{angularbilinear}  and (\ref{equation6884}) in Lemma \ref{bilinearest}, the following estimate holds, 
\[
\sum_{\max\{j_1', j_2'\}\geq m+ k_2-k_2'+\tilde{l}-2\beta m}2^{\delta j}\| \mathcal{F}^{-1}[H_{l,\tilde{l};j_1',j_2'}]\|_{B_{k,j}}\lesssim \sum_{\max\{j_1', j_2'\}\geq m+ k_2-k_2'+\tilde{l}-2\beta m} \sup_{|\lambda|\leq 2^{\beta m } } 2^{2\delta m + \alpha k  } \]
\[
\times 2^{ m +jk_1-k_2-2\max\{l,k_{1,-} \}} 2^{(k_1-k_2)/2}\| e^{-i (t+2^{-k_2-2\max\{l, k_{1,-}\}} \lambda)} f_{k_1}(t)\|_{L^\infty}  2^{k_2-m-\max\{j_1',j_2'\} -2\alpha k_2'} \|f_{k_1', j_1'}\|_Z \]
\[
\times \|f_{k_2',j_2'}\|_Z +2^{-10 m +k+\alpha k+m+j+k_1-k_2-2\max\{l,k_{1,-}\} +k_2}\| f_{k_1}\|_{L^2} 2^{-m-\max\{j_1',j_2'\}-2\alpha k_2'} \|f_{k_1', j_1'}\|_Z \|f_{k_2',j_2'}\|_Z
\]
\begin{equation}\label{equation7055}
\lesssim 2^{-3k_2/2-k_1/2+k_2'+10\beta m -m } \epsilon_0
\lesssim 2^{-m/10 +11\beta m -k_1/2}\epsilon_0\lesssim 2^{-\alpha m/2 -k_1/2}\epsilon_0.
\end{equation}
On the other hand, from (\ref{eqn293}) and (\ref{equation6884}) in Lemma \ref{bilinearest}, the following estimate also holds, 
\[
\sum_{\max\{j_1', j_2'\}\geq m+ k_2-k_2'+\tilde{l}-2\beta m}2^{\delta j}\| \mathcal{F}^{-1}[H_{l,\tilde{l};j_1',j_2'}]\|_{B_{k,j}}\lesssim \sum_{\max\{j_1', j_2'\}\geq m+ k_2-k_2'+\tilde{l}-2\beta m} 2^{\alpha k  + m +j +k_1-k_2} 
\]
\begin{equation}\label{equation7056}
 \times 2^{-2\max\{l, k_{1,-}\}+k_1+k_2+l} \| \widehat{f_{k_1}}(t,\xi)\|_{L^\infty_\xi} 2^{k_2-m-\max\{j_1',j_2'\}-2\alpha k_2'} \|f_{k_1',j_1'}\|_{Z} \| f_{k_2',j_2'}\|_{Z} \lesssim 2^{k_1 +10\beta m }\epsilon_0. 
\end{equation}
Therefore, combining estimates (\ref{equation7055}) and (\ref{equation7056}), we can derive the following estimates,
\[
\sum_{\max\{j_1', j_2'\}\geq m+ k_2-k_2'+\tilde{l}-2\beta m}2^{\delta j}\| \mathcal{F}^{-1}[H_{l,\tilde{l};j_1',j_2'}]\|_{B_{k,j}}\lesssim \big( 2^{-\alpha m/2 -k_1/2}\epsilon_0\big)^{1/2} (2^{k_1 +10\beta m }\epsilon_0)^{1/2}\]
\[
\lesssim 2^{-2\delta m -2\delta j}\epsilon_0.
\]
 $\bullet$\quad Lastly, we consider the case when $2k_{1,-}' < l-6\beta m /5$ and  $\tilde{l}=\bar{\tilde{l}} =   l -6\beta m/5  $. Hence,  we have $k_{2}'\leq  -3\beta m /5$ and $k_2\leq -18\beta m /5$.

If we view the bilinear term $T^{\nu\tau,\nu \kappa}_{\tilde{l}}(t)$ in (\ref{equation7070}) as a single input, then it is easy to see that  the estimate of $J^{\mu, \nu, \tau, \kappa, 1}_{l;k_1', k_2'}$ is very similar to   the estimate of the  endpoint case in the proof of  Lemma \ref{highlowbadlemma9}.  More precisely, we separate into two cases based on whether $j_1$ is smaller than $j_2$. 

(i)\quad If $j_2\leq j_1$.  Then we can first rule out the case when $j_1\leq  m +l -\beta m $ by doing integration by parts in ``$\eta$'' many times for (\ref{equation7070}). It remains to consider the case when $j_1\geq m  +l -\beta m$. From the estimates (\ref{eqn293}) and (\ref{equation6884}) in Lemma \ref{bilinearest}, the following estimate holds, 
\[
\sum_{j_2\leq j_1, m+l-\beta m \leq j_1}2^{\delta j} \|\mathcal{F}^{-1}[ H^{j_1, j_2}_{l,\tilde{l}}]\|_{B_{k,j}} \lesssim \sum_{m+l-\beta m \leq j_1} 2^{2\delta m + \alpha k + m +j+k_{1}-k_2-2\max\{l, k_{1,-}\} + k_2+l/2}  \| f_{k_1,j_1}\|_{L^2}\]
\[
\times  \| T^{\nu\tau,\nu \kappa}_{\tilde{l}}(t)\|_{L^2} \lesssim \sum_{j_1\geq m +l -\beta m } 2^{-k_{1}/2  +2m +l+ 100\delta m -j_1} \|f_{k_1,j_1}\|_Z   2^{k_2} \| e^{-it \Lambda} f_{k_1'}\|_{L^\infty} \| f_{k_2'}\|_{L^2}
\]
\[
  \lesssim 2^{k_2/2+(1-2\alpha)k_2'+\beta m+100\delta m } \epsilon_0\lesssim 2^{-2\delta m -2\delta j}\epsilon_0.
\]
 
(ii)\quad If $j_2\geq j_1$. For this case, we can  rule out the case when $j_2\leq m +l -\beta m$ by doing integration by parts in ``$\eta$'' many times for  (\ref{equation7070}). Therefore, it remains to consider the case when   $j_2\geq m +l -\beta m$. Note that $j_2\geq m + l -\beta m \geq m + \tilde{l}+\beta m /6$, $k_2+2\tilde{l}\geq k_2+2 l -12\beta m /5\geq -m + 8\beta m/5$, and $k_2'\geq k_2+3\beta m$. Now, it is easy to see that  all conditions in Lemma \ref{largejscenario} are satisfied.   Therefore, from (\ref{eqn293}) in Lemma \ref{bilinearest}, (\ref{L2estimate}) in Lemma \ref{L2estimatelemma}, and (\ref{equation7080}) in Lemma \ref{largejscenario}, the following  estimate holds,
\[
 \sum_{j_1\leq j_2, m+l-\beta m \leq j_2}2^{\delta j}\|\mathcal{F}^{-1}[ H^{j_1, j_2}_{l,\tilde{l}}]\|_{B_{k,j}} \lesssim \sum_{m+l-\beta m \leq j_2} 2^{2\delta m + \alpha k   + m +j+k_{1}-k_2-2\max\{l, k_{1,-}\} + k_2+l/2 }  \]
\[
\times \| f_{k_1}\|_{L^2} \| Q_{k_2,j_2}[\mathcal{F}^{-1}[T^{\nu\tau,\nu \kappa}_{\tilde{l}}(t)]]\|_{L^2} \lesssim \sum_{ m +l -\beta m\leq j_2 } 2^{m+(1-2\alpha)k_2 +3l/2-j_2+ 200\delta m }\epsilon_0 \lesssim 2^{-2\delta m -2\delta j}\epsilon_0.
\]
Hence finishing the proof.
\end{proof}
\begin{lemma}\label{largejscenario}
Under the bootstrap assumption \textup{(\ref{smallness})}, the following estimate holds if $(k_1, k_2, \mu, \nu)\in \mathcal{P}^k_{bad}$ and $t\in[2^{m-1},2^m]$,
\[
 \| Q_{k,j}\Big[ \mathcal{F}^{-1}\big[\int_{\R^2}e^{it\Phi^{\mu, \nu}(\xi, \eta)
} \widehat{f_{k_1}^{\mu}}(t, \xi-\eta) \widehat{f_{k_2}^{\nu}}(t,\eta) q_{\mu, \nu}(\xi-\eta, \eta) \varphi_{\bar{l};l}(\angle(\xi, \nu\eta)) d \eta\big]\Big]\|_{L^2}  
\]
\begin{equation}\label{equation7080}
\lesssim 2^{(1-2\alpha) k -m-j+2\delta m}\epsilon_0, \quad \textup{if $j\geq \max\{m+l, -k-l, -k_2-l\} +100\delta m$ and $\bar{l}\geq 2k_{1,-}$}.
\end{equation}
\end{lemma}
\begin{proof}
To prove the desired estimate (\ref{equation7080}), we only need to redo the proof of Lemma \ref{badhhlemma1} and the proof of Lemma \ref{highlowbadlemma1}. From  the second estimate in (\ref{equation6603}) and (\ref{equation10010}) instead of the last estimate in (\ref{equation6603}) and (\ref{equation10010}), it is easy to see that our desired estimate    (\ref{equation7080}) holds.
\end{proof}
\begin{lemma}\label{hlowcubiczlemma2}
Under the bootstrap assumption \textup{(\ref{smallness})}  and the assumption that Proposition \textup{\ref{propZnorm2}} holds, if $  j\leq \max\{m+l,\min\{-k_2-l,m\}\} + 100\delta m  $, $k$,  $k_1$, and $k_2$   satisfy \textup{(\ref{badrange1})},  $(2-2\alpha)k_2\geq -m-20\beta m $, and  $k_2+2l\geq -m +4\beta m $, $ k_1', k_2' \in [-2m, 2\beta m ]$, then the following estimate holds, 
\be\label{noveqn862}
 2^{\delta j}\| \mathcal{F}^{-1}[J^{\mu, \nu, \tau, \kappa, 2}_{l;k_1', k_2'}]\|_{B_{k,j}}  \lesssim 2^{-2\delta m -2\delta j}\epsilon_0. 
\ee
\end{lemma}
\begin{proof}
Recall (\ref{equation7001}) and (\ref{noveqn812}). Note that $(k_1',k_2')\in \chi_{k_1}^1\cup \chi_{k_1}^2\cup \chi_{k_1}^3$.

We first rule out the case when $\tau=\kappa=-$. Note that the  phase $\Phi_2^{\tau, \kappa, \nu}(\xi, \eta, \sigma)$ is  at least  of size $2^{k_1}$ for this case. Hence,
   same as what we did in the estimate of $J^{\mu, \nu, \tau, \kappa, 1}_{l;k_1', k_2'}$,  we can do integration by parts in time to take the advantage of the fact that the size of phase is big. With minor modifications in (\ref{neweqn450}), (\ref{neweqn460}) and (\ref{neweqn490}), it is easy to see that our desired estimate (\ref{noveqn862}) holds. 

    For the case when $(\tau, \kappa)\neq (-,-)$, we divide it into four cases as follows.  

\textbf{Case $1$:\quad} If $l\geq -2\alpha m/3$ and $k_2'+2k_1\leq -m/2-\alpha m+3\beta m $. 
Recall that $k_2+2l\geq -m +4\beta m $, which means that the frequencies are away from the time resonance set. 
 From (\ref{equation6911}) in Lemma \ref{angularbilinear} and (\ref{L2estimate}) in Lemma \ref{L2estimatelemma}, it is easy to see that the following estimate holds   after putting the input $f_{k_2}$ in $L^\infty$ and $T^{\tau, \kappa}(f_{k_1'},f_{k_2'})$ in $L^2$,
\[
2^{\delta j}\| \mathcal{F}^{-1}[J^{\mu, \nu, \tau, \kappa, 2}_{l;k_1', k_2'}]\|_{B_{k,j}}\lesssim 2^{   \alpha k  + m +j+k_1-k_2-2\max\{l,k_{1,-} \} } \big( \sup_{|\lambda|\leq 2^{\beta m} }  \| e^{i(t+2^{-k_2-2\max\{l, k_{1,-}\}}\lambda)} f_{k_2}(t)\|_{L^\infty}\]
\begin{equation}\label{equation7040}
\times 2^{2\delta m+k_1+k_{1,+}'} \| e^{-it\Lambda} f_{k_1'}\|_{L^\infty} \| f_{k_2'}\|_{L^2}+2^{-10 m + k_1 +k_{1,+}'}\| f_{k_2}\|_{L^2}\| e^{-it\Lambda} f_{k_1'}\|_{L^\infty}  \| f_{k_2'}\|_{L^2}\big)
\end{equation}
\[
 \lesssim 2^{m+2k_1+2\beta m +k_2'-k_2-l}  (2^{2k_2+\delta m })^{1/2} ( 2^{-m+\alpha m/3})^{1/2}\epsilon_0\lesssim 2^{-\alpha m/6+6\beta m}\epsilon_0\lesssim 2^{-2\delta m -2\delta j}\epsilon_0.
\]
Note that we used the following fact in the above estimate, 
\[
\| e^{-it \Lambda} f_{k_2}(t)\|_{L^\infty} \lesssim  \min\{ 2^{-m+\alpha m /3}\epsilon_1, 2^{k_2}\|P_{k_2} f(t)\|_{L^2}\} \lesssim  \min\{ 2^{-m+\alpha m/3}, 2^{2k_2+\delta m }\}\epsilon_1.
\]
 
\textbf{Case $2$:\quad} If $l\geq -2\alpha m/3$ and $k_2'+2k_1\geq -m/2-\alpha m+3\beta m $. Recall that $(k_1',k_2')\in  \chi_{k_1}^1\cup \chi_{k_1}^2 \cup \chi_{k_1}^3$.  For this case we have $\max\{k_1, k_1'\}\geq -m/6-\alpha m /3+\beta m $ and $\min\{k_1,k_2'\}+4k_1'\geq -5m/6-5\alpha m/3+5\beta m\geq -m +5\beta m $.

When  $\tau\kappa=-$, from (\ref{dece1}) in Lemma \ref{phasesize}, it is easy to see that we are away from the space resonance in ``$\sigma$'' set. Hence,  we can do integration by parts in $\sigma$ many times to rule out the case $\max\{j_1',j_2'\}\leq m +k_1'+k_1-4\beta m $.

 For the case when $\tau\kappa=+$, i.e., $\tau=\kappa=+$ as the case   $\tau=\kappa=-$ is ruled out, we separate into two cases based on the  size of the angle $\angle(\xi-\sigma, \sigma-\eta)$. 

If  $\angle(\xi-\sigma, \sigma -\eta) \geq 2^{-\alpha m}$, then the net gain of doing integration by parts in ``$\sigma$ '' once is at least  $\max\{2^{-m+\alpha m-\max\{j_1', j_2'\}+2\beta m}, 2^{-m-k_1+2\alpha m +2\beta m }\}$, which is less than $2^{-\beta  m}$ when $\max\{j_1', j_2'\}\leq m -2\alpha m$. If $\angle(\xi-\sigma, \sigma -\eta) \leq 2^{-\alpha m}$, then we have $\angle(\xi-\eta, \sigma-\eta) \leq 2^{-\alpha m}$ and $\angle(\sigma-\eta, \nu \eta)\sim \angle(\xi-\eta, \nu \eta)\sim 2^l\gtrsim 2^{-2\alpha m/3}$. For this case, we do integration by parts in ``$\eta$''. The net gain of doing integration by parts in ``$\eta$'' once is at least $\max\{2^{-m-l-\max\{j_2, j_2'\}+2\beta m}, 2^{-m-k_2-2l+2\beta m }\}$, which is less than $2^{-\beta  m}$ when $\max\{j_2 , j_2'\}\leq m -2\alpha m$. 

Therefore,  in whichever case, we can rule out the case when $\max\{j_2,j_1',j_2'\}\leq \min\{m+k_1+k_1'-4\beta m , m-2\alpha m \}$. It is sufficient to consider the case when $\max\{j_2,j_1',j_2'\}\geq \min\{m+k_1+k_1'-4\beta m , m-2\alpha m \}\geq m+k_1+k_1'-4\beta m-2\alpha m   $. From (\ref{equation6911}) in Lemma \ref{angularbilinear}, the following estimate holds, 
\[
\sum_{\max\{j_2,j_1',j_2'\}\geq  m+k_1+k_1'-4\beta m-2\alpha m }2^{\delta j}\| \mathcal{F}^{-1}[H_{l,j_2;j_1',j_2'}^{\mu,\nu,\tau, \kappa,2}]\|_{B_{k,j}}\lesssim  \sum_{\max\{j_2,j_1',j_2'\}\geq  m+k_1+k_1'-4\beta m-2\alpha m   }   2^{2\delta m + \alpha k  } \]
\[ \times 2^{   m +j k_1-k_2-2\max\{l,k_{1,-} \}} 2^{-m-\alpha k_2}\| f_{k_2,j_2}\|_Z
2^{k_1-m-\max\{j_2,j_1',j_2'\} -\alpha k_1'-\alpha k_2' } \| f_{k_1',j_1'}\|_Z \| f_{k_2',j_2'}\|_Z \]
\[
    +2^{-10 m +k+\alpha k+m+j+k_1-k_2-2\max\{l,k_{1,-}\} }\| f_{k_2,j_2}\|_{Z}2^{k_1-m-\max\{j_2,j_1',j_2'\} -\alpha k_1'-\alpha k_2' } \| f_{k_1',j_1'}\|_Z \| f_{k_2',j_2'}\|_Z
\]
\[
\lesssim  2^{-(1+\alpha)k_2-l+2\alpha m -\alpha k_2'+4\beta m -m}\epsilon_0 \lesssim 2^{-2\delta m -2\delta j}\epsilon_0.
\]

\textbf{Case $3$:\quad} If $l\leq -2\alpha m/3$ and $k_2'+2k_1 \leq k_2+2\max\{l, k_{1,-}\} +\beta m$. From (\ref{equation6911}) in Lemma \ref{angularbilinear} and (\ref{L2estimate}), estimate (\ref{equation7040}) also holds and 
\[
(\ref{equation7040})\lesssim 2^{l+2k_1+k_2'-k_2-2\max\{l, k_{1,-}\} +\alpha m/3 + 8 \beta m}\epsilon_0  + 2^{-2\delta m -2\delta j}\epsilon_0\]
\[\lesssim  2^{-2\alpha m /3+\alpha m/3+9\beta m }\epsilon_0+ 2^{-2\delta m -2	\delta j}\epsilon_0\lesssim 2^{-2\delta m -2\delta j}\epsilon_0.
\]

\textbf{Case $4$:\quad} If $l\leq -2\alpha m/3$ and $k_2'+2k_1 \geq k_2+2\max\{l, k_{1,-}\} +\beta m$. The assumption in this case implies that  $k_1\leq l/2\leq -\alpha m /3$,  $k_2'\geq k_2 +\beta m $. Recall (\ref{neweqn2}). From the estimate (\ref{noveqn519}) in Lemma \ref{roughestimatephase2}, the following estimate holds for the size of phase $\Phi_2^{\mu, \nu,\tau, \kappa}(\xi, \eta,\sigma)$, 
\begin{equation}\label{equation7041}
|\Phi_2^{\mu, \nu,\tau, \kappa}(\xi, \eta,\sigma)|\gtrsim 2^{k_2'+2k_1}-2^{k_2+2\max\{l,k_{1,-}\}}\gtrsim 2^{k_2'+2k_1}.
\end{equation}
With this observation, it motives us to do integration by parts in time again and have a similar identity as in (\ref{equation7082}). Very similar to the proof of the estimate (\ref{neweqn450}),  after using the inverse Fourier transform twice, the  following estimate holds from  the estimate  (\ref{eqn293}) in Lemma \ref{bilinearest}, (\ref{L2estimate}), (\ref{L2derivativeestimate}) and (\ref{equation6726}) in Lemma \ref{L2estimatelemma},
\[
2^{\delta m}\| \mathcal{F}^{-1}[J^{\mu, \nu,\tau, \kappa, 2}_{l;k_1',k_2'}] \|_{B_{k,j}} \lesssim \sum_{\kappa_2 \geq k_2'+2k_1}  \sup_{|\lambda_1|,|\lambda_2|\leq 2^{\beta m/10}} \sum_{i=1,2} 2^{2\delta m + \alpha k + j+2k_1 -k_2-2\max\{k_1,l\}-\kappa_2+k_2+l/2} 
\]
\[
\times \| f_{k_2}(t_i)\|_{L^2} \| e^{-i(t_i +2^{-\kappa_2}\lambda_2) \Lambda}f_{k_1'}\|_{L^\infty} \| f_{k_2'}\|_{L^2}  + 2^{\alpha k + m+ j+2k_1 -k_2-2\max\{k_1,l\}-\kappa_2 +k_2+l/2} 
 \]
 \[
\times\big(  \| \p_t  f_{k_2}\|_{L^2} \| e^{-i(t_i +2^{-\kappa_2}\lambda_2) \Lambda}f_{k_1'}\|_{L^\infty}\| f_{k_2'}\|_{L^2} + \|  f_{k_2}\|_{L^2} \| e^{-i(t_i +2^{-\kappa_2}\lambda_2) \Lambda}f_{k_1'}\|_{L^\infty} \| \p_t f_{k_2'}\|_{L^2}\big) \]
\[ + 2^{\alpha k + m+ j+2k_1 -k_2-2\max\{k_1,l\}-\kappa_2 +2k_2+l}  \|  \widehat{f_{k_2}}(t,\xi)\|_{L^\infty_\xi}\| e^{-i(t_i +2^{-\kappa_2}\lambda_2) \Lambda}f_{k_2'}\|_{L^\infty} \| \p_t f_{k_1'}\|_{L^2} \]
\[
+ 2^{-10m-k_2-2\max\{k_1,l\}-\kappa_2} \big(\|{f_{k_2}}\|_{L^2}\| {f_{k_1'}}\|_{L^2}
\| f_{k_2'}\|_{L^2} + \| {f_{k_2}}\|_{L^2}\| {\p_t f_{k_1'}}\|_{L^2}
\| {f_{k_2'}}\|_{L^2}
\]
\[+ \| {\p_t f_{k_2}}\|_{L^2}\| {f_{k_1'}}\|_{L^2}
\| {f_{k_2'}}\|_{L^2} + \| {f_{k_2}}\|_{L^2}\| {f_{k_1'}}\|_{L^2}
\| {\p_t f_{k_2'}}\|_{L^2}\big)
\]
\[
\lesssim 2^{k_2 -\alpha k_2'-\max\{l, k_{1,-}\}/2+\beta m }\epsilon_0 +2^{ k_2 +k_1'+2l-k_2'-2\max\{l, k_{1,-}\}+\beta m +\alpha m /3  }\epsilon_0  + 2^{-2\delta m-2\delta j}\epsilon_0\]
\[
\lesssim 2^{(1-2\alpha)k_2/2+2\beta m }\epsilon_0 + 2^{ l+\alpha m/3+2\beta m} \epsilon_0+ 2^{-2\delta m-2\delta j}\epsilon_0\lesssim 2^{-2\delta m -2\delta j}\epsilon_0.
\]
In the above estimate, we used the fact that $k_2\leq k_1-5$, $k_2\leq k_2'-\beta m$, $(k_1',k_2')\in \chi_{k_1}^1\cup  \chi_{k_1}^2\cup  \chi_{k_1}^3$, $l\leq -2\alpha m/3$ and $k_2\leq k_1\leq -\alpha m/3$ in the above estimate. 	Hence finishing the proof. 

 \end{proof}

\subsection{The estimate of  $K^{\mu, \nu}(f_{k_1}^{\mu}, f_{k_2}^{\nu})$   in $\textup{bad}_k$.  }\label{comparablebad} 
In this subsection, we estimate the last term in `` $\textup{bad}_k$'', see (\ref{badtypeterms}). Hence finishing the proof of Proposition \ref{propZnorm4bad}.

Recall (\ref{badtypeterms}) and (\ref{neweqn95}). Note that the output frequency and the two input frequencies are all comparable as $(k_1,k_2)\in \chi_k^3$. From the estimate (\ref{sizeofnormalform}), the estimates (\ref{L2estimate}) and (\ref{L2derivativeestimate}) in Lemma \ref{L2estimatelemma}, the following estimate holds from the $L^2-L^2$ type estimate and the volume of support of $\xi$, 
\[
2^{\delta j}\| K^{\mu, \nu}(f_{k_1}^{\mu}, f_{k_2}^{\nu})\|_{B_{k,j}}\lesssim 2^{\alpha k  + 7k_{+}-2k_{-}+ m+(1+\delta)j+k}\big(\|\p_t f_{k_1}(t)\|_{L^2} \|f_{k_2}(t)\|_{L^2} + \|\p_t f_{k_2}(t)\|_{L^2} \|f_{k_1}(t)\|_{L^2}\big)
\] 
\be\label{deceqn31}
\lesssim 2^{(1+\alpha)k+(1+\delta)j - (N_0-10)k_{+}}\epsilon_0.
\ee
Hence, we can rule out the case when $k\leq -(1+\delta)j/(1+\alpha)-2\delta m$  or $k\geq j/(N_0-20)+10\delta m$. Moreover, it is easy to see that the proof of Lemma \ref{hhgoodj1} is still valid. As a result, it would be sufficient to   prove the following estimate, 
\be\label{desiredestimatecomparable}
(\textup{Desired estimate}): \quad 2^{\delta j}\| \int_{t_1}^{t_2} \mathcal{F}^{-1}\big[ K^{\mu, \nu}
(f^{\mu}_{k_1}, {f}^{\nu}_{k_2})\big] d t\|_{B_{k,j}} \lesssim 2^{-2\delta m - 2\delta j}\epsilon_0,
\ee
where fixed $k, k_1,k_2$ and $j$ satisfy the following estimate, 
\be\label{noveqn1000}
 - (1+100\delta)m/(1+\alpha)\leq k  \leq  2\beta m, \quad (k_1,k_2)\in \chi_k^3, \quad   j\leq (1+20 \delta ) m.
\ee

Recall (\ref{neweqn95}), after plugging the equation satisfied by $\p_t f$ in  (\ref{eqn1400}),  the following equality holds,  
\be\label{noveqn1008}
\int_{t_1}^{t_2} K^{\mu, \nu}( {f}^{\mu}_{k_1}, f^{\nu}_{k_2}) = \sum_{k_1', k_2'\in \mathbb{Z}} \sum_{\tau, \kappa\in \{+, -\}}\sum_{i=1,2} \sum_{\bar{l}_i\leq l\leq 2}\big[ \sum_{ j=1,2} K^{\mu, \nu, \tau, \kappa,i}_{l;k_1', k_2',j}\big] + \textup{JR}_{l;k_1,k_2}^{\mu, \nu,i}, 
\ee
where $ \bar{l}_1:=2k_{-},  \bar{l}_2:=0,$
\[
K^{\mu, \nu, \tau, \kappa,i}_{l; k_1', k_2',1}=  -\int_{t_1}^{t_2}\int_{\R^2} \int_{\R^2}  e^{i t \Phi_1^{\mu, \tau, \kappa}(\xi, \eta, \sigma)} c^{ \tau, \kappa,i}_{\mu, \nu,1}(\xi, \eta, \sigma)\widehat{
f_{k_1 }^{\mu}
}(t,\xi-\eta)\widehat{f_{k_1'  }^{\tau}}(t,\eta-\sigma) \]
\be
\times \widehat{f_{k_2' }^{\kappa}}(t,\sigma) \varphi_{\bar{l}_i; l}(\angle(\xi, \nu \eta))d \eta d\sigma   d t,\quad i=1,2,
\ee
\[
K^{\mu, \nu, \tau, \kappa,i}_{l; k_1', k_2',2}=  -\int_{t_1}^{t_2}\int_{\R^2} \int_{\R^2} e^{i t \Phi_2^{\tau, \kappa, \nu}(\xi, \eta, \sigma)}c^{  \tau, \kappa,i}_{\mu , \nu  ,2}(\xi, \eta, \sigma)\widehat{
f_{k_1' }^{\tau}
}(t,\xi-\sigma)\widehat{f_{k_2'  }^{\kappa}}(t,\sigma-\eta) \]
\be\label{deceqn2}
\times \widehat{f_{k_2 }^{\nu}}(t,\eta) \varphi_{\bar{l}_i; l}(\angle(\xi, \nu \eta))d \eta d\sigma   d t,\quad i=1,2,
\ee
\[
\textup{JR}^{\mu, \nu,i}_{l;k_1,k_2}=-  \int_{t_1}^{t_2} \int e^{i t\Phi^{\mu  ,\nu}(\xi, \eta)}  \big(\widehat{ ({\mathcal{R}'})^{\mu}  }_{k_1}(t,\xi-\eta)  \widehat{f_{k_2}^{\nu}}(t,\eta)  +  \widehat{ {f}^{\mu}_{k_1}}(t,\xi-\eta)\widehat{ (\mathcal{R'})^{\nu}_{k_2}}(t, \eta)\big)\]
\be\label{deceqn1}
\times m_{\mu,\nu}^i(\xi-\eta, \eta) \varphi_{\bar{l}_i; l}(\angle(\xi, \nu \eta)) d \eta d t,
\ee
where    the symbols $ m^i_{\mu,\nu} (\xi-\eta, \eta)$ and $c^{  \tau, \kappa}_{\mu, \nu,i}(\xi, \eta,\sigma)$, $i\in \{1,2\}$, are defined as follows,
\[
c^{  \tau, \kappa,i}_{\mu, \nu,1}(\xi, \eta,\sigma):= m^i_{\mu,\nu}(\xi-\eta, \eta)  \big(q_{\tau\nu, \kappa\nu}(\eta-\sigma, \sigma)\big)^{\nu}   \psi_{k }(\xi )\psi_{k_1}(\xi-\eta)\psi_{k_2}(\eta), 
\]
\[
 c^{  \tau, \kappa,i}_{\mu, \nu, 2}(\xi, \eta, \sigma):= m^i_{\mu,\nu} (\xi-\eta, \eta)\big( q_{\mu\tau, \mu\kappa}(\xi-\sigma, \sigma-\eta)\big)^{\mu} \psi_{k }(\xi )\psi_{k_1}(\xi-\eta)\psi_{k_2}(\eta),
\]
\be\label{deceqn11}
 m^1_{\mu,\nu} (\xi-\eta, \eta):= m_{\mu,\nu} (\xi-\eta, \eta)(1-c_{\mu, \nu}(\xi, \eta)), \quad  m^2_{\mu,\nu} (\xi-\eta, \eta):= m_{\mu,\nu} (\xi-\eta, \eta) c_{\mu, \nu}(\xi, \eta),
\ee
where the symbols $m_{\mu, \nu}(\xi-\eta, \eta) $  $c_{\mu,\nu}(\xi, \eta)$, $\mu,\nu\in\{+,-\}$, are defined in   (\ref{neweqn90}), (\ref{auxillarycutoff1}) and (\ref{auxillarycutoff2}) respectively.

In (\ref{noveqn1008}), we separated the cubic terms into two parts based on whether $(\xi, \eta)$ is close to the support of $c_{\mu,\nu}(\xi, \eta)$. We did this decomposition because the size of phases is not small when  $(\xi, \eta)$ is close to the support of $c_{\mu,\nu}(\xi, \eta)$, see (\ref{improvedphase}) in Lemma \ref{roughestimatephase2}.

From the estimate (\ref{sizeofsymboluniform}) in Lemma \ref{sizeofsymbol}, (\ref{noveqn519}) in Lemma \ref{roughestimatephase2}, the following estimate holds,
\[
\|c^{  \tau, \kappa,1}_{\mu, \nu,1}(\xi, \eta,\sigma)\psi_{k_1'}(\eta-\sigma) \psi_{k_2'}(\sigma) \varphi_{\bar{l}_1;l}(\angle(\xi, \nu\eta))\|_{L^\infty_{\xi, \eta, \sigma}} \]
\be\label{noveq719}
+ \|c^{  \tau, \kappa,1}_{\mu, \nu,2}(\xi, \eta,\sigma)\psi_{k_1'}(\xi-\sigma) \psi_{k_2'}(\sigma-\eta) \varphi_{\bar{l}_1;l}(\angle(\xi, \nu\eta))\|_{L^\infty_{\xi, \eta, \sigma}} \lesssim 2^{k-2\max\{l,k_{-}\}+k_{+}+k_{1,+}' }. 
\ee
From   the estimate (\ref{sizeofsymboluniform}) in Lemma \ref{sizeofsymbol}, (\ref{improvedphase}) in Lemma \ref{roughestimatephase2}, the following estimate holds,
\[
\|c^{  \tau, \kappa,2}_{\mu, \nu,1}(\xi, \eta,\sigma)\psi_{k_1'}(\eta-\sigma) \psi_{k_2'}(\sigma) \varphi_{\bar{l}_2;l}(\angle(\xi, \nu\eta))\|_{L^\infty_{\xi, \eta, \sigma}} \]
\be\label{noveq929}
+ \|c^{  \tau, \kappa,2}_{\mu, \nu,2}(\xi, \eta,\sigma)\psi_{k_1'}(\xi-\sigma) \psi_{k_2'}(\sigma-\eta) \varphi_{\bar{l}_2;l}(\angle(\xi, \nu\eta))\|_{L^\infty_{\xi, \eta, \sigma}} \lesssim 2^{k +k_{+}+k'_{1,+}}. 
\ee
 For $K^{\mu, \nu, \tau, \kappa,i}_{l; k_1', k_2'}$, $i\in\{1,2\}$, we do spatial localizations for all inputs. As a result, the following decompositions hold, 
\[
K^{\mu, \nu, \tau, \kappa,i}_{l;k_1', k_2',1}= \sum_{j_1\geq -k_{1,-}, j_1'\geq -k_{1,-}', j_2'\geq -k_{2,-}'}K^{\mu, \nu, \tau, \kappa,i}_{l;j_1, j_1', j_2'}, \quad K^{\mu, \nu, \tau, \kappa,i}_{l;k_1', k_2',2}= \sum_{j_2\geq -k_{2,-}, j_1'\geq -k_{1,-}', j_2'\geq -k_{2,-}'}K^{\mu, \nu, \tau, \kappa,i}_{l;j_2, j_1', j_2'},
\]
where
 \[
K^{\mu, \nu, \tau, \kappa,i}_{l;j_1, j_1', j_2'}=  -\int_{t_1}^{t_2}\int_{\R^2} \int_{\R^2}  e^{i t \Phi_1^{\mu, \tau, \kappa}(\xi, \eta, \sigma)} c^{ \tau, \kappa,i}_{\mu, \nu,1}(\xi, \eta, \sigma)\widehat{
f_{k_1,j_1 }^{\mu}
}(t,\xi-\eta)\widehat{f_{k_1' , j_1'}^{\tau}}(t,\eta-\sigma) \]
\be
\times \widehat{f_{k_2',j_2' }^{\kappa}}(t,\sigma) \varphi_{\bar{l}_i; l}(\angle(\xi, \nu \eta))d \eta d\sigma   d t,\quad i=1,2,
\ee
\[
K^{\mu, \nu, \tau, \kappa,i}_{l;j_2, j_1', j_2'}=  -\int_{t_1}^{t_2}\int_{\R^2} \int_{\R^2} e^{i t \Phi_2^{\tau, \kappa, \nu}(\xi, \eta, \sigma)}c^{  \tau, \kappa,i}_{\mu , \nu  ,2}(\xi, \eta, \sigma)\widehat{
f_{k_1',j_1'}^{\tau}
}(t,\xi-\sigma)\widehat{f_{k_2', j_2' }^{\kappa}}(t,\sigma-\eta) \]
\be\label{deceqn5}
\times \widehat{f_{k_2,j_2}^{\nu}}(t,\eta) \varphi_{\bar{l}_i; l}(\angle(\xi, \nu \eta))d \eta d\sigma   d t,\quad i=1,2.
\ee

\begin{lemma}\label{cubictermsestimatelemma}
Under the bootstrap assumption \textup{(\ref{smallness})}  and the assumption that Proposition \textup{\ref{propZnorm2}} holds, if fixed $k,k_1,k_2$ and $j$ satisfy the estimate \textup{(\ref{noveqn1000})}, then the following estimate holds,
\be
\sum_{i_1,i_2=1,2, k_1',k_2'\in \mathbb{Z}}\sum_{\bar{l}_{i_1}\leq l\leq 2} 2^{\delta j}\| \mathcal{F}^{-1}[ K^{\mu, \nu, \tau, \kappa,i_1}_{l; k_1', k_2',i_2}]\|_{B_{k,j}} \lesssim 2^{-2\delta m -2 \delta j}\epsilon_0.
\ee
\end{lemma}
\begin{proof}

We first rule out the case  when  $(k_1',k_2')\in  \cup_{i=1,2}\chi_{k_i}^1\cup \chi_{k_i}^2$ or $(k_1',k_2')\in  \cup_{i=1,2} \chi_{k_i}^3, \tau\kappa=-$. From the estimates (\ref{dece1}) and (\ref{dece2}) in Lemma \ref{phasesize}, we know that $\nabla_\sigma  \Phi_1^{\mu, \tau , \kappa}(\xi, \eta, \sigma)$ always has a good lower bound. Therefore, there is no extra difficulty caused by the fact that $(k_1,k_2)\in \chi_k^3$.  With minor modifications,  we can redo the argument used in the estimate of `` $\textup{good}_k$'' and the estimate of ``$\textup{bad}_k$''     to estimate those scenarios. Hence, We omit  the details here for those cases.

 Now, we restrict ourself to the case when $(k_1',k_2')\in \chi_{k_2}^3$ and $\tau \kappa=+$. We separate into three cases based on the possible size of $k$.

$\bullet$\quad If $k+2l\leq -m+\beta m $.\quad Note that this assumption implies that $k\leq -m/5+\beta m $. From the estimate (\ref{noveq929}), (\ref{eqn293}) in Lemma \ref{bilinearest} and (\ref{L2estimate}) in Lemma \ref{L2estimatelemma} , the following estimate holds, 
\[
 \sum_{i=1,2} 2^{\delta j}\|\mathcal{F}^{-1}[K_{l;k_1',k_2',i}^{\mu , \nu,\tau, \kappa,2}]\|_{B_{k,j}} \lesssim \sup_{t\in[t_1,t_2]}  2^{2\delta m + \alpha k + m + j } 2^{k+k'_{1,+} } 2^{2k+l}  \| \widehat{f}_{k_2}(t, \xi)\|_{L^\infty_\xi}\| e^{-it \Lambda}f_{k_1'}\|_{L^\infty}  \|f_{k_2'}\|_{L^2} \]
\be\label{febneweqn201}
  \lesssim  2^{ 200\delta m} 2^{ m+4k+l}  \epsilon_0\lesssim 2^{m+2k+2l+3\beta m} \epsilon_0 \lesssim  2^{-2\delta  m-2\delta j}\epsilon_0.
\ee
Now we proceed to estimate $ K_{l;k_1',k_2'}^{\mu, \nu,\tau, \kappa,1}$. From the estimate (\ref{angularrelation}) in Lemma \ref{angularrelationlemma}, it is easy to see that  the proof of Lemma \ref{badhhlemma1} is also valid. Hence, we can rule out the case when $\max\{m+l,\min\{-k-l,m\}\} + 100\delta m \leq j \leq m +20\delta m$. Now, it would be sufficient to consider the case when $j\leq -k-l + 100\delta m$.

 From (\ref{noveq719}),  the estimate (\ref{eqn293}) in Lemma \ref{bilinearest} and the estimates  (\ref{L2estimate}) and (\ref{equation6726}) in Lemma \ref{L2estimatelemma}, the following estimate holds, 
\[
\sum_{i=1,2} 2^{\delta j}\| \mathcal{F}^{-1}[ K_{l;k_1',k_2',i}^{\mu , \nu,\tau, \kappa,1}]\|_{B_{k,j}} 
 \lesssim \sup_{t\in[t_1,t_2]}  2^{2\delta m + \alpha k + m + j } 2^{k-2\max\{k, l\}+k'_{1,+}} 2^{2k+l}  \| \widehat{f}_{k_2}(t, \xi)\|_{L^\infty_\xi}   \]
\be\label{neweqn201}
 \times  \| e^{-it \Lambda}f_{k_1'}\|_{L^\infty} \|f_{k_2'}\|_{L^2}  \lesssim  2^{ 3\beta m +k}  \epsilon_0\lesssim 2^{- 2\delta  m- 2\delta  j}\epsilon_0.
\ee
$\bullet$\quad If $-m+\beta m\leq k+2l \leq -\beta m/100.$\quad Note that this assumption implies that $k\leq -\beta m/500$. From the estimate (\ref{noveqn519}) in Lemma \ref{roughestimatephase2}, we know that $|\Phi^{\mu  ,\nu}(\cdot, \cdot)|$ is greater than $2^{-m+\beta m}$. From the estimate  (\ref{equation6911}) in Lemma \ref{angularbilinear}, and  the $L^2_x-L^\infty_x$ type bilinear estimate,  the following estimate holds after putting $T^{\tau, \kappa}(f_{k_1'}^\tau, f_{k_2'}^\kappa)$ in $L^2$,
\[
 \sum_{i=1,2}  2^{\delta j}\| \mathcal{F}^{-1}[ K_{l;k_1',k_2',i}^{\mu , \nu,\tau, \kappa,2}]\|_{B_{k,j}} 
 \lesssim 2^{2\delta m + \alpha k +m +j+k-2m-3\alpha k +k}\epsilon_1^3 +2^{-2\delta m -2\delta j}\epsilon_1^3\lesssim 2^{-2\delta m - 2\delta j}\epsilon_0.
  \]

Same as the previous case,  we can rule out the case   when $\max\{m+l,\min\{-k-l,m\}\} + 100\delta m \leq j \leq m +20\delta m$ for the estimate of $K_{l;k_1',k_2'}^{\mu , \nu,\tau, \kappa,1}$. Hence, it would be sufficient to consider the case when  $j\leq m+l + 100\delta m$, then from the estimate  (\ref{equation6911}) in Lemma \ref{angularbilinear} and the $L^2_x-L^\infty_x$ type bilinear estimate, the following estimate holds after putting $T^{\tau, \kappa}(f_{k_1'}^\tau, f_{k_2'}^\kappa)$ in $L^2$,
\[
  \sum_{i=1,2}  2^{\delta j}\| \mathcal{F}^{-1}[ K_{l;k_1',k_2',i}^{\mu , \nu,\tau, \kappa,1}]\|_{B_{k,j}}  \lesssim 
 2^{\alpha k + m + j+k -2\max\{k,l\}} 2^{-2m-3\alpha k+k} \epsilon_1^3 + 2^{-2\delta m - 2\delta j} \epsilon_1^3.
\]
\be\label{neweqn203}
\lesssim    2^{(1-2\alpha)k +200\delta m  } \epsilon_0  +  2^{-2\delta m - 2\delta j} \epsilon_0   \lesssim 2^{-2\delta m - 2\delta j} \epsilon_0.
\ee
$\bullet$\quad If $  -\beta m/100\leq k +2l.$\quad Recall (\ref{noveqn1000}). Note that this assumption implies that $k\in [-\beta m/1000, 2\beta m]$.
For the case we are considering, all frequencies are almost of size ``$1$'', which means that  the localized angle $\angle(\xi, \nu \eta)$, which is of size greater than $2^{2k_{-}}$, and  the degenerated phase, which is of size greater than $2^{3k_{-}}$,  play  little role. As a result,  there is little difference between estimating $ K_{l;k_1',k_2',1}^{\mu , \nu,\tau, \kappa,1}$  and  $ K_{l;k_1',k_2',i}^{\mu , \nu,\tau, \kappa,j}$, $i,j\in\{1,2\}$. For simplicity,  we only estimate  $ K_{l;k_1',k_2',1}^{\mu , \nu,\tau, \kappa,1}$ in details here.

 From the  $L^2_x-L^\infty_x-L^\infty_x$ type trilinear estimate, and  the following estimate holds   when $\max\{j_1,j_1',j_2'\}\geq 10 \beta m $ after putting the input with the maximum spatial concentration in $L^2$ and the other two inputs in $L^\infty$,
\[
\sum_{ \max\{j_1,j_1',j_2'\}\geq 10 \beta m}  2^{\delta j}\| \mathcal{F}^{-1}[ K_{l;j_1,j_1',j_2'}^{\mu , \nu,\tau, \kappa,1}]\|_{B_{k,j}} \lesssim \sum_{ \max\{j_1,j_1',j_2'\}\geq 10 \beta m} 2^{\beta m +m+j - 10k_{+}-2m-3\alpha k -\max\{j_1,j_1',j_2'\}}\epsilon_1^3\]
\be
 \lesssim 2^{-2\delta m -2\delta j}\epsilon_0.
\ee
It remains to consider the case when $\max\{j_1,j_1',j_2'\}\leq 10 \beta m$. 

In the estimate of ``$\textup{good}_k$'' and ``$\textup{bad}_{k}$'', we used the fact that either $|\xi-\eta|\approx |\eta|$ or $|\eta|\leq 2^{-5}|\xi-\eta|$ to show that the space resonance  in ``$\eta$''  set doesn't intersect with the space resonance in `` $\sigma$ ''set (when $\sigma=\eta/2$), which means that we can alway do integration by parts in ``$\sigma$'' or `` $\eta$'' to take the advantage of the high oscillation either in ``$\sigma$'' or ``$\eta$''. The only extra difficulty caused by the fact that   $(k_1,k_2)\in \chi_k^3$ is that there exists a space resonance  in ``$\eta$'' and ``$\sigma$'' set, i.e., $\nabla_\eta \Phi^{\mu,\tau, \kappa}_i(\xi, \eta, \sigma)$ and  $\nabla_\sigma \Phi^{\mu,\tau, \kappa}_i(\xi, \eta, \sigma)$, $i=1,2$, can equal to zero at the same time.

Therefore, we can decompose the support of frequencies into three regions: (i) the frequencies  are  far  away from the space resonance in ``$\sigma$'' set ; (ii) the frequencies are close to the space resonance in ``$\sigma$'' set but far away from the space resonance in ``$\eta$'' set; (iii) the frequencies  are close to the space resonance  ``$\sigma$'' and in ``$\eta$'' set. More precisely, we decompose the symbols $ c^{  \tau, \tau,1}_{\mu, \nu,1}(\xi, \eta, \sigma)$ and $ c^{  \tau, \tau,1}_{\mu, \nu,2} (\xi, \eta, \sigma)$  into three pieces as follows, 
\[
 c^{  \tau, \tau,1}_{\mu, \nu,1}(\xi, \eta, \sigma)=\sum_{j=1,2,3}  e^{  \tau, \tau,j}_{\mu, \nu }(\xi, \eta, \sigma), \quad   e^{  \tau, \tau,1 }_{\mu, \nu}(\xi, \eta, \sigma)= c^{  \tau, \tau,1}_{\mu, \nu,1}(\xi, \eta, \sigma) \psi_{\geq  \tilde{l}_{\mu,\tau} }(\sigma-\eta/2),\]
\[
e^{  \tau, \tau,2}_{\mu, \nu }(\xi, \eta, \sigma)= c^{  \tau, \tau,1}_{\mu, \nu,1}(\xi, \eta, \sigma) \psi_{\geq  \tilde{l}_{\mu,\tau}  }\big((\xi-\eta)-\mu\tau(\eta-\sigma)\big) \psi_{<  \tilde{l}_{\mu,\tau} }(\sigma-\eta/2),
\]
 \[
e^{  \tau, \tau,3}_{\mu, \nu  }(\xi, \eta, \sigma)= c^{  \tau, \tau,1}_{\mu, \nu,1}(\xi, \eta, \sigma)  \psi_{< \tilde{l}_{\mu,\tau}}\big((\xi-\eta)-\mu\tau(\eta-\sigma)\big) \psi_{< \tilde{l}_{\mu,\tau}  }(\sigma-\eta/2),
\]
 \[ \tilde{l}_{+,-}=\tilde{l}_{+,+}=\tilde{l}_{-,-}=-2\beta m , \quad \tilde{l}_{-,+}= -m/2+10\beta m.
 \]
Because $\max\{j_1,j_1',j_2'\}\leq 10\beta m $ and the threshold $\tilde{l}_{\mu, \nu}$ we choose is away from $-m/2$, by doing integration by parts in $\sigma $ or $\eta$ many times, the terms with symbols $e^{  \tau, \tau,1}_{\mu, \nu }(\xi, \eta, \sigma)$ and $e^{  \tau, \tau,2}_{\mu, \nu }(\xi, \eta, \sigma)$ decay rapidly over time. 

Now, we consider the cubic term with the symbol $e^{  \tau, \tau,3}_{\mu, \nu }(\xi, \eta, \sigma)$. An important observation for the phase $\Phi^{\mu, \tau,\tau}(\xi, \eta,\sigma)$, $(\mu, \nu, \tau)\in \{(+,-,-),(+,+,+),(-,-,-)\}$, is that the space resonance in $\eta$ and $\sigma$ set is far away from the time resonance set.  More precisely, the following estimate holds, 
 \be
|\Phi_{1}^{\mu, \tau, \tau}(\xi, \eta, \sigma)|  \psi_{< \tilde{l}_{\mu,\tau}}\big((\xi-\eta)-\mu\tau(\eta-\sigma)\big) \psi_{< \tilde{l}_{\mu,\tau}  }(\sigma-\eta/2) \gtrsim 2^{-\beta m  },\quad (\mu, \tau, \tau)\neq (-,+,+).
 \ee
 Therefore, we can first do integration by parts in time once for this case. As a result, we can gain $2^{-m}$ by paying the price of $2^{-\beta m}$, the extra gain of $2^{-m+\beta m }$ is sufficient to close the argument. 
 
Lastly, we consider the case when $(\mu,\tau,\tau)=(-,+,+)$. Note that the following equality and estimate hold around  the space resonance in ``$\eta$'' and ``$\sigma$'' set,  
\[
\nabla_\xi \Phi_1^{-,+,+}(\xi, \eta,\sigma)\big|_{(\eta/2, \eta,\eta/2)} =\Lambda'(|\xi|)\frac{\xi}{|\xi|}+ \Lambda'(|\xi-\eta|) \frac{\xi-\eta}{|\xi-\eta|}\big|_{(\eta/2, \eta,\eta/2)} = 0,
\]
\be\label{noveq1009}
|\nabla_\xi \Phi_1^{-,+,+}(\xi, \eta,\sigma)|\psi_{< \tilde{l}_{-,+}}\big((\xi -\sigma)\big) \psi_{< \tilde{l}_{-,+}  }(\sigma-\eta/2) \lesssim 2^{-k_{-}+\beta m + \tilde{l}_{-,+} }\lesssim 2^{-m/2+12\beta m }.
\ee
Recall that $\max\{j_1,j_1',j_2'\}\leq 10\beta m $. From the above estimate  (\ref{noveq1009}), we can rule out the case when $j\geq m/2+14\beta m$ by doing integration by parts in $\xi$ many times,   see the argument used in the proof of Lemma \ref{badhhlemma1}. 
  For the case when $j\leq m/2+14\beta m $, the following estimate holds after using  the volume of support of $\eta$ and $\sigma$,
  \[
\sum_{ \max\{j_1,j_1',j_2'\}\leq  10 \beta m}  2^{\delta j}\| \mathcal{F}^{-1}[ K_{l;j_1,j_1',j_2'}^{- , \nu,\nu , \nu,1}]\|_{B_{k,j}} \lesssim 2^{\beta m +\alpha k +m +j +4\tilde{l}_{-,+}}\epsilon_1^3\lesssim 2^{-2\delta m -2\delta {j}}\epsilon_0.
  \]
Hence finishing the proof. 
\end{proof}

\begin{lemma}
Under the bootstrap assumption \textup{(\ref{smallness})}  and the assumption that Proposition \textup{\ref{propZnorm2}} holds, if fixed $k,k_1,k_2$ and $j$ satisfy the estimate  \textup{(\ref{noveqn1000})}, then the following estimate holds, 
\be
 \sum_{i=1,2}\sum_{\bar{l}_i\leq l\leq 2 } 2^{\delta j}\| \mathcal{F}^{-1}[ \textup{JR}^{\mu, \nu,i }_{l;k_1 , k_2 }]\|_{B_{k,j}} \lesssim 2^{-2\delta m -2 \delta j}\epsilon_0.
\ee
\end{lemma}
\begin{proof}
Recall (\ref{deceqn1}). As there are at most ``$m$'' cases of ``$l$'', we first fix ``$l$''. From the estimate  (\ref{eqn293}) in Lemma \ref{bilinearest} and   (\ref{equation7760}) in Proposition \ref{propZnorm2}, the following estimate holds if $k\leq -\beta m $,
\[
 \sum_{i=1,2}  2^{\delta j}\| \mathcal{F}^{-1}[ \textup{JR}^{\mu, \nu,i }_{l;k_1 , k_2 }]\|_{B_{k,j}}    \lesssim  \sum_{i=1,2} \sup_{t\in [t_1, t_2]}  2^{2\delta m + \alpha k   + m+j  -2\max\{l,k\} } 2^{2k+l}  \| P_{k_i}\mathcal{R}(t)\|_{L^2}\]
\[
 \times \| \widehat{f}_{k_{3-i}}(t,\xi)\|_{L^\infty_\xi}\lesssim 2^{k +200\delta m } \epsilon_0 \lesssim 2^{-\beta m /2}\epsilon_0.
\]
 From the $L^\infty-L^2$ type bilinear estimate (\ref{bilinearesetimate}) in Lemma
 \ref{multilinearestimate} and  (\ref{equation7760}) in Proposition \ref{propZnorm2}, the following estimate holds if $k\geq -\beta m$,
\[
 \sum_{i=1,2}  2^{\delta j}\| \mathcal{F}^{-1}[ \textup{JR}^{\mu, \nu,i }_{l;k_1 , k_2 }]\|_{B_{k,j}}   \lesssim \sum_{i=1,2}\sup_{t\in [t_1, t_2]}   2^{m+j+50\beta m } \| P_{k_i}\mathcal{R}(t)\|_{L^2}\| e^{-it \Lambda} f_{k_{3-i}}\|_{L^\infty}
\lesssim 2^{-\alpha m /4}\epsilon_0.
\]
Hence  finishing the proof.
\end{proof}

\section{Remainder estimate and the proof of Lemma \textup{\ref{L2Znormestimate} }}\label{remainderZnorm}

This section is devoted to prove Proposition \ref{propZnorm2} and Lemma \ref{L2Znormestimate}. The main idea of proving  Proposition \ref{propZnorm2} can be summarized   as follows,
\begin{enumerate}
\item[(i)] We   first decompose the remainder term $\mathcal{R}$ into two parts: cubic type terms, which don't depend on $\Lambda_{\geq 3}[B(h)\psi]$ and terms that do depend on $\Lambda_{\geq 3}[B(h)\psi]$. We will prove a $Z$-norm estimate for a general trilinear form, which is sufficient to estimate the cubic type terms.

\item[(ii)] To estimate  the $Z$-norm of the profile of  $\Lambda_{\geq 3}[B(h)\psi]$,  it would be sufficient to estimate the profile of $\Lambda_{\geq 3}[\nabla_{x,z}\varphi]$ in the $L^\infty_z Z$-normed space, where ``$\varphi$'' is defined in (\ref{potentialinxzcoordinate}). Due to the small data regime, based on the equality (\ref{fixedpoint}), we can use a fixed point type argument to estimate the   $L^\infty_z Z$-norm   of $\Lambda_{\geq 3}[\nabla_{x,z}\varphi]$.

\end{enumerate}

Step (i) is straightforward. Recall (\ref{remainderterms}), we have
\[
\mathcal{R}= \Lambda_{\geq 3}[(1+|\nabla h|^2) B(h)\psi] + i \Lambda \Lambda_{\geq 3}[(1+|\nabla h|^2 ) (B(h)\psi)^2]=  \Lambda_{\geq 3}[(1+|\nabla h|^2)( \Lambda_{\leq 2}[B(h)\psi]
\]
\[
+ \Lambda_{\geq 3}[B(h)\psi])] + i \Lambda \Lambda_{\geq 3}[(1+|\nabla h|^2 ) (\Lambda_{\leq 2}[B(h)\psi]+ \Lambda_{\geq 3}[B(h)\psi])^2]
= I_{\textup{cubic}}+ I_{\textup{fps}},
\]
where
\begin{equation}\label{cubictype}
 I_{\textup{cubic}}= |\nabla h|^2  \Lambda_{\leq 2}[B(h)\psi]+ i \Lambda\Big(|\nabla h|^2  (\Lambda_{\leq 2}[B(h)\psi])^2 + ( \Lambda_{ 2}[B(h)\psi])^2 + 2\Lambda_{2}[B(h)\psi]\Lambda_{1}[B(h)\psi] \Big),
\end{equation}
\[
I_{\textup{fps}}= (1+|\nabla h|^2) \Lambda_{\geq 3}[B(h)\psi] + i \Lambda\Big( (1+|\nabla h|^2)(  \Lambda_{\geq 3}[B(h)\psi])^2 
\]
\begin{equation}\label{highorderdepend}
+2(1+|\nabla h|^2)( \Lambda_{\leq 2}[B(h)\psi]) (\Lambda_{\geq 3}[B(h)\psi])\Big).
\end{equation}
Since   the explicit formula of $\Lambda_{\leq 2}[B(h)\psi]$ is known, we can  explicitly represent  ``$I_{\textup{cubic}} $''    in terms of $h$ and $\psi$.  More precisely, we can rewrite  ``$I_{\textup{cubic}}$'' as follows,
\begin{equation}\label{cubicrepresentation}
I_{\textup{cubic}}= \sum_{\mu, \nu, \tau\in\{+,-\}}C_{\mu, \nu, \tau}(u^{\mu}, u^{\nu}, u^{\tau}) + C'_{\mu, \nu}(u^{\mu}, u^{\nu}, h_1) +  C_{\mu}(u^{\mu}, h_2, h_3)+ C(h_4, h_5, h_6),
\end{equation}
where $h_i$, $1\leq i \leq 6$, denotes some determined quadratic term in terms of $u$ and $\bar{u}$, whose explicit formulas are not pursued here. Generally speaking,  they can be represented as follows,
\[
h_{i} = \sum_{\mu, \nu\in \{+,-\}} T_{\mu, \nu}^{i}(u^{\mu}, u^{\nu}), \quad 1\leq i \leq 6,
\]
where $T_{\mu, \nu}^i(\cdot,\cdot)$, $i\in\{1,\cdots,6\}$, are some determined bilinear operators.
 
\begin{proof}[Proof of Proposition \ref{propZnorm2} ]
Recall (\ref{smallness}) and   (\ref{energyestimate}). From (\ref{Rough est}) in  Lemma \ref{bilinearZnorm}, we have
\[
\sup_{1\leq i\leq 6}\sup_{t\in[2^{m-1}, 2^{m+1}]} \| e^{ it \Lambda} h_i \|_{Z}\lesssim  \epsilon_0. 
\]
From the above estimate and estimates (\ref{cubictypefixedtime}), (\ref{equation6740}), and (\ref{newequation1001}) in Lemma \ref{bilinearZnorm}, the following estimate holds for $k\in \mathbb{Z}$,  $\theta\in[0,1]$, and $t, t_1,t_2\in [2^{m-1}, 2^{m+1}]$,
\[
 \|e^{it \Lambda}\big[I_{\textup{cubic}}\big]\|_{Z}+ 2^{-(1-\theta)k+\theta m}\| P_{k}\big(e^{it \Lambda}\big[I_{\textup{cubic}}\big]\big)\|_{L^2}
 \lesssim  2^{-m }\epsilon_0,
\]
\[
 \sup_{k\in \mathbb{Z}, j\geq\max\{-k,0\}} 2^{\delta  j}  \| \int_{t_1}^{t_2} e^{it \Lambda}\big[I_{\textup{cubic}}\big] d t \|_{B_{k,j}} \lesssim 2^{-\delta m } \epsilon_0.
\]
From $L^2-L^2-L^\infty$ type trilinear estimate (\ref{trilinearesetimate}) in Lemma \ref{multilinearestimate}, we put the input with the medium frequency in $L^\infty$ and the other inputs in $L^2$. As a result, the following estimate holds, 
\be\label{neweqn609}
\sup_{t\in [2^{m-1}, 2^m]} \| \widehat{I_{\textup{cubic}}}(t, \xi)\|_{L^\infty_\xi} \lesssim 2^{-m} \sum_{1\leq i \leq 6}\big(\| e^{it \Lambda} u\|_{Z}+ \|e^{it \Lambda} h_i\|_Z\big)^3 \lesssim 2^{-m} \epsilon_0.
\ee
Combing the above estimates with  estimates (\ref{higherZnorm3}) and (\ref{lowfrequency}), it's easy to see that Proposition \ref{propZnorm2} holds.
\end{proof}

\subsection{$Z$-- norm estimate of terms depend on $\Lambda_{\geq 3}[B(h)\psi]$} In this subsection, we mainly do step (ii), which is stated at the beginning of this section. More precisely, we have the following lemma,
\begin{lemma}\label{higherorderZ}
Under the bootstrap assumption \textup{(\ref{smallness})} and the improved energy estimate \textup{(\ref{energyestimate})}, the following estimates hold for any $k\in \mathbb{Z}$,  $\theta\in[0,1]$, and $t, t_1,t_2\in [2^{m-1}, 2^{m+1}]$,
\begin{equation}\label{higherZnorm3}
 \|e^{it \Lambda}\big[I_{\textup{fps}}\big]\|_{Z}+ 2^{-(1-\theta)k+\theta m}\| P_{k}\big(e^{it \Lambda}\big[I_{\textup{fps}}\big]\big)\|_{L^2}  + \| \widehat{I_{\textup{fps}}}(t, \xi)\|_{L^\infty_\xi} 
 \lesssim  2^{-m }\epsilon_0,
\end{equation}	
\be\label{lowfrequency}
 \sup_{k\in \mathbb{Z}, j\geq\max\{-k,0\}} 2^{\delta  j}  \|  \int_{t_1}^{t_2} e^{it \Lambda}\big[I_{\textup{fps}}\big] d t\|_{B_{k,j}} \lesssim 2^{-\delta m } \epsilon_0.
\ee
\end{lemma}

\begin{proof}
To estimate the $Z$-norm of $\Lambda_{\geq 3}[B(h)\psi]$, it is sufficient to estimate the $L^\infty_z Z$-norm of $\Lambda_{\geq 3}[\nabla_{x,z}\varphi]$. Recall (\ref{fixed1}). We define 
\[
 \tilde{u}_1 = \tilde{h}_1+ i 2\Lambda \psi, \quad \tilde{u}_2 = \tilde{h}_2 + i \Lambda \psi, 
\]
hence
\[
\tilde{h}_1 = \frac{2h+h^2}{(1+h)^2} \Longrightarrow \tilde{u}_1 = 2 u - \frac{3h}{2} \frac{2h+ h^2}{(1+h)^2} - \frac{h^3}{2(1+h)^2}\]
\begin{equation}\label{auxiliary2}
= 2 u - \frac{3(u+\bar{u})}{4} \frac{\tilde{u}_1 + \overline{\tilde{u}_1}}{2} - \frac{(u+\bar{u})}{4} \Big( \frac{\tilde{u}_2+\overline{\tilde{u}_2}}{2} \Big)^2,
\end{equation}
\begin{equation}\label{auxiliary3}
\tilde{h}_2 = \frac{h}{1+h} \Longrightarrow \tilde{u}_2 = u - h \frac{h}{1+h} = u - \frac{(u+\bar{u})(\tilde{u}_2+ \overline{\tilde{u}_2})}{4}.
\end{equation}
With the above notation, we can easily transfer the fixed point type formulation (\ref{fixed1}) into a fixed point type formulation  in terms of $u$, $\bar{u}$, $\tilde{u}_i$ and $\overline{\tilde{u}_i}, $ $i\in\{1,2\}.$ 

Let us first estimate the $Z$-norm of the profile of $\tilde{u}_i$. From (\ref{auxiliary2}) and (\ref{auxiliary3}), the following estimate holds by using the estimate (\ref{Rough est}) in Lemma \ref{bilinearZnorm},
\[
\sum_{i=1,2}\| e^{i t\Lambda} \tilde{u}_i \|_{Z}\lesssim \epsilon_1 + \epsilon_1 \sum_{i=1,2}\| e^{i t\Lambda} \tilde{u}_i \|_{Z},\Longrightarrow \sum_{i=1,2}\| e^{i t\Lambda} \tilde{u}_i \|_{Z}\lesssim \epsilon_1.
\]

Now we are ready to prove Lemma \ref{higherorderZ}.  From the  estimates (\ref{Rough est}), (\ref{equation6739}), (\ref{cubictypefixedtime}), (\ref{equation6740}), and  (\ref{newequation1001}) in Lemma \ref{bilinearZnorm}, and H\"older type estimates, we can derive the following estimate from (\ref{fixed1}),  
\be\label{dece201}
\| e^{it \Lambda}\Lambda_{\geq 3}[\nabla_{x,z}\varphi]\|_{L^\infty_z Z} \lesssim 2^{-m  } \epsilon_0 + \epsilon_1 \| e^{it \Lambda}\Lambda_{\geq 3}[\nabla_{x,z}\varphi]\|_{L^\infty_z Z},
\ee
\be\label{dece202}
\|\widehat{\Lambda_{\geq 3}[\nabla_{x,z}\varphi]}\|_{L^\infty_z L^\infty_\xi} \lesssim 2^{-m} \epsilon_0 + \epsilon_0 \|e^{i t\Lambda}\Lambda_{\geq 3}[\nabla_{x,z}\varphi]\|_{L^\infty_z Z},
\ee
\be\label{dece203}
2^{-(1-\theta) k + \theta m}\|   P_{k}[\Lambda_{\geq 3}[\nabla_{x,z}\varphi]]\|_{L^\infty_z L^2} \lesssim 2^{-m } \epsilon_0 + \epsilon_0 \| e^{it \Lambda} [\Lambda_{\geq 3}[\nabla_{x,z}\varphi]]\|_{L^\infty_z Z},
\ee
\be\label{dece204}
\sup_{k\in \mathbb{Z}, j\geq -k_{-}} 2^{\delta j } \| \int_{t_1}^{t_2} e^{i t\Lambda} \Lambda_{\geq 3}[\nabla_{x,z}\varphi]\|_{L^\infty_z B_{k,j}} \lesssim 2^{-\delta m }\epsilon_0 + 2^{3m/2} \sup_{k\in \mathbb{Z}}\| P_k[\Lambda_{\geq 3}[\nabla_{x,z}\varphi]]\|_{L^\infty_z L^2}\lesssim 2^{-\delta m }\epsilon_0.
\ee
In the above estimate, we used the fact that the case when $j\geq m +100\delta m $ can be ruled out easily as we did in the previous two sections. Also we  used the $L^2-L^\infty$ type estimate for all quartic-and-higher order terms in (\ref{fixed1}).

From the estimates (\ref{dece201}),  (\ref{dece202}),  (\ref{dece203}), and  (\ref{dece204}),  we have the following estimates,
\[
\| e^{i t \Lambda} \Lambda_{\geq 3}[B(h)\psi]\|_{Z} + 2^{-(1-\theta) k + \theta m}\|   P_{k}[B(h)\psi ]\|_{L^2} \lesssim \| e^{it \Lambda}\Lambda_{\geq 3}[\nabla_{x,z}\varphi]\|_{L^\infty_z Z} 
\]
\begin{equation}\label{equation6747}
+ 2^{-(1-\theta) k + \theta m}\|  P_{k}[\Lambda_{\geq 3}[\nabla_{x,z}\varphi]]\|_{L^\infty_z L^2}\lesssim 2^{-m }\epsilon_0,
\end{equation}
\[
\sup_{k\in \mathbb{Z}, j\geq -k_{-}} 2^{\delta j } \| \int_{t_1}^{t_2} e^{i t\Lambda} \Lambda_{\geq 3}[ B(h)\psi]\|_{L^\infty_z B_{k,j}} \lesssim 2^{-\delta m }\epsilon_0,
\]
\[
\| \Lambda_{\geq 3}[B(h, \psi)]\|_{L^\infty_\xi}\lesssim  \|\widehat{\Lambda_{\geq 3}[\nabla_{x,z}\varphi]}\|_{L^\infty_z L^\infty_\xi} \lesssim    2^{-m} \epsilon_0.
\]
Following the same procedure, recall (\ref{highorderdepend}), it's easy to see that our desired estimates  (\ref{higherZnorm3}) and (\ref{lowfrequency}) hold. 
\end{proof}

\begin{lemma}\label{bilinearZnorm}
For any $\mu, \nu, \kappa \in\{+,-\}$ and   $f$, $g$,$h\in H^{N_0}\cap Z$,  which satisfy the following estimates,
\[
\| f \|_{H^{N_0}} + \| g\|_{H^{N_0}} +\| h\|_{H^{N_0}}\leq A, \quad \| f \|_{Z} + \| g\|_{Z} +\| h\|_{Z} \leq B,
\]
  the following estimates  for  any $t,t_1,t_2\in [2^{m-1}, 2^{m+1}]$, $m\in\mathbb{Z}_{+}$, and $\theta\in[0,1]$,
\begin{equation}\label{Rough est}
\| e^{it \Lambda }Q((e^{-it \Lambda}f)^{\mu},(e^{-it \Lambda}g)^{\nu}) \|_{Z} \lesssim   \|f\|_Z \|g\|_Z + 2^{-10\delta m}(A+B)^2,
\end{equation}
\begin{equation}\label{equation6739}
\sup_{k\in \mathbb{Z}} 2^{-(1-\theta)k} \|  P_{k}\big[e^{it \Lambda }Q((e^{-it \Lambda}f)^{\mu},(e^{-it \Lambda}g)^{\nu})\big] \|_{L^2}\lesssim 2^{-\theta m } B^2. 
\end{equation}
\begin{equation}\label{cubictypefixedtime}
\| e^{it \Lambda} C\big((e^{-it \Lambda}f)^{\mu}, (e^{-it \Lambda} g)^{\nu}, (e^{-it\Lambda} h)^{\kappa}\big)\|_{Z}\lesssim 2^{-m} \|f\|_Z \|g\|_Z\|h\|_Z + 2^{-m-10\delta m}(A+B)^2,
\end{equation}
\begin{equation}\label{equation6740}
\sup_{k\in \mathbb{Z}} 2^{-(1-\theta)k} \| P_k \big[e^{it \Lambda} C\big((e^{-it \Lambda}f)^{\mu}, (e^{-it \Lambda} g)^{\nu}, (e^{-it\Lambda} h)^{\kappa}\big)\big]\|_{L^2}\lesssim 2^{-(1+\theta) m} B^3,
\end{equation}
\be\label{newequation1001}
 \sup_{k\in \mathbb{Z}, j\geq -k_{-}} 2^{\delta  j}  \| \int_{t_1}^{t_2} e^{it \Lambda} C\big((e^{-it \Lambda}f)^{\mu}, (e^{-it \Lambda} g)^{\nu}, (e^{-it\Lambda} h)^{\kappa}\big) d t \|_{B_{k,j}} \lesssim 2^{-10\delta m } (A+B)^3,
\ee
where the symbol $q(\xi-\eta, \eta)$ of bilinear operator $Q(\cdot, \cdot)$ and the symbol $c(\xi-\eta, \eta-\sigma, \sigma)$ of trilinear operator $C(\cdot, \cdot, \cdot)$ satisfy the following estimates respectively, 
\begin{equation}\label{symbolestimatequadratic}
\|q(\xi-\eta, \eta)\|_{\mathcal{S}^\infty_{k,k_1, k_2}}\lesssim 2^{3\max\{k_1, k_2\}_{+}}, 
\end{equation}
\begin{equation}\label{symbolestimatecubic}
\| c(\xi-\eta, \eta-\sigma, \sigma)\|_{\mathcal{S}^{\infty}_{k,k_1, k_2, k_3}}\lesssim 2^{4\max\{k_1, k_2, k_3\}_{+}}.
\end{equation}

\end{lemma}
\begin{proof}
$\bullet$\quad We first prove the desired  estimates (\ref{equation6739}) and (\ref{equation6740}). Since  (\ref{equation6740}) can be proved very similarly, we only prove (\ref{equation6739}) in details here.  From the bilinear estimate (\ref{bilinearesetimate}) in Lemma \ref{multilinearestimate}, the following estimate holds for any $k\in \mathbb{Z}$ and any $\theta\in[0,1]$,
\[
\sup_{k\in \mathbb{Z}} 2^{-(1-\theta)k} \|  P_{k}\big[e^{it \Lambda }Q((e^{-it \Lambda}f)^{\mu},(e^{-it \Lambda}g)^{\nu})\big] \|_{L^2} \lesssim \sum_{k_2\leq k_1-10} 2^{-(1-\theta)k_1+3k_{1,+}}\big[ \| e^{-it \Lambda} f_{k_1}\|_{L^\infty} \| g_{k_2}\|_{L^2} \]
\[+ \| e^{-it \Lambda} g_{k_1}\|_{L^\infty} \| f_{k_2}\|_{L^2}\big] + \sum_{|k_2-k_1|\leq 10} 2^{-(1-\theta)k+(1-\theta)k+ 3k_{1,+}} \| e^{-it \Lambda} f_{k_1}\|_{L^\infty}^{\theta} \| f_{k_1}\|_{L^2}^{(1-\theta)} \| g_{k_2}\|_{L^2}  \lesssim 2^{-\theta m} B^2.
\]

$\bullet$\quad Now we proceed to prove the desired estimates (\ref{Rough est}) and (\ref{cubictypefixedtime}). Since the proof of the desired estimates (\ref{cubictypefixedtime}) and  (\ref{Rough est}) are very similar, we only prove  (\ref{Rough est}) in details here.  

Firstly, we   do dyadic decomposition   for two inputs. From the $L^2-L^\infty$ type bilinear estimate and the $L^\infty\longrightarrow L^2$ type Sobolev embedding, the following estimate holds,
\[
\| e^{it \Lambda }Q((e^{-it \Lambda}f_{k_1})^{\mu},(e^{-it \Lambda}g_{k_2})^{\nu}) \|_{B_{k,j}}\lesssim 2^{\alpha k + 6k_+ + j +\min\{k_1,k_2\} +3\max\{k_1,k_2\}_{+}} \| f_{k_1}(t)\|_{L^2} \| f_{k_2}(t)\|_{L^2}
\]
\[
\lesssim 2^{j+ \max\{k_1,k_2\}+(2-\alpha)\min\{k_1,k_2\} +\delta m -(N_0-20) \max\{k_1,k_2\}_{+}} (A+B)^2.
\]
Due to the symmetry between inputs, without loss of generality, we assume that $k_2\leq k_1+5$.  From the above estimate, we can rule out the very-low-frequency case and the relatively-high-frequency case. From now on, we restrict ourself to the following case, 
\[
k_2\leq k_1+5, \quad k_1+(2-\alpha)k_2\geq -j -10\delta m, \quad k_1\leq 2\beta j +\delta m. 
\]
 With minor modifications in the proof of Lemma \ref{hhgoodj1}, we can rule out the case when $j\geq   m +10 $. It remains to consider the case when $j\leq   m +10 $. Note that we have $k_2\geq -m/(2-\alpha)-\beta m$ and $k_1\leq 3\beta m$ for this case. From the $L^2-L^\infty$ type bilinear estimate (\ref{bilinearesetimate}) in Lemma \ref{multilinearestimate}, the following estimate holds if $k_2\leq -\alpha m/2,$ 
\[
\| e^{it \Lambda }Q((e^{-it \Lambda}f_{k_1})^{\mu},(e^{-it \Lambda}g_{k_2})^{\nu}) \|_{B_{k,j}}\lesssim 2^{\alpha k + j +6k_{+} +3k_{1,+}} \| e^{-it \Lambda} f_{k_1}\|_{L^\infty} \| g_{k_2}\|_{L^2}
\]
\be
\lesssim 2^{30\beta m +(1-\alpha)k_2} \|f\|_Z\|g\|_Z \lesssim 2^{-10\delta m} \|f\|_Z\|g\|_Z.
\ee
To sum up, it remains to consider the case when $k_1,k_2\in [-\alpha m/2, 3\beta m]$. If moreover $(k_1,k_2)\in \chi_k^1\cup \chi_k^2$,  then from the estimates (\ref{dece1}) and (\ref{dece2}) in Lemma \ref{roughestimatephase2}, it is easy to see that we are away from the space resonance in ``$\eta$'' set. Hence, we can do integration by parts in ``$\eta$'' many times 
to rule out the case when $\max\{j_1,j_2\}\leq m-2\alpha m $. As a result, we have
\[
\sum_{j\leq m +10}\| e^{it \Lambda }Q((e^{-it \Lambda}f)^{\mu},(e^{-it \Lambda}g)^{\nu}) \|_{B_{k,j}} \lesssim \sum_{j\leq m +10} 2^{\alpha k + j + 6 k_+} \Big( \sum_{(k_1,k_2)\in \chi_k^3 } 2^{   3k_{1,+}} \| e^{-it \Lambda} f_{k_1}\|_{L^\infty}
 \]
 \[
 \times \| g_{k_2}\|_{L^2} +    2^{-10\delta m } (A+B)^2  + \sum_{
\begin{subarray}{c}
   -\alpha m/2	 \leq k_2\leq k_1+5\leq 3\beta m\\
   \max\{j_1,j_2\}\geq m-2\alpha m\\
   \end{subarray}
   }   2^{  3k_{1,+}} 2^{-m-\max\{j_1,j_2\}-\alpha k_1 -\alpha k_2} \| f_{k_1,j_1}\|_Z \| g_{k_2,j_2}\|_Z\Big)
 \]
 \[
\lesssim \|f\|_Z \|g\|_Z + 2^{-10\delta m}(A+B)^2.
 \]
	
 $\bullet$\quad Now, we proceed to prove (\ref{newequation1001}). The major difference between the estimate (\ref{equation6740}) and (\ref{newequation1001}) is that   we can take   advantage of the oscillation in time for (\ref{newequation1001}). Firstly, we do dyadic decompositions for all the inputs. Due to the symmetry between inputs, without loss of generality, we assume that $k_3\leq k_2 \leq k_1$. From the   $L^2-L^\infty-L^\infty$ type estimate, the following estimate holds for any $t_1, t_2\in[2^{m-1}, 2^m]$,
 \[
2^{\delta  j}  \| \int_{t_1}^{t_2} e^{it \Lambda} C\big((e^{-it \Lambda}f_{k_1})^{\mu}, (e^{-it \Lambda} g_{k_2})^{\nu}, (e^{-it\Lambda} h_{k_3})^{\kappa}\big) d t \|_{B_{k,j}}\lesssim \sup_{t\in[2^{m-1}, 2^m]} 2^{\alpha k +6k_{+}+m +(1+\delta)j + 4k_{1,+} }
 \]
 \[
\times\| e^{-it\Lambda} f_{k_3}\|_{L^\infty} \| e^{-it\Lambda} f_{k_2}\|_{L^\infty} \|   f_{k_1}\|_{L^2} \lesssim 2^{(2-2\alpha)k_3+(1+\delta)j  -(N_0-12)k_{1,+}+\delta m} (A+B)^3.
 \]
Therefore, from the above estimate, we can rule out the case when $k_1\geq  \beta j +\delta m$ or $k_3\leq -(1+\delta)j/(2-2\alpha) -\beta m$. Hence, it would be sufficient to consider the following case, 
\[
k_3\leq k_2\leq k_1 \leq  \beta j +\delta m, \quad k_3\geq -(1+\delta)j/(2-2\alpha) -\beta m.
\]
As before,  with minor modifications in the proof of Lemma \ref{hhgoodj1}, we can first rule out the case when $j\geq m+10$. For the case when $j\leq m+10$, from the $L^2-L^\infty-L^\infty$ type estimate, we have
\[
2^{\delta  j}  \| \int_{t_1}^{t_2} e^{it \Lambda} C\big((e^{-it \Lambda}f_{k_1})^{\mu}, (e^{-it \Lambda} g_{k_2})^{\nu}, (e^{-it\Lambda} h_{k_3})^{\kappa}\big) d t \|_{B_{k,j}}\lesssim  \sup_{t\in[2^{m-1}, 2^m]} 2^{\alpha k +10k_{1,+}+2m+\delta m }\| f_{k_3}(t)\|_{L^2}
\]
\[
\times \| e^{-it \Lambda} f_{k_1}\|_{L^\infty} \| e^{-it \Lambda} f_{k_2}\|_{L^\infty} \lesssim 2^{(1-2\alpha)k_3+9\beta m}B^3.
\]
From the above estimate, we can rule out further the case when $k_3\leq -12\beta m$. Therefore, it remains to consider the case when $k_1,k_2,k_3\in[-12\beta m, \beta m ]$. In other words,  all frequencies are of size almost like ``$1$''. The $Z$-norm estimate of a  trilinear form of this type has already been considered in the third case of the proof of Lemma \ref{cubictermsestimatelemma}.  We omit details here.	
\end{proof}

\subsection{Proof of Lemma \textup{\ref{L2Znormestimate}}}\label{znormnormalform}.
 Note that $(k_1,k_2)\in \chi_k^3$  and $t\in [2^{m-1},2^m]$, $m\in \mathbb{Z}_{+}$. Since we have already proved the Proposition \ref{propZnorm2} in the previous subsection, under the bootstrap assumption (\ref{smallness}), the estimates in Lemma \ref{L2estimatelemma} are valid in this subsection.

 From the estimate (\ref{sizeofnormalform}), the estimates (\ref{L2estimate}) and (\ref{L2derivativeestimate}) in Lemma \ref{L2estimatelemma}, the following estimate holds from the $L^2-L^2$ type estimate and the volume of support of $\xi$, the following estimate holds, 
 \[
 \sum_{\mu, \nu\in\{+,-\}}  2^{\delta j}\|  e^{i t\Lambda} A_{\mu,\nu} (u^{\mu}_{k_1}(t),u^{\nu}_{k_2}(t)) \|_{B_{k,j}} \lesssim 2^{\alpha k + (1+\delta )j +8k_{+} - 2k+k}\|f_{k_1}(t)\|_{L^2}\|f_{k_2}(t)\|_{L^2}\]
 \[
 \lesssim 2^{(1+\alpha)k+10\delta m +j - (N_0-10)k_{+}}\epsilon_1^2. 
 \]
 Therefore, we can rule out the case when $k\leq -(1+20\delta)j/(1+\alpha)-20\delta m $ and $k\geq j/(N_0-20)+10\delta m $. Moreover, it is easy to check that the proof of Lemma \ref{hhgoodj1} is still valid. Hence we can rule  out the case when $j\geq (1+20)\delta m $. It would be sufficient to consider the case when $k$ and $j$ satisfy the following estimate, 
 \be\label{deceqn41}
-(1+100\delta)m/(1+\alpha)\leq k\leq 2\beta m , \quad j\leq (1+20)\delta m. 
 \ee

  Recall the decomposition of symbol $m_{\mu, \nu}(\xi-\eta, \eta)$ in (\ref{deceqn11}). We localize the angle between $\xi$ and $\nu \eta$ and have the following decomposition, 
 \be\label{dec18eqn1}
\mathcal{F}[e^{i t\Lambda} A_{\mu, \nu}(u^{\mu}_{k_1}, u^{\nu}_{k_2})](\xi)=\sum_{i=1,2}\sum_{\bar{l}_i \leq l\leq 2 } \widetilde{T}_{l;k_1,k_2}^{\mu, \nu,i}(t,\xi),\quad \widetilde{T}_{l;k_1,k_2}^{\mu, \nu,i}(t,\xi)= \sum_{j_1\geq -k_{1,-}, j_2\geq -k_{2,-}} \widetilde{T}_{l;k_1,j_1,k_2,j_2}^{\mu, \nu,i}(t,\xi),
 \ee
 where $\bar{l}_1:=2k_{-}, \bar{l}_2:=-4$,
 \[
\widetilde{T}_{l;k_1,j_1,k_2,j_2 }^{\mu, \nu,i}(t,\xi)= \int_{\R^2} e^{it \Phi^{\mu, \nu}(\xi, \eta)}\widehat{f^{\mu}_{k_1,j_1}}(t, \xi-\eta)\widehat{f^{\nu}_{k_2,j_2}}(t, \eta)m_{\mu, \nu}^i(\xi-\eta, \eta) \varphi_{\bar{l}_i;l}(\angle(\xi, \nu \eta))  d \eta. 
 \]
Therefore, 
\[
\mathcal{F}^{-1}[ \widetilde{T}_{l;k_1,j_1,k_2,j_2}^{\mu, \nu,i}(t,\xi)\psi_k(\xi)]= \int_{\R^2} \int_{\R^2} e^{i x\cdot \xi + it \Phi^{\mu, \nu}(\xi, \eta)}\widehat{f^{\mu}_{k_1,j_1}}(t, \xi-\eta)\widehat{f^{\nu}_{k_2,j_2}}(t, \eta)\]
\[
\times  m_{\mu, \nu}^i(\xi-\eta, \eta)  \varphi_{\bar{l}_i;l}(\angle(\xi, \nu \eta))\psi_k(\xi)  d \eta d\xi,
\]
where symbols $  m_{\mu, \nu}^i(\xi-\eta, \eta)$, $i\in\{1,2\}$, are defined in (\ref{deceqn11}). Recall (\ref{dec18eqn1}). It is easy to see that our goal can be reduced to prove the following two Lemmas. 
\begin{lemma}
Under the bootstrap assumption \textup{(\ref{smallness})}, the following estimate holds if fixed $k$ satisfies the estimate \textup{(\ref{deceqn41})},
\be
\sum_{0\leq j\leq (1+20\delta)m} \sum_{\bar{l}_2 \leq l\leq 2 } \| \mathcal{F}^{-1}\big[ \widetilde{T}_{l;k_1,k_2}^{\mu, \nu,2}(t,\xi)\big] \|_{B_{k,j}} \lesssim \epsilon_0.
\ee
\end{lemma}
\begin{proof}
Recall that $\bar{l}_2=-4$. From the $L^2_x-L^\infty_x$ type bilinear estimate, the following estimate holds,
\[
\sum_{0\leq j\leq m+20} \sum_{\bar{l}_2 \leq l\leq 2 } \| \mathcal{F}^{-1}\big[ \widetilde{T}_{l;k_1,k_2}^{\mu, \nu,2}(t,\xi)\big] \|_{B_{k,j}} \lesssim \sum_{0\leq j\leq m+20} 2^{\alpha k + j +8k_{+}} \|e^{-it\Lambda} f_{k_1}(t) \|_{L^\infty} \| f_{k_2}(t)\|_{L^2}\]
\[
\lesssim \sum_{0\leq j\leq m+20} 2^{(1-\alpha)k-2k_{+}+j-m}\epsilon_1^2\lesssim \epsilon_0.
\]
Note that   the following estimate holds if $|x|\sim 2^j$ and $j\geq m+20$,
\[
|\nabla_\xi[x\cdot \xi + t\Phi^{\mu, \nu}(\xi, \eta)|\sim 2^{j}.
\]
Therefore, after doing integration by parts $\xi$ many times, we can rule out the case when $\min\{j_1,j_2\}\leq j-\delta j$.  If $\min\{j_1,j_2\}\geq j-\delta j$, then the following estimate holds after using the volume of support of $\xi$ first and then using the $L^2-L^2$ type estimate, 
\[
\sum_{\min\{j_1,j_2\}\geq j-\delta j}\|\mathcal{F}^{-1}[ \widetilde{T}_{l;k_1,j_1,k_2,j_2}^{\mu, \nu,i}(t,\xi)\psi_k(\xi)]\|_{B_{k,j}}\lesssim 	\sum_{ j_1,j_2  \geq j-\delta j} 2^{\alpha k + j+k+8k_{+}-j_1-j_2-2\alpha k}\epsilon_1^2 \lesssim 2^{-m/2}\epsilon_0.
\]
Hence finishing the proof. 
\end{proof}
\begin{lemma}
Under the bootstrap assumption \textup{(\ref{smallness})}, the following estimate holds if fixed $k$ satisfies the estimate \textup{(\ref{deceqn41})},
\be\label{deceqn100}
\sum_{0\leq j\leq (1+20\delta)m} \sum_{\bar{l}_1 \leq l\leq 2 } \| \mathcal{F}^{-1}\big[ \widetilde{T}_{l;k_1,k_2}^{\mu, \nu,1}(t,\xi)\big] \|_{B_{k,j}} \lesssim \epsilon_0.
\ee
\end{lemma}
\begin{proof}
Recall that $\bar{l}_1:=2k_{-}$. Note that the estimate (\ref{angularrelation}) holds for the case we are considering. 
Following  the same argument used in the proof of Lemma \ref{badhhlemma1}, we first rule out the case when $j\geq \max\{m+l, -k-l\}+100\delta m$. It would be sufficient to consider the case when $ j\leq \max\{m+l, -k-l\}+100\delta m$. Based on the possible size of $j, k$, and $l$, we separate into three cases as follows.

$\bullet$ If  $k+2l\leq -m +50\delta m  $, then this assumption implies that $j\leq -k-  {l}+200\delta m$ and $k\leq -m/5+\beta m$. From the estimate (\ref{eqn293}) in Lemma \ref{bilinearest} and the estimates (\ref{L2estimate}) and (\ref{equation6726}) in Lemma \ref{L2estimatelemma}, the following estimate holds,
\[
 \| \mathcal{F}^{-1}[ \widetilde{T}_{l;k_1 ,k_2 }^{\mu, \nu,i}(t,\xi)\psi_k(\xi)]\|_{B_{k,j}} \lesssim   2^{\alpha k+ j -2\max\{k_{-}, l\}+2k+l } \|\widehat{ f_{k_1}}(t, \xi)\|_{L^\infty_\xi}\|f_{k_2}\|_{L^2}
\]
\be
 \lesssim 2^{\alpha k+200\delta m }\epsilon_1^2\lesssim 2^{-\beta   m }\epsilon_0.
\ee

$\bullet$ If   $k+2l\geq -m +50\delta m$ and $l=\bar{l}_1$, then this assumption implies that $k\geq -m/5 +10\delta m $. Note that the following estimate holds from the estimate (\ref{equation6911}) in Lemma \ref{angularbilinear}, 
 \be
 \sum_{j\leq m + l+80}\| \mathcal{F}^{-1}[ \widetilde{T}_{l;k_1 ,k_2 }^{\mu, \nu,i}(t,\xi)\psi_k(\xi)]\|_{B_{k,j}} \lesssim  \sum_{j\leq m + l+80} 2^{  j -2\max\{k_{-},l\}  -m +(1-\alpha)k-2k_{+	}}\epsilon_1^2 + 2^{-m}\epsilon_1^2\lesssim  \epsilon_0.
 \ee
 If $j\geq m +l +80$, then from the estimate (\ref{angularrelation}) in Lemma \ref{angularrelationlemma}, the following estimate holds if $|x|\sim 2^{j}$, and $\angle(\xi, \nu\eta)\sim 2^{l}$, 
 \[
|\nabla_\xi[x\cdot \xi + t\Phi^{\mu, \nu}(\xi, \eta)|\sim 2^{j}.
 \]
 Hence, we can do integration by parts in $\xi$ many times to rule out the case when $\min\{j_1,j_2\}\leq j-\delta j$. From the estimate (\ref{eqn293}) in Lemma \ref{bilinearest}, the following estimate holds if $\min\{j_1,j_2\}\geq j-\delta j$,
 \[
\sum_{\min\{j_1,j_2\}\geq j-\delta j} \| \mathcal{F}^{-1}[ \widetilde{T}_{l;k_1,j_1,k_2,j_2}^{\mu, \nu,i}(t,\xi)\psi_k(\xi)]\|_{B_{k,j}} \lesssim \sum_{\min\{j_1,j_2\}\geq j-\delta j} 2^{\alpha k + j -2\max\{k_{-},l\} +k+\bar{l}_1/2} \| f_{k_1, j_1}(t)\|_{L^2}
 \]
 \[
\times  \| f_{k_2,j_2}(t)\|_{L^2} \lesssim 2^{-(1-2\delta)j -\alpha k }\epsilon_1^2 \lesssim 2^{-(1-2\delta)m -(2+\alpha)k_{-}} \epsilon_0\lesssim 2^{-m/3}\epsilon_0. 
 \]
 Note that the above estimate is more than sufficient to cover the logarithm loss of size ``$m$ '' caused by   the summation with respect to $j$.
 
 $\bullet$ If  $k+2l\geq -m  +50\delta m $ and $l>\bar{l}_1$, then this assumption implies that $j\leq m +l+100\delta m $ and $\max\{k,l\}\geq -m/3+10\delta m$. Note that we are away from the space resonance   in ``$\eta$'' set for the case we are considering. Therefore, we can rule out the case when $\max\{j_1,j_2\}\leq m + l-4\beta m$ by doing integration by parts $\eta$ many times. From   the estimate (\ref{equation6911}) in Lemma \ref{angularbilinear} and the estimates (\ref{L2estimate}) and (\ref{equation6726}) in Lemma \ref{L2estimatelemma}, the following estimate holds,  
 \[
\sum_{\max\{j_1,j_2\}\geq m +l-4\beta m}  \| \mathcal{F}^{-1}[ \widetilde{T}_{l;k_1,j_1,k_2,j_2}^{\mu, \nu,i}(t,\xi)\psi_k(\xi)]\|_{B_{k,j}} \lesssim \sum_{\max\{j_1,j_2\}\geq m +l-4\beta m} 2^{\alpha k  +j-2\max\{k_{-},\bar{l}_1\}}
 \]
 \[
\times 2^{-m-2\alpha k-\max\{j_1,j_2\} -4k_{+}}\|f_{k_1,j_1}(t)\|_Z \|f_{k_2,j_2}(t)\|_Z\lesssim 2^{-m+5\beta m -2\max\{k_{-},\bar{l}_1\} -\alpha k }\epsilon_0\lesssim 2^{-\beta m}\epsilon_0.
 \]
 Note that the above estimate is more than sufficient to cover the logarithm loss of size ``$m^2$ '' caused by   the summation with respect to $j$ and $l$. Hence finishing the proof of the desired estimate (\ref{deceqn100}).

\end{proof}

\appendix

\section{Analysis of the Phases}\label{analysisphase}

In  this appendix, we analyze  and estimate the phase $\Phi^{\mu, \nu}(\xi, \eta)$, where $\mu, \nu\in\{+,-\}$. Recall that the phase $\Phi^{\mu, \nu}(\xi, \eta)$ is defined as follows, 
\[
\Phi^{\mu, \nu}(\xi, \eta):= \Lambda(|\xi|)-\mu \Lambda(|\xi-\eta|) -\nu \Lambda(|\eta|), \quad \xi, \eta\in \R^2,  \Lambda(|\xi|):=\sqrt{|\xi|\tanh|\xi|}.
\]
Note that
\[
\nabla_\xi\Phi^{\mu,\nu}(\xi, \eta) =\Lambda'(|\xi|) \frac{\xi}{|\xi|} - \mu \Lambda'(|\xi-\eta|)\frac{\xi-\eta}{|\xi-\eta|},
\quad 
\nabla_\eta \Phi^{\mu,\nu}(\xi, \eta) =    \mu \Lambda'(|\xi-\eta|)\frac{\xi-\eta}{|\xi-\eta|}- \nu\Lambda'(|  \eta|)\frac{ \eta}{| \eta|}. 
\]
It turns out that the relative size between $\nabla_{\xi}\Phi^{\mu, \nu}(\xi, \eta)$ and $\nabla_{\eta}\Phi^{\mu, \nu}(\xi, \eta)$ plays an essential role. Hence it is necessary to consider the relation  between $\angle(\xi, \mu(\xi-\eta))$ and  $\angle(\mu(\xi-\eta), \nu \eta)$.

We will show that either the phase $\Phi^{\mu, \nu}(\xi, \eta)$ is big or  the sizes of     angles $\angle(\xi, \mu(\xi-\eta))$ and  $\angle(\mu(\xi-\eta), \nu \eta)$  are  proportional to each other  if the   phase $\Phi^{\mu, \nu}(\xi, \eta)$ is small, see estimate (\ref{angularrelation}) in Lemma \ref{angularrelationlemma} and estimate \ref{improvedphase} in Lemma \ref{roughestimatephase2}. To this end,  we define   axillary functions as follows,
\be\label{auxillarycutoff1}
c_{-, -}(\xi, \eta)=1,\quad   c_{-,+}(\xi, \eta)= \widetilde{\varphi}((|\eta|-|\xi |)/|\xi-\eta|), 
\ee
\be\label{auxillarycutoff2}
  c_{+,+}(\xi, \eta)= \widetilde{\varphi}((|\xi|-|\eta|)/|\xi-\eta|),\quad  c_{+,-}(\xi, \eta)= \widetilde{\varphi}(\big(|\xi-\eta|-|\xi|)/|\eta|),
\ee
where  $\widetilde{\varphi}(\cdot) $ is a  cutoff function such that $\widetilde{\varphi}(x)=1$ if $  x\leq  2^{-100}$ and it is supported inside $(-\infty,  2^{-50}]$. 

 \begin{lemma}\label{angularrelationlemma}
For any $\mu, \nu\in\{+,-\}$,  $|\xi|\sim 2^{k}, $ $|\xi-\eta|\sim 2^{k_1}, |\eta|\sim 2^{k_2}, $ $k,k_1,k_2\in \mathbb{Z},$ if $(\xi, \eta)\in \textit{supp}(1-c_{\mu,\nu}(\xi, \eta))$, then it is easy to verify that 
\be\label{angularrelation}
\angle(\xi, \mu(\xi-\eta)) \sim 2^{k_2-k_1} \angle(\xi, \nu \eta), \quad  \angle(\mu(\xi-\eta), \nu \eta) \sim 2^{k-k_1} \angle(\xi, \nu \eta).
\ee
Moreover, if $(k_1,k_2, \mu, \nu)\in \mathcal{P}_{\textup{bad}}^k$, then the above estimate also holds. 
\end{lemma}
\begin{proof}
To prove the desired estimate (\ref{angularrelation}), it would be sufficient to consider the case when $\angle(\xi, \nu \eta)\leq 2^{-100} $. If $(\xi, \eta)\in \textit{supp}(1-c_{\mu,\nu}(\xi, \eta)) $, then we only have to consider the case when $(\mu, \nu)\neq (-,-)$. Recall (\ref{auxillarycutoff1}) and (\ref{auxillarycutoff2}). For the case when $(\mu, \nu)=(-,+)$, $(\xi, \eta)\in \textit{supp}(1-c_{\mu,\nu}(\xi, \eta)) $ implies that $|\eta|\geq |\xi| + 2^{-100}|\xi-\eta|$, which further implies that the angle $(\xi, \eta-\xi)$ is small when $\angle(\xi, \eta)$ is small. The other two cases follows very similarly, we omit details here. 

Recall (\ref{badtypephase}). We first consider the case when $(k_1,k_2, \mu, \nu)\in \chi_k^1\times\{(+,-), (-,+)\}$, i.e., $|\xi|\leq 2^{-5}|\eta|$ and $\mu\nu=-$ for this case. It is easy to see that $\angle(\xi, \nu(\eta-\xi))$ is of same size as $\angle(\xi, \nu \eta)$ as the angle $\angle(\nu\eta, \nu(\eta-\xi))$ is much smaller than $\angle(\xi, \nu \eta)$.  The desired estimate (\ref{angularrelation}) also holds  very similarly   for the case when $(k_1,k_2, \mu, \nu)\in \chi_k^2\times\{(+,-), (+,+)\}$.
\end{proof}

\begin{lemma}\label{roughestimatephase2}
Given any $k, k_1,k_2,l \in \mathbb{Z}$,  $\mu, \nu\in\{+,-\}$, s.t., $k_2\leq k_1+5$, $l\leq 2$,    $|\xi|\sim 2^k, |\xi-\eta|\sim 2^{k_1}, |\eta|\sim 2^{k_2}$,   and $\angle(\xi, \nu\eta)\sim 2^l$. 
Then the following rough  estimate holds,
\be\label{noveqn519}
|\Phi^{\mu,\nu}(\xi, \eta)| \gtrsim  2^{\min\{k, k_2\}/2+\min\{k, k_2\}_{-}/2+2 \max\{k, k_2\}_{-}} + 2^{-k_{1,+}/2+\min\{k,k_2\}+2l}. 
\ee
If $(\xi, \eta)\in \textit{supp}(c_{\mu,\nu}(\xi, \eta))$, then the following improved estimate holds, 
\be\label{improvedphase}
|\Phi^{\mu,\nu}(\xi, \eta)| \gtrsim 2^{k_{2}-k_{1,+}/2}.
\ee
If $(k_1,k_2,\mu, \nu)\in \mathcal{P}_{\textup{bad}}^k$, then the following estimate holds, 
\be\label{dece31}
   |\Phi^{\mu,\nu}(\xi, \eta)| \sim 2^{ -\min\{k,k_2\}_{+}/2+ \min\{k,k_2\}+2 k_{1,-}} +  2^{\min\{k,k_2\}-k_{1,+}/2+2l}.
\ee
\end{lemma} 
\begin{proof}
  Note that the following estimate holds easily, 
\be\label{noveqn512}
|\Phi^{-, -}(\xi, \eta)| \geq \Lambda(|\xi-\eta|)\gtrsim 2^{k_{1}/2+k_{1,-}/2}.
\ee
Note that
\be\label{noveqn967}
 \Phi^{-, +}(\xi, \eta)   = \Lambda(|\xi|) +\Lambda(|\xi-\eta|)-\Lambda(|\xi|+|\xi-\eta|) +\Lambda(|\xi|+|\xi-\eta|) -\Lambda(|\eta|),
\ee
\be\label{noveqn622}
\Phi^{+, +}(\xi, \eta)   = \Lambda(|\xi|) -\Lambda(|\xi-\eta|+|\eta|) + \Lambda(|\xi-\eta|+|\eta|) - \Lambda(|\xi-\eta|)-\Lambda(|\eta|),
\ee
\be\label{noveqn621}
\Phi^{+,-}(\xi, \eta) =   \Lambda(|\xi|) +\Lambda(|\eta|)- \Lambda(|\xi|+|\eta|)+\Lambda(|\xi|+|\eta|) - \Lambda(|\xi-\eta|).
\ee
Recall that $|\xi|\sim 2^{k}, |\xi-\eta|\sim 2^{k_1}, |\eta|\sim 2^{k_2}, \angle(\xi, \nu \eta)\sim 2^{l}, $ and $k_2\leq k_1+5$. 
From the estimates (\ref{noveqn461}) and (\ref{noveqn462}) in Lemma \ref{auxillarylemma4}, it's easy to see that the following estimate holds,
\be\label{noveqn513}
|\Phi^{-, +}(\xi, \eta)|  \gtrsim 2^{k/2+k_{-}/2+2k_{1,-}} + 2^{-k_{1,+}/2+k+2l},
\ee
\be\label{noveqn514}
|\Phi^{+, +}(\xi, \eta)|  \gtrsim 2^{k_2/2+k_{2,-}/2+2 k_{1,-}} +2^{-k_{1,+}/2+k_2+2l},
\ee
\be\label{noveqn515}
|\Phi^{+, -}(\xi, \eta)|  \gtrsim 2^{\min\{k, k_2\}/2+\min\{k, k_2\}_{-}/2+2 k_{1,-}}+ 2^{-k_{1,+}/2+\min\{k,k_2\}+2l}. 
\ee
To sum up, from (\ref{noveqn512}), (\ref{noveqn513}), (\ref{noveqn514}), and (\ref{noveqn515}), it is easy to see that our desired estimate (\ref{noveqn519}) holds. 
 
 Now, let's proceed to prove the desired estimate (\ref{improvedphase}). If $(\mu,\nu)=(-,+)$, then the assumption  $(\xi, \eta)\in \textit{supp}(1-c_{-,+}(\xi,\eta)) $ implies that $|\eta|\leq |\xi|+2^{-100}|\xi-\eta|$. Recall (\ref{noveqn967}). From the estimates (\ref{noveqn461}) and (\ref{noveqn462}) in Lemma \ref{auxillarylemma4}, we have
 \[
|\Phi^{-,+}(\xi, \eta)| \gtrsim 2^{k_{1}-k_{1,+}/2}.
 \]
 The proof for the case  $(\mu,\nu)\in\{ (+,+), (+,-)\}$  is very similar, we omit details here.
 
 Lastly, we prove the desired estimate (\ref{dece31}). Recall (\ref{badtypephase}). We first consider the case when $(k_1,k_2, \mu, \nu)\in \chi_k^1\times \{(-,+), (+,-)\}$. Since $|\xi|\ll |\xi-\eta|\approx |\eta|$, from the equalities (\ref{noveqn967}) and (\ref{noveqn621}), it is easy to see that our desired estimate (\ref{dece31}) follows directly from the estimates (\ref{noveqn461}) and (\ref{noveqn462})  in Lemma \ref{auxillarylemma4}. The case when $(k_1,k_2, \mu, \nu)\in \chi_k^2\times \{(+,+), (+,-)\}$ follows very similarly. Hence finishing the proof.

\end{proof}

\begin{lemma}\label{phasesize}
Given $k_1, k_2, k, l\in \mathbb{Z}$, $l\leq 2$, $k_2\leq k_1+4$, $\mu, \nu \in \{+,-\}$, such that  $|\xi|\sim 2^{k}, $ $ |\xi-\eta|\sim 2^{k_1}, |\eta|\sim 2^{k_2}$, $\angle(\xi, \nu \eta)\sim 2^{l}$. 
If $(k_1,k_2)\in \chi_k^1$, then we have
\begin{equation}\label{eqn134}
\big| \Lambda'(|\xi-\eta|) \frac{\xi-\eta}{|\xi-\eta|} +\Lambda'(|\eta|) \frac{\eta}{|\eta|}\big| \sim 2^{k-k_{1}+\max\{l, 2k_{1,-}\}-k_{1,+}/2},
\end{equation}
If    
 $(k_1,k_2)\in \chi_k^2$,  then  we have, 
\begin{equation}\label{eqn140}
 2^{-k_{1,+}/2+\max\{l,2k_{1,-}\}} \lesssim \big| \Lambda'(|\xi-\eta|) \frac{\xi-\eta}{|\xi-\eta|} - \nu\Lambda'(|\eta|) \frac{\eta}{|\eta|}\big|  \lesssim 2^{ \max\{l,2k_{1,-}\}}  .
\end{equation}
Moreover,  the following estimates always hold, 
\be\label{dece1}
\big| \Lambda'(|\xi-\eta|) \frac{\xi-\eta}{|\xi-\eta|} +\Lambda'(|\eta|) \frac{\eta}{|\eta|}\big|\gtrsim 2^{-3k_{1,+}/2+k+k_{1,-}}, 
\ee
\be\label{dece2}
 \big| \Lambda'(|\xi-\eta|) \frac{\xi-\eta}{|\xi-\eta|} -\Lambda'(|\eta|) \frac{\eta}{|\eta|}\big|\psi_{\geq k-10}(|\eta -\xi/2|) \gtrsim 2^{-3k_{1,+}/2+k+k_{1,-}}.
\ee
\end{lemma}

\begin{proof}

  $\bullet$\quad If $(k_1,k_2)\in \chi_k^1$, i.e., $|k_1-k_2|\leq 5$, $k\leq k_1-5$, which means that   $0< |\xi|\ll |\xi-\eta|\approx |\eta|$. Note that
\be\label{deceqn501}
\big| \Lambda'(|\xi-\eta|) \frac{\xi-\eta}{|\xi-\eta|} +\Lambda'(|\eta|) \frac{\eta}{|\eta|}\big|\sim 2^{-k_{1,+}/2} \angle( \xi-\eta, -\eta) + 2^{k_{1,-}-3k_{1,+}/2}\big| |\xi-\eta|-|\eta| \big|.
\ee
\be\label{eqnnew3}
\angle(\xi, \nu \eta)= 2^{l}, k\leq k_1-5, \Longrightarrow \angle(\xi-\eta, -\eta) \sim 2^{k-k_1+l}  ,
\ee
 From (\ref{deceqn501}) and (\ref{eqnnew3}), it is easy to see  that the desired estimate  (\ref{eqn134}) holds if $l \geq 2k_{1,-}-10$. 

 It remains to consider the case when $l\leq 2k_{1,-}-10$. Because the angle between $\xi$ and $\nu \eta$ is very small, as a result we have $\big||\xi-\eta|-|\eta|\big|\sim 2^{k}$. Again from  (\ref{deceqn501}) and (\ref{eqnnew3}), it is easy to see  that the desired estimate  (\ref{eqn134}) also  holds .

  $\bullet$\quad If $(k_1,k_2)\in \chi_k^2$, i.e., $|k_1-k |\leq 4$, $k_2\leq k -5$, which means that   $0< |\eta|\ll |\xi-\eta|\approx |\xi|$. 
   Note that 
   \[
\big| \Lambda'(|\xi-\eta|) \frac{\xi-\eta}{|\xi-\eta|} - \nu\Lambda'(|\eta|) \frac{\eta}{|\eta|}\big|  \sim \max\{|\Lambda'(|\xi-\eta|)-\Lambda'(|\eta|)|, 2^{-k_{1,+}/2}\angle(\xi-\eta,\nu \eta)\},
\]
\[
\angle(\xi, \nu \eta) \sim \angle(\xi-\eta, \nu \eta),\]
\[
  2^{-k_{1,+}/2+2k_{1,-}} \leq  \int_{|\xi-\eta|/2}^{|\xi-\eta|} |\Lambda''(r) |d r 
\leq |\Lambda'(|\xi-\eta|)-\Lambda'(|\eta|)| \leq  \int_{0}^{|\xi-\eta|} |\Lambda''(r) |d r \leq 2^{2k_{1,-}}.
  \]
Hence, from the above estimates, it's easy to see that our desired estimate (\ref{eqn140}) holds.

Lastly, let's proceed to prove the desired estimates (\ref{dece1}) and (\ref{dece2}). Note that
\[
\Lambda'(|\xi-\eta|) \frac{\xi-\eta}{|\xi-\eta|} + \Lambda'(|\eta|) \frac{\eta}{|\eta|}=0,\quad \Longrightarrow \xi =0,
\]
\[
\Lambda'(|\xi-\eta|) \frac{\xi-\eta}{|\xi-\eta|} - \Lambda'(|\eta|) \frac{\eta}{|\eta|}=0,\quad \Longrightarrow \eta = \xi/2,
\]
Therefore, those two quantities have corresponding lower bounds when the frequencies are localized away from these two points. Following a very similar analysis as the proof of the estimates (\ref{eqn134}) and (\ref{eqn140}), it is easy to see that the desired estimates (\ref{dece1}) and (\ref{dece2}) hold.

\end{proof}

\begin{lemma}\label{auxillarylemma4}
Let $f(r):=\sqrt{r\tanh r}$, then the following estimate holds for any $r, s\in \mathbb{R},  r\geq s \geq 0,$
\be\label{noveqn461}
f(r+s) - f(r)  \sim s  \max\{r, 1\}^{-1/2},
\ee
\be\label{noveqn462}
f(r)+ f(s)-f(r+s) \sim s^{\h}  \min\{s,1\}^{\h
} \min\{r,1\}^{2}.
\ee
\end{lemma}
\begin{proof}
From direct computations, we have
\be\label{noveqn432}
f'(r) = \displaystyle{\frac{4r e^{2r} + e^{4 r}-1}{2(1+e^{2r})^{3/2} (e^{2r}-1)^{1/2}r^{1/2}}}\geq 0,\quad r\geq 0,
\ee
\be\label{noveqn431}
f''(r) = \displaystyle{\frac{2e^{4r}(1+8r^2) - 1-e^{8r} - 8 e^{6r}r(2r-1) - 8 e^{2r}r(1+2r)}{4(1+e^{2r})^{5/2} (e^{2r}-1)^{3/2} r^{3/2}}}, \quad r\geq 0. 
\ee
 An important observation is that $f''(r)\leq 0$ and  $f''(r)=0$ if and only if $r=0$. To prove this claim, we only have to prove that the numerator is non-positive. We define $\hat{f}(r)$ to be the numerator of $  f''(r)$ in (\ref{noveqn431}). Note that the following decomposition holds, 
\[
\hat{f}(r):= g(r)+ h(r),\quad g(r):= 16 e^{4r} r^2 - 8 e^{6r} r^2 - 8 e^{2r} r^2=-8e^{2r}r^2(e^{2r}-1)^2,
\]
\[
h(r):=2e^{4r} - 1-e^{8r} - 8 e^{6r}r(r-1) - 8 e^{2r}r(1+r).
\]
Obviously, $g(r)\leq 0 $ and $g(r)=0$ if and only $r=0$. 
It remains to check $h(r)$.  After taking up to four derivatives for $h(r)$, we have $h^{(n)}(0)=0$ for $ n \in\{0,1,2,3\}$, and the following estimate holds, 
\[
h^{(4)}(r)= -128 e^{2 r} (5  + 5 r + r^2 + 
   27 e^{4 r }( r + 3 r^2) +32 e^{6 r} -27 e^{4r}- 4 e^{2 r}) < 0,\]
hence $h(r)\leq 0$ and $h(r)=0$ if and only $r=0$.

Note that $f(0)=0$. Hence, 
\be\label{noveqn469}
f(r+s)-f(r)= \int_{r}^{r+s} f'(\tau) d \tau \geq 0,
\ee
\be\label{equation1222}
f(r)+f(s)-f(r+s)= \int_{0}^{r} [f'(\tau_1)- f'(s+\tau_1)] d \tau_1=\int_{0}^{r} \int_{0}^{s} - f''(\tau_1+\tau_2) d \tau_2 d \tau_1\geq 0.
\ee
Note that
\be\label{noveqn471}
f'(r)\sim \left\{\begin{array}{ll}
1  & \textup{if}\, r\leq 1\\
r^{-1/2} &   \textup{if}\, r\geq 1,\\
\end{array}\right. \qquad f''(r)\sim \left\{\begin{array}{ll}
-r  & \textup{if}\, r\leq 1\\
-r^{-3/2} &   \textup{if}\, r\geq 1.\\
\end{array}\right.
\ee
Hence, from (\ref{noveqn469}) and  (\ref{noveqn471}), we have
\be\label{noveqn491}
f(r+s)-f(s) \sim s \max\{r, 1\}^{-1/2}. 
\ee
From (\ref{equation1222})  and  (\ref{noveqn471}),   the following estimate holds if $r,s\leq 1$,
\be\label{noveqn492}
f(r)+f(s)-f(r+s)\sim r^2 s, 
\ee
If $s \leq 1\leq r$,   we have
\be\label{noveqn493}
s\lesssim \int_{1/2}^1\int_{0}^{s} -f''(\tau_1+\tau_2) d \tau_1 d \tau_2  \leq  f(r)+f(s)-f(r+s) \leq f(s) \leq s.	
\ee
Lastly, if $1\leq s\leq r$,   we have 
\be\label{noveqn494}
 s^{1/2}\lesssim \int_{s/2}^{s}\int_{s/2}^{s} -f''(\tau_1+\tau_2) d\tau_2 d \tau_1\leq    f(r)+f(s)-f(r+s) \leq f(s) \leq s^{1/2}.	
 \ee
To sum up, from estimates (\ref{noveqn492}), (\ref{noveqn493}), (\ref{noveqn494}), we have
\begin{equation}\label{sizeofII}
 f(r)+f(s)-f(r+s) \sim s^{\h}  \min\{s,1\}^{\h
} \min\{r,1\}^{2}, \quad 0\leq s \leq r.
\end{equation}
Hence finishing the proof.
\end{proof}

\end{document}